\documentclass{amsart}

\RequirePackage{color}

\let\oldtocsection=\tocsection
 
\let\oldtocsubsection=\tocsubsection
 
\let\oldtocsubsubsection=\tocsubsubsection
 
\renewcommand{\tocsection}[2]{\hspace{0em}\oldtocsection{#1}{#2}}
\renewcommand{\tocsubsection}[2]{\hspace{2em}\oldtocsubsection{#1}{#2}}
\renewcommand{\tocsubsubsection}[2]{\hspace{2em}\oldtocsubsubsection{#1}{#2}}

\def\smallskip{\vskip\smallskipamount}
\def\medskip{\vskip\medskipamount}
\def\bigskip{\vskip\bigskipamount}

\usepackage{amsmath,amsthm,bm,mathrsfs,amscd,tikz}
\usepackage{ytableau}
\usepackage{comment}
\usepackage{mathdots}
\usepackage[all]{xy}
\usepackage{graphicx}
\usepackage[colorinlistoftodos]{todonotes}
\usepackage{enumerate}
\usepackage{feynmp-auto}
\usepackage[english]{babel}
\usepackage[utf8]{inputenc}
\usepackage{float}
\usepackage{tikz}
\usepackage{tikz-cd}
\usepackage{amssymb}
\usepackage{mathtools}
\usepackage{hyperref}
\usetikzlibrary{decorations.pathmorphing}

\usepackage[utf8]{inputenc}

\newtheoremstyle{thmstyle}{}{}{\itshape}{}{\bfseries}{ }{5pt}{}
\newtheoremstyle{exstyle}{}{}{}{}{\bfseries}{ }{5pt}{}
\newtheoremstyle{defstyle}{}{}{}{}{\bfseries}{ }{5pt}{}
\newtheoremstyle{remstyle}{}{}{}{}{\bfseries}{ }{5pt}{}

\theoremstyle{thmstyle}
\newtheorem{thm}{Theorem}[section]
\newtheorem{theorem}[thm]{Theorem}

\newtheorem{lemma}[thm]{Lemma}

\newtheorem{proposition}[thm]{Proposition}

\newtheorem{corollary}[thm]{Corollary}

\theoremstyle{exstyle}

\newtheorem{example}[thm]{Example}

\theoremstyle{defstyle}

\newtheorem{definition}[thm]{Definition}

\newtheorem{def-prop}[thm]{Definition-Proposition}
\newtheorem{def-lem}[thm]{Definition-Lemma}
\newtheorem{rem-convention}[thm]{Remark-Convention}
\newtheorem{def-note}[thm]{Definition-Notation}

\theoremstyle{remstyle}

\newtheorem{remark}[thm]{Remark}

\theoremstyle{remstyle}

\newcommand{\Hom}{\operatorname{Hom}}
\newcommand{\Fac}{\operatorname{Gen}}
\newcommand{\Ext}{\operatorname{Ext}}

\DeclareMathOperator*{\rad}{rad}

\DeclareMathOperator*{\modu}{mod}

\DeclareMathOperator*{\add}{add}

\DeclareMathOperator*{\ann}{ann}
\newcommand{\Str}{\operatorname{Str}}

\DeclareMathOperator*{\Brick}{brick}
\DeclareMathOperator*{\ind}{ind}
\DeclareMathOperator*{\tors}{tors}
\DeclareMathOperator*{\ftors}{f-tors}

\DeclareMathOperator*{\End}{End}

\DeclareMathOperator*{\Cogen}{Cogen}

\DeclareMathOperator*{\rigid}{rigid}
\DeclareMathOperator*{\irigid}{\mathtt{i}rigid}
\DeclareMathOperator*{\ttilt}{-tilt}
\DeclareMathOperator*{\tilt}{tilt}
\DeclareMathOperator*{\trig}{-rigid}
\DeclareMathOperator*{\nn}{nn}
\DeclareMathOperator*{\soc}{soc}
\DeclareMathOperator*{\topm}{top}

\DeclareMathOperator*{\node}{node}
\DeclareMathOperator*{\secl}{secl}

\DeclareMathOperator*{\pd}{pd}
\DeclareMathOperator*{\id}{id}
\DeclareMathOperator*{\gl.dim}{gl.dim}
\DeclareMathOperator*{\mND}{mND}

\DeclareMathOperator*{\mSB}{mSB}
\DeclareMathOperator*{\Mri}{Mri}
\DeclareMathOperator*{\Mtaui}{M\text{$\tau$}i}

\DeclareMathOperator*{\repf}{Rf}

\DeclareMathOperator*{\sB}{sB}
\DeclareMathOperator*{\B}{B}
\DeclareMathOperator*{\G}{G}
\DeclareMathOperator*{\St}{S}
\DeclareMathOperator*{\D}{D}

\makeatletter
\newcommand{\doublewidetilde}[1]{{%
  \mathpalette\double@widetilde{#1}%
}}
\newcommand{\double@widetilde}[2]{%
  \sbox\z@{$\m@th#1\widetilde{#2}$}%
  \ht\z@=.9\ht\z@
  \widetilde{\box\z@}%
}
\makeatother

\begin{document}

\title{$\tau$-tilting finiteness of Biserial Algebras}
\author[Kaveh Mousavand]{Kaveh Mousavand} 
\address{LaCIM, UQAM, Montréal, Québec, Canada}
\email{mousavand.kaveh@gmail.com }
\thanks{The author is partially supported by ISM Scholarship.}


\maketitle

\section*{Abstract}
\vskip 0.3cm

In this paper we treat the $\tau$-tilting finiteness of biserial (respectively special biserial) algebras over algebraically closed (respectively arbitrary) fields. Inside these families, to compare the notions of representation-finiteness and $\tau$-tilting finiteness, we reduce the problem to the $\tau$-tilting finiteness of minimal representation-infinite (special) biserial algebras. Building upon the classification of minimal representation-infinite algebras, we fully determine which minimal representation-infinite (special) biserial algebras are $\tau$-tilting finite and which ones are not. To do so, we use the brick-$\tau$-rigid correspondence of Demonet, Iyama and Jasso, and the classification of minimal representation-infinite special biserial algebras due to Ringel.

Furthermore, we introduce the notion of minimal $\tau$-tilting infinite algebras, analogous to the notion of minimal representation infinite algebras, and prove that a minimal representation-infinite (special) biserial algebra is minimal $\tau$-tilting infinite if and only if it is a gentle algebra. As a consequence, we conclude that a gentle algebra is $\tau$-tilting infinite if and only if it is representation infinite.
We also show that for every minimal representation-infinite (special) biserial algebra, the notions of tilting finiteness and $\tau$-tilting finiteness are equivalent. This implies that a mild (special) biserial algebra is tilting finite if and only if it is brick finite.

\tableofcontents

\section{Introduction}\label{Introduction}

Throughout, unless specified otherwise, $Q$ denotes a finite and connected quiver, $k$ is a field and all algebras are considered to be basic, connected, associative and finite dimensional over $k$. For an algebra $\Lambda$, the category of finitely generated left $\Lambda$-modules is denoted by $\modu \Lambda$ and by a subcategory $\mathcal{C}$ of $\modu \Lambda$ we always mean $\mathcal{C}$ is full and closed under isomorphism classes and direct summands. If $I$ is an admissible ideal in the path algebra $kQ$ and $\Lambda=kQ/I$, then $(Q,I)$ denotes the bound quiver of $\Lambda$. All $\Lambda$-modules are assumed to be finitely generated and basic and we identify them with representations  of the bound quiver $(Q,I)$. 

For the sake of readability, we divide this section into three parts: 
The first subsection explains how this work should be seen as part of a more general project that we have conducted with the same theme. In this section, we thoroughly describe our methodology and elaborate on the significance of the problems treated here.
The second subsection states our main results and suggests how we would like to approach them in more generality. The closing subsection fixes the notations and terminology used throughout the entire text. The terminology which is not defined in this section will be introduced in the following sections, wherever they are used.

\subsection{General problem and methodology}
Recall that $\Lambda$ is \emph{representation-finite} (or rep-finite, for short) if $\modu \Lambda$ contains only finitely many isomorphism classes of indecomposable objects. A $\Lambda$-module $M$ is called a \emph{brick} if 
every nonzero morphism $f$ in $\End_{\Lambda}(M)$ is invertible.
Analogously, we say $\Lambda$ is \emph{brick-finite} if there are only finitely many isomorphism classes of bricks in $\modu \Lambda$. Evidently, every rep-finite algebra is brick-finite, while there are examples showing that the converse does not hold (explicit examples of wild and rep-infinite tame algebras which are brick-finite respectively appear in \cite{Mi} and Sections \ref{section: Nody algebras} and \ref{Section:tau-tilting finite gentle algebras are representation-finite} of this paper). 
Bricks have played significant roles in different aspects of representation theory, among which one can refer to their applications in representation theory of species \cite{R1}, in torsion classes and wide subcategories \cite{MS}, in the study of geometry of the moduli spaces of representations \cite{Sc}, as well as in the stability conditions \cite{T}, to mention a few.

The new investigations in the study of (functorially finite) torsion classes of algebras via the notion of $\tau$-tilting theory, recently introduced by Adachi, Iyama and Reiten \cite{AIR}, has once again draw a lot of attention to bricks and their prominent roles.
In particular, the new result of Demonet, Iyama and Jasso \cite{DIJ} shows that a good knowledge of bricks in the module category of an algebra becomes decisive in many important problems which relate to the $\tau$-tilting theory of algebras. This is the approach we adopt in most parts of the paper, where we are concerned with the family of special biserial algebras. In the last section, we also consider the family of biserial algebras over algebraically closed fields to extend our results in that setting.

In \cite{DIJ}, the authors show that the notion of $\tau$-tilting finiteness, which is a modern generalization of representation finiteness, is in fact equivalent to brick-finiteness. This equivalence provides new insight into the study of $\tau$-tilting objects, because, as we observed above, bricks and brick-finiteness could be defined for any given algebra and without the advanced tools from the Auslander-Reiten theory. 
Thus, the following question, which naturally arises even independent of the concept of $\tau$-tilting theory, could be phrased in terms of this modern language and a complete answer to it sheds light on this new and active area of research:
\vskip 0.2cm
\textbf{Question (1)}: Let $\mathfrak{F}$ be a family of $k$-algebras. For which algebras in $\mathfrak{F}$ are the notions of $\tau$-tilting finiteness and representation-finiteness equivalent? 
\vskip 0.2cm

Let us outline our general methodology and the more specific settings we consider in this paper to approach an explicit answer to the above question. Because every rep-finite algebra is evidently $\tau$-tilting finite, to study nontrivial $\tau$-tilting finite algebras, we necessarily focus on the rep-infinite algebras.
For a family of algebras $\mathfrak{F}$, $\tau$-tilting finiteness implies rep-finiteness if and only if each rep-infinite algebra $\Lambda$ in $\mathfrak{F}$ is $\tau$-tilting infinite. 
Moreover, it is easy to verify that if $\phi: \Lambda_1 \rightarrow \Lambda_2$ is a surjective algebra morphism and $\Lambda_1$ is rep-finite (similarly brick-finite), then so is $\Lambda_2$. Therefore, according to the above-mentioned result of \cite{DIJ}, rep-finiteness and $\tau$-tilting finiteness of algebras are preserved under quotients.
Thus, if $\mathfrak{F}$ is a family of $k$-algebras, to answer the above question for the entire family, we need to check whether every rep-infinite algebra in $\mathfrak{F}$ which is minimal with respect to this property is in fact $\tau$-tilting infinite. Let us remark that various versions of the notion of minimality of algebras in different families are studied in the literature (see for example, \cite{Bo3}, \cite{HV}, \cite{R2}, \cite{Sk}, to just mention a few). Following the terminology widely accepted, we say an algebra $\Lambda$ is \emph{minimal representation-infinite} (or min-rep-infinite, for short) if $\Lambda$ is rep-infinite, but every proper quotient of it is rep-finite.

In dealing with the above question for various families of algebras, we often implement reductions of this form, where, instead of the entire family, we consider a particular subfamily of the algebras. For a given family $\mathfrak{F}$, in order to efficiently employ this reduction, one needs to find a concrete description of the algebras in the subfamily $\Mri(\mathfrak{F})$, which consists of the isomorphism classes of those algebras in the family $\mathfrak{F}$ that are rep-infinite and minimal in $\mathfrak{F}$ with respect to this property (i.e, each $\Lambda$ in $\Mri(\mathfrak{F})$ is rep-infinite and for every proper quotient $\Lambda'$ of $\Lambda$, either $\Lambda'$ is rep-finite, or $\Lambda' \notin \mathfrak{F}$).
As explained in the following, this article is the first part of a more extensive project, where we investigate $\tau$-tilting finiteness of all minimal representation-infinite algebras and introduce a new concept analogous to min-rep-infinite algebras, where minimality is considered with respect to $\tau$-tilting infiniteness.

Here, we are primarily concerned with the comparison of the notions of rep-infiniteness and $\tau$-tilting infiniteness of the family of special biserial algebras, which we denote by $\mathfrak{F}_{\sB}$.
Recall that $\Lambda$ is \emph{special biserial} if it is Morita equivalent to an algebra $kQ/I$ such that $(Q,I)$ satisfies the following conditions:
\begin{enumerate}[(B1)]
\item At every vertex $x$ in $Q$,  there are at most two incoming and at most two outgoing arrows.
\item For each arrow $\alpha$, there is at most one arrow $\beta$ such that $\beta \alpha \notin I$ and at most one arrow $\gamma$ such that $\alpha \gamma \notin I$.
\end{enumerate}

Through a reduction outlined above, we reduce our problem to the subfamily $\Mri(\mathfrak{F}_{\sB})$ and fully determine which algebras in it are $\tau$-tilting finite and which ones are not. 
Along the way, we also derive new results on the $\tau$-tilting finiteness of string and gentle algebras, respectively denoted by $\mathfrak{F}_{\St}$ and $\mathfrak{F}_{\G}$ (for definition, see Sections \ref{Priliminary} and \ref{Section:tau-tilting finite gentle algebras are representation-finite}).
To put the above-mentioned families of algebras in comparison, recall that the chain of containments $\mathfrak{F}_{\G} \subsetneq \mathfrak{F}_{\St} \subsetneq \mathfrak{F}_{\sB}$ is known, where every inclusion is strict. 
We further show some interesting features of the Auslander-Reiten quivers of those rep-infinite algebras which are minimal in each of the above families (see Sections \ref{Section:Reduction to mild special biserial algebras} and \ref{subsection: More on fully reduced gentle algebras}).

Before we restrict our attention to these families, let us remark that for an arbitrary family of algebras $\mathfrak{F}$, there may exist $\Lambda$ in $\Mri(\mathfrak{F})$ with a proper quotient algebra $\Lambda'$ which is rep-infinite (for concrete examples of this phenomenon, see Section \ref{Section:tau-tilting finite gentle algebras are representation-finite}, where we describe the algebras in $\Mri(\mathfrak{F}_{\G})$ in terms of their bound quivers). Namely, these two notions of minimality do not always imply each other.
In contrast, if $\mathfrak{F}$ is a family closed under taking the algebra quotients, every $\Lambda$ in $\Mri(\mathfrak{F})$ is in fact a minimal representation-infinite.
Motivated by this simple observation, we say a family $\mathfrak{F}$ of $k$-algebras is \emph{quotient-closed} if for any surjective algebra map $\phi: \Lambda \rightarrow \Lambda'$, if $\Lambda$ belongs to $\mathfrak{F}$, then so does $\Lambda'$.
In particular, for such a family of algebras, we can repose \text{Question (1)} as follows:
\vskip 0.2cm

\textbf{Question (2)}: Suppose $\mathfrak{F}$ is a quotient-closed family of $k$-algebras. Which minimal representation-infinite algebras in $\mathfrak{F}$ are $\tau$-tilting infinite?
\vskip 0.2cm

Despite what it may look at the first sight, the condition that $\mathfrak{F}$ is a quotient-closed family is not very restrictive and many interesting families of algebras could be investigated from this point of view. In this paper we observe that the family of (special) biserial algebras are quotient-closed and then completely solve the above question.
Furthermore, we employ the same strategy to answer the same question for some families of algebras which are not necessarily quotient-closed, but we view them as a subfamily of a quotient-closed family of algebras. 

To clarify this point, let us consider the family of string algebras and that of gentle algebras, which we have already denoted by $\mathfrak{F}_{\St}$ and $\mathfrak{F}_{\G}$. These families have received a lot of attention in the past few decades, mostly due to their rich combinatorics that make them fruitful in multiple disciplines (for example, see \cite{BR}, \cite{CB1}, \cite{GMM}, \cite{GP}, \cite{HKK}, and \cite{GR}). However, one easily observes that neither of these families is quotient-closed. Thus, asking the above question about them may seem doomed at the first glance.

In contrast to $\mathfrak{F}_{\G}$ and $\mathfrak{F}_{\St}$, the family of special biserial algebras $\mathfrak{F}_{\sB}$, which strictly contains both of them, is quotient-closed.
Hence, it fits into the setting of the second version of our question, and we determine exactly for which algebras in $\Mri(\mathfrak{F}_{\sB})$ $\tau$-tilting infiniteness and  representation infiniteness are equivalent.
Because a min-rep-infinite algebra cannot have a projective-injective module (see Proposition \ref{Monomial ideal}), we get $\Mri(\mathfrak{F}_{\sB})= \Mri(\mathfrak{F}_{\St})$, meaning that for most of the work we only need to deal with certain classes of string algebras which feature a tractable framework.

Before we state our results from the viewpoint of $\tau$-tilting finiteness, let us elaborate on the significance of the classes of string algebras in $\Mri(\mathfrak{F}_{\St})$ that we study in this paper and look at the scope of our work through the lens of the minimal representation-infinite algebras, a topic that has remained an active domain of research in the past few decades and has been vastly studied by different groups of mathematicians.

Although the methodical study of the subject was initiated in 60's, an extensive search for a complete classification of the minimal representation-infinite algebras over algebraically closed fields received a lot of attentions in 70's and 80's. This was mainly due to their crucial roles in the proofs of some fundamental theorems in representation theory, among which perhaps what is nowadays called the Second Brauer-Thrall conjecture is celebrated the most. The first complete proof of the aforementioned conjecture (that every representation infinite algebra is strongly unbounded) appeared in the work of Bautista \cite{Ba}, which relied on his former work and the influential contributions of others, including Bongartz, Gabriel, Jans, Roiter, Ringel, Salmer\'on, Skowro\'nski (for more details, see \cite{Bo1}, \cite{Bo2}, \cite{BG+}, \cite{J}, \cite{R2} and \cite{Sk}).

Recently, and surprisingly around at the same time that the notion of $\tau$-tilting theory was first introduced in \cite{AIR}, a full classification of min-rep-infinite algebras received new attentions, mostly from those who had already significantly contributed to the progress of the subject. Building upon the new results of Bongartz \cite{Bo3}, in \cite{R2}, Ringel shows that over an algebraically closed field, any min-rep-infinite algebra falls into at least one of the following families:
\begin{itemize}
\item[(mND)] Algebras with non-distributive ideal lattice. 
\item[(mUC)] Algebras with a good universal cover $\widetilde{\Lambda}$ such that a convex subcategory of $\modu \widetilde{\Lambda}$ is tame-concealed of type $\widetilde{\mathbb{D}}_n$ or $\widetilde{\mathbb{E}}_{6,7,8}$. 
\item[(mSB)] Special biserial algebras.
\end{itemize}

As mentioned earlier, this paper is the first part of a more extensive research project that we have conducted  to analyze the above families from the perspective of $\tau$-tilting theory. 
Our approach naturally led us to a new direction that we introduce here and will further pursue in another work and in more generality. In particular, we also consider a notion of minimality with respect to $\tau$-tilting infiniteness.
Analogous to the subfamily $\Mri(\mathfrak{F})$ that we already introduced for a given family of algebras $\mathfrak{F}$, we consider the subfamily $\Mtaui(\mathfrak{F})$ which consists of those $\Lambda$ in $\mathfrak{F}$ that are $\tau$-tilting infinite and minimal with respect to this property (i.e, if $\Lambda$ in $\Mtaui(\mathfrak{F})$, then it is $\tau$-tilting infinite and for every proper quotient algebra $\Lambda'$ of $\Lambda$, either $\Lambda'$ is $\tau$-tilting finite or $\Lambda' \notin \mathfrak{F}$). Then, as in the classical case, if $\mathfrak{F}$ is the family of all $k$-algebras, every $\Lambda$ in $\Mtaui(\mathfrak{F})$ is called \emph{minimal $\tau$-tilting infinite}. 
This notion has been independently introduced in the recent work of Wang \cite{W}.

As our first step towards a long-term goal to establish a classification of the minimal $\tau$-tilting infinite algebras similar to those for minimal representation-infinite algebras, we explicitly describe the minimal $\tau$-tilting infinite gentle algebras in terms of their bound quivers, while $\Mtaui(\mathfrak{F}_{\St})$ is treated in a follow-up paper.
By definition, it is evident that for a family $\mathfrak{F}$ of algebras, the intersection of $\Mri(\mathfrak{F})$ and $\Mtaui(\mathfrak{F})$ consists of those rep-infinite algebras in $\mathfrak{F}$ which are minimal with respect to rep-infiniteness and are $\tau$-tilting infinite. It is not hard to see that there are algebra $\Lambda$ in $\Mtaui(\mathfrak{F})$ which are not min-rep-infinite (a concrete family of string algebras in $\Mtaui(\mathfrak{F}_{\St})\setminus \Mri(\mathfrak{F}_{\St})$ is given in the following).

Although we view some of our results from the perspective of minimal $\tau$-tilting infinite algebras, in this paper, we primarily focus on the study of $\tau$-tilting finiteness of the algebras in $\Mri(\mathfrak{F}_{\sB})$, while, for most parts, we do not assume $k$ is algebraically closed. Namely, we study the family $(\mSB)$, but over arbitrary fields.
Note that the subfamilies appearing on the above list are not pairwise disjoint. In particular, in order to extend our results from $\Mri(\mathfrak{F}_{\sB})$ to the family $\Mri(\mathfrak{F}_{\B})$, in Section \ref{section:Minimal representation-infinite algebras and more}, we work on algebraically closed fields and use the new results of Bongartz \cite{Bo3} on $(\mND)$ to characterize the biserial algebras which are non-distributive.

\subsection{Main results}
Now that we have explained the main objectives and methodology, and have also portrayed this work as part of a larger framework, in this subsection we state some of our important results on min-rep-infinite special biserial algebras. At the end of the section we will show that all of the following statements on $\tau$-tilting finiteness of min-rep-infinite special biserial algebras over arbitrary fields extend verbatim to their counterparts in the family of biserial algebras over algebraically closed fields.

To put this work into wider perspective and better fit it into a bigger picture that we aim to complete in our future investigations, some of the statements in this introductory section appear in an order different from the sections where they are shown. However, if a statement appears in wording different from those proved in the text, we give the address to those results in the text from which one can deduce the assertion.

We remark that over any special biserial algebra, a complete classification of all indecomposable $\Lambda$-modules is known and from \cite{BR}, \cite{CB1}, \cite{CB2} and \cite{WW}, we have a good knowledge of the representation theory of such algebras.
In the study of min-rep-infinite algebras and their $\tau$-tilting finiteness, certain vertices of the bound quivers play a pivotal role. In particular, in a bound quiver $(Q,I)$, a vertex $x$ of $Q$ is called a \emph{node} if for any arrow $\alpha$ incoming to $x$ and every arrow $\beta$ outgoing from $x$, we have $\beta \alpha \in I$. We say $\Lambda=kQ/I$ is \emph{node-free} if $(Q,I)$ has no nodes.

In \cite{R2}, where the min-rep-infinite special biserial algebras are studied from a different point of view, Ringel shows that for the objectives of his work it is enough to consider node-free algebras. This is because a process known as resolving the nodes preserves the representation-type of the algebra. 
Among the various fundamental properties of the node-free min-rep-infinite special biserial algebras he studies, an explicit description of them is given in terms of their bound quivers, which are called \emph{cycle, barbell} and \emph{wind wheel} (see Section \ref{Section:Bound Quivers of Mild Special Biserial Algebras} for more details). 
However, as we show in Section \ref{section: Nody algebras}, the process of resolving the nodes does not preserve $\tau$-tilting-type. Therefore, to classify all min-rep-infinite special biserial algebras with respect to $\tau$-tilting finiteness, in addition to the three classes of bound quivers introduced in the Ringel's work, we also need to consider a fourth class where the bound quiver has a node, which we simply call \emph{nody algebras}.

We now briefly describe cycle barbell and nody algebras, leaving the wind wheel algebras, whose definition is somewhat complicated, for later.

It is well-known that for every $n \in \mathbb{Z}_{>0}$ and each acyclic orientation of the affine quiver $\widetilde{\mathbb{A}}_n$, the algebra $k \widetilde{\mathbb{A}}_n$ is min-rep-infinite and $\tau$-tilting infinite. In \cite{R2}, every such algebra is called a cycle algebra and we also use this term to refer to this type of min-rep-infinite special biserial algebras.

Because the nody algebras are defined according to their vertex set (which contains at least one node), after the cycles algebras, they are perhaps the second easiest class of min-rep-infinite biserial algebras one can talk about without specifying the explicit description of their bound quiver. Theorem \ref{node-free-algebra} and Proposition \ref{nody algebras are tau-finite}, imply the following result.

\begin{theorem}
Let $\Lambda=kQ/I$ be a minimal representation-infinite special biserial algebra. If $(Q,I)$ has a vertex $x$ of degree $4$, then $x$ is a node and $\Lambda$ is $\tau$-tilting finite.
\end{theorem}

Before we state the next theorem, let us introduce an important type of bound quiver which plays a curical role in the remainder of the paper. 
For the other purposes in this work, it is convenient for us to take a slightly more general definition of the barbell algebras, comparing to those given in \cite{R2}. In particular, we say $(Q,I)$ is a \emph{generalized barbell}, if it is of the following form:

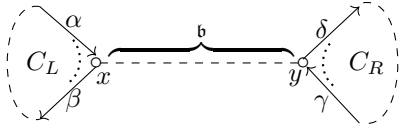
\begin{center}
\begin{tikzpicture}
 \draw [->] (1.25,0.75) --(2,0.1);
    \node at (1.7,0.55) {$\alpha$};
 \draw [<-] (1.25,-0.75) --(2,-0.05);
    \node at (1.7,-0.5) {$\beta$};
  \draw [dashed] (1.25,0.75) to [bend right=100] (1.25,-0.75);
   \node at (1.3,0) {$C_L$};
    \node at (2,0) {$\circ $};
    \node at (2.1,-0.2) {$x$};
    
    \draw [dotted,thick] (1.65,-0.25) to [bend right=50](1.65,0.35);
 \draw [dashed] (2.05,0) --(4.7,0);
 
 \node at (3.4,0.3) {$\overbrace{\qquad \qquad\quad \qquad}^{\mathfrak{b}}$};
 
 \node at (4.75,0.0) {$\circ$};
 \draw [<-] (5.5,0.75) --(4.75,0.05);
    \node at (5,0.45) {$\delta$};
 \draw [->] (5.50,-0.8) --(4.8,-0.05);
    \node at (5,-0.55) {$\gamma$};
  \draw [dashed] (5.55,0.8) to [bend left=100] (5.55,-0.8);
   \node at (5.6,0) {$C_R$};
   \node at (4.65,-0.2) {$y$};
   
\draw [dotted,thick] (5.15,-0.3) to [bend left=50](5.15,0.3);
\end{tikzpicture}
\end{center}
where the bar $\mathfrak{b}$ is a copy of $\mathbb{A}_m$ (possibly of length zero) and $I=\langle \beta \alpha, \delta \gamma \rangle$, while the orientation and length of the dashed segments are arbitrary and we may also have $\alpha=\beta$ and (or) $\delta=\gamma$. The left and right cyclic strings, respectively given by $C_L:=\alpha \cdots \beta$ and $C_R:=\gamma \cdots \delta$, have at most one vertex in common (which occurs exactly when  $l(\mathfrak{b})=0$, implying that $x=y$).
Because $kQ/I$ is a finite dimensional algebra by assumption, if $C_L$ and $C_R$ are both linearly oriented (serial) paths in $Q$, then we must have $l (\mathfrak{b})>0$ (otherwise there exists an oriented cycle in $Q$ which does not belong to $I$).
Following the Ringel's work \cite{R2}, a generalized barbell quiver is called \emph{barbell} if the bar $\mathfrak{b}$ is of positive length.

The next result is essential in this work. The first version of the theorem, which addresses the barbell algebras studied in \cite{R2}, is proved in Sections \ref{Section:tau-Tilting Finite Node-free Special Biserial Algebras}, whereas the proof for the generalized barbell algebras appears in \ref{Section:tau-tilting finite gentle algebras are representation-finite}.

\begin{theorem}
Every generalized barbell algebra is minimal $\tau$-tilting infinite.
\end{theorem}

Building upon the work of Ringel \cite{R2}, one can explicitly describe the bound quivers of all classes of min-rep-infinite biserial algebras that we should study from the viewpoint of $\tau$-tilting theory. 
The following theorem summarizes our main results in this paper and classifies all min-rep-infinite biserial algebras with respect to $\tau$-tilting finiteness. Different parts of the statement are shown in Sections \ref{Section:Bound Quivers of Mild Special Biserial Algebras}, \ref{Section:tau-Tilting Finite Node-free Special Biserial Algebras} and \ref{section: Nody algebras}. 
To help the reader with the less visible aspects of this result, some parts of the following classification which are less immediate to deduce from the table will follow the next theorem as separate assertions. 

\begin{theorem}
Let $\Lambda=kQ/I$ be a minimal representation-infinite special biserial algebra. Then $(Q,I)$ is a cycle, barbell, wind wheel or nody bound quiver. Moreover, such a $\Lambda$ is $\tau$-tilting infinite if and only if $(Q,I)$ is a cycle or a barbell quiver.
\end{theorem}

Let us remark that if a bound quiver $(Q,I)$ contains a sink $a$ and a source $b$, one can simply create a node by changing $(Q,I)$ to $(Q',I')$ via a process known as \emph{gluing} (which identifies the vertices $a$ and $b$ and puts a quadratic relation on the composition of any arrow $\alpha$ incoming to $a$ with any arrow $\beta$ outgoing from $b$).
This, in particular, allows us to create nody algebras from all other types of min-rep-infinite special biserial algebras, provided they have a sink and a source in their bound quiver. Consequently, for almost all integers $d \in \mathbb{Z}_{>0}$, the number of isomorphism classes of $\tau$-tilting finite min-rep-infinite special biserial algebras with $d$ simple modules is significantly larger than their $\tau$-tilting infinite counterparts (see Remark \ref{size of barbell and affine vs nody and wind wheel}).

Now that in the preceding theorem we captured the classification of min-rep-infinite special biserial algebras with respect to $\tau$-tilting finiteness, we state some of the other interesting consequences of our methodology.

\begin{theorem}\label{main-thm}
Let $\Lambda=kQ/I$ be a minimal representation-infinite special biserial algebra. Then, $\Lambda$ is $\tau$-tilting infinite if and only if it is a gentle algebra.
\end{theorem}

As we show in Section \ref{Section:tau-tilting finite gentle algebras are representation-finite}, the generalized barbell quivers could be obtained via the study of $\tau$-tilting finiteness of gentle algebras. This results in a complete classification of minimal $\tau$-tilting infinite gentle algebras, as stated in the following result.

\begin{theorem}
A gentle algebra $A=kQ/I$ is minimal $\tau$-tilting infinite if and only if $(Q,I)$ is one of the following:

\begin{enumerate}
    \item $A$ is a cycle algebra;
    \item $(Q,I)$ is a generalized barbell quiver.
\end{enumerate}
\end{theorem}

Because every min-rep-infinite special biserial algebra is a string algebra, from the previous theorem and the classification of min-rep-infinite special biserial algebras, one immediately concludes that the family of gentle algebras $\mathfrak{F}_{\G}$ and its complement $\mathfrak{F}_{\St}\setminus \mathfrak{F}_{\G}$ inside the family of string algebras $\mathfrak{F}_{\St}$ behave significantly different with respect to the notion of $\tau$-tilting finiteness.
This also gives an explicit family of string algebras $\Lambda$ such that $\Lambda$ is minimal $\tau$-tilting infinite which admits proper algebra quotients $\Lambda/J$ that are rep-infinite, but not necessarily min-rep-infinite.

This shows that for a given family of algebras $\mathfrak{F}$, there could be a tower of algebras between $\tau$-tilting finite members of $\Mri(\mathfrak{F})$ and those algebras in $\Mtaui(\mathfrak{F})$ which are not min-rep-infinite. 
This observation highlights the need for a systematic study of minimal $\tau$-tilting infinite algebras and a classification analogous to the one we briefly reviewed above for the min-rep-infinite algebras over an algebraically closed field.
To approach such a classification, in a forthcoming paper, we further give a complete list of minimal $\tau$-tilting infinite string algebras in terms of their bound quivers.

As a by-product of the above results, we also get the following important corollary.

\begin{corollary}
A gentle algebra is $\tau$-tilting finite if and only if it is representation-finite. 
\end{corollary}

We remark that at the closing stage of this project, we were notified that in \cite{P}, Plamondon also considers the problem of $\tau$-tilting finiteness of gentle algebras and proves the previous assertion independently. In retrospect, our result could be viewed as a generalization of the scope of his work from the subfamily $\Mri(\mathfrak{F}_{\G})$ in the family of gentle algebras to $\Mri(\mathfrak{F}_{\sB})$ in the family of special biserial algebras.

To state the next statement succinctly, let us recall that $\Lambda$ is said to be \emph{mild} if every proper quotient of $\Lambda$ is rep-finite. Obviously, every mild algebra is either rep-finite or min-rep-infinite.
In Sections \ref{Section:Reduction to mild special biserial algebras} and \ref{section:Minimal representation-infinite algebras and more} we show that over every mild special biserial algebra $\Lambda$, as well as any generalized barbell algebra, almost all indecomposable rigid modules are $\tau$-rigid (for definitions and properties, see Sections \ref{Priliminary} and \ref{Section:Reduction to mild special biserial algebras}).
The following theorem states that for all classes of mild special biserial algebras, in fact tilting finiteness and $\tau$-tilting finiteness are equivalent. Hence, viewing this result through the brick-$\tau$-rigid correspondence in \cite{DIJ}, for the class of mild special biserial algebras we can give a new characterization of $\tau$-tilting finiteness, which is not necessarily true for an arbitrary algebra.

\begin{theorem}
For a mild special biserial algebra $\Lambda=kQ/I$, almost every $\tau$-tilting $\Lambda$-module is a tilting module. In particular, the following are equivalent:
    
\begin{enumerate}
    \item $\Lambda$ is $\tau$-tilting finite;
    \item $\Lambda$ is tilting finite;
    \item $\Lambda$ is rigid finite;
    \item $\Lambda$ is $\tau$-rigid finite;
    \item $\Lambda$ is brick finite.
\end{enumerate}
\end{theorem}

The above results allow us to address other aspects of $\tau$-tilting theory of mild special biserial algebras, such as their mutation graphs and their connections to silting theory, etc.
Note that if $\Lambda$ is a mild biserial algebra, the mutation graph of the tilting and $\tau$-tilting $\Lambda$-modules almost coincide, meaning that they only differ in at most finitely many vertices.
Because the nody and wind wheel algebras are $\tau$-tilting finite, the mutation graph of (support) $\tau$-tilting modules is finite. On the other hand, the cycle and (generalized) barbell algebras fit into the setting of the recent joint work of the author in \cite{BD+}, where the mutation of support $\tau$-tilting modules is given in terms of certain collections of combinatorial objects, known as maximal set of non-kissing strings (for more details, see \cite{BD+}). 
Moreover, for every pair of string modules $X$ and $Y$ over any gentle (thus each generalized barbell or cycle) algebra $A$, concrete basis for the spaces $\Hom_{A}(X,\tau Y)$ and $\Ext^1_A(Y,X)$ appear in the aforementioned work.
Using such results, one can further explore the $\tau$-tilting theory of these $\tau$-tilting infinite algebras. We do not investigate this aspect of the $\tau$-tilting theory in this work.

Let us finish this section by an interesting observation which allows us to view most of our results in the significantly larger family of biserial algebras $\mathfrak{F}_{\B}$, provided we work over algebraically closed fields. We recall that $\Lambda$ is \emph{biserial} if for every left and right indecomposable projective $\Lambda$-module $P$, we have $\rad(P)$ is the sum of at most two uniserial modules $X$ and $Y$ and $X \cap Y$ is either zero or a simple module.

Although the representation theory of special biserial algebras is well-understood, one should note that if $\Lambda$ is an arbitrary biserial algebra, so far there is no full classification of indecomposable modules in $\modu \Lambda$. In fact, it is known that the representation theory of biserial algebras is more complicated than that of the special biserial algebras.
Despite this fundamental difference between the family of biserial algebras $\mathfrak{F}_{\B}$ and its proper subfamily $\mathfrak{F}_{\sB}$, in Section \ref{section:Minimal representation-infinite algebras and more} we show that over algebraically closed fields all of our results on $\tau$-tilting finiteness of $\Mri(\mathfrak{F}_{\sB})$ could be extended to those of biserial algebras. To achieve this generalization, we use the fact that the family of biserial algebras is also quotient-closed. From that, it turns out that we only need to verify that among the minimal non-distributive algebras, studied in \cite{Bo3}, all biserial ones are in fact special biserial.

\begin{theorem}\label{Introduction theorem min-rep-inf biserial}
Over an algebraically closed field, $\Mri(\mathfrak{F}_{\B})=\Mri(\mathfrak{F}_{\sB})$. Hence, a minimal representation-infinite biserial algebra is either a cycle, barbell, wind wheel or nody algebra.
\end{theorem}

Through the notion of $\tau$-tilting finiteness, one can also deduce important facts about the components of the Auslander-Reiten quiver of an algebra.
In Section \ref{subsection:Module category of representation-infinite algebras} we make some more general observations on how $\tau$-tilting finiteness interacts with certain components of the Auslander-Reiten quiver of a rep-infinite algebras as well as with the infinite radical of the module category. Particularly, we note that any rep-infinite algebra with a preprojective or preinjective component is $\tau$-tilting infinite.
As a more specific case, in Sections \ref{Section:Reduction to mild special biserial algebras}, \ref{section: Nody algebras} and \ref{subsection: More on fully reduced gentle algebras}, we consider $\Mri(\mathfrak{F}_{\sB})$ and $\Mri(\mathfrak{F}_{\G})$ to highlight the intuition that $\tau$-tilting finiteness can provide into the components of the Auslander-Reiten quivers of the algebras from these families. In particular, we obtain the following result, which results in some explicit sufficient conditions for $\tau$-tilting infinite of arbitrary algebras in Section \ref{subsection:Explicit criteria for tau-tilting infiniteness}.

\begin{theorem}
Let $\Lambda$ be a minimal representation-infinite special biserial algebra. The Auslander-Reiten quiver of $\Lambda$ has a preprojective component if and only if $\Lambda=k\widetilde{\mathbb{A}}_n$ (for some $n \in \mathbb{Z}_{>0}$).
\end{theorem}

\subsection{Notations and conventions}
In addition to what has been already introduced above, here we collect some notations and conventions which will be used in the rest of this paper. The standard materials from representation theory of algebras that appear in this work could be found in \cite{ASS}.
For an algebra $\Lambda$, let $\ind(\Lambda)$ denote the set of isomorphism classes of indecomposable modules in $\modu \Lambda$. For every $M$ in $\modu \Lambda$, let $|M|$ be the number of non-isomorphic summands in the decomposition of $M$ into indecomposable modules. We always assume $|A|=n$. By $\Fac(M)$ we denote the subcategory of $\modu \Lambda$ generated by $M$ (i.e, $\Fac(M)$ consists of all quotients of direct sums of $M$), while $\add(M)$ is the subcategory that consists of all finite direct sums of direct summands of $M$. 
The Auslander-Reiten translation of $M$ is denoted by $\tau_{\Lambda} M$. When there is no confusion, we suppress $\Lambda$ appearing in the notations.
$M$ is called \emph{uniserial} if it has a unique
composition series, namely if the lattice of submodules of $M$ is a chain.

A \textit{quiver} $Q$ always means a finite and connected directed graph, formally given by a quadruple $(Q_0,Q_1,s,e)$, where $Q_0$ and $Q_1$, respectively called the sets of \textit{vertices} and \textit{arrows} of $Q$, are finite. Moreover, $s, e:Q_1\to Q_0$ are two functions sending an arrow $\gamma \in Q_1$ to its \textit{start} and \textit{end}, respectively denoted by $s(\gamma)$ and $e(\gamma)$. A vertex $x \in Q_0$ is a \emph{source} (respectively \emph{sink}) if there is no arrow incoming to (respectively outgoing from) $x$. If there are $m$ arrows incoming to $x$ and $n$ arrows outgoing from $x$, then $x$ is of \emph{degree} $k=m+n$ (or $x$ is called a \emph{$k$-vertex}, for short).
Lower case Greek letters $\alpha$, $\beta$, $\gamma$, $\ldots$ are typically used for arrows of $Q$ and $\beta \alpha$ means first $\alpha$ and then $\beta$. Each \emph{path of length} $d\geq 1$ in $Q$ is a finite sequence of arrows $\gamma_{d}\cdots\gamma_{2}\gamma_{1}$ with $s(\gamma_{j+1})=e(\gamma_{j})$, for every $1 \leq j \leq d-1$. The \emph{lazy} path (of length $0$) associated to vertex $x$ is denoted by $e_x$, and $s(e_x) = e(e_x) = x$.
Let $\overline{Q}$ denote the quiver obtained from $Q$ by adding the reverse arrow $\gamma^*$ for every $\gamma \in Q_1$.

For a quiver $Q$, the \emph{arrow ideal} of the path algebra $kQ$ is denoted by $R_Q$. It is the ideal generated by all arrows of $Q$. An ideal $I \subseteq kQ$ is \emph{admissible} if $R_Q ^m \subseteq I \subseteq R_Q^2$, for some $m\geq 2$. An admissible ideal is called \textit{monomial} if there exists a set of paths in $Q$ which generate $I$. By $P_x$, $I_x$ and $S_x$ we respectively denote the indecomposable projective, injective and simple $kQ/I$-module associated to vertex $x$. For any module $M$ in $\modu (\Lambda)$, by $\pd_{\Lambda} (M)$ and $\id_{\Lambda}(M)$, we respectively denote the projective and injective dimension of $M$.

If $\mathcal{A}$ is an infinite family, for brevity, we say that property $\mathfrak{P}$ holds for \emph{almost all of $\mathcal{A}$} if all but finitely many members of $\mathcal{A}$ satisfy property $\mathfrak{P}$.

\section{Preliminaries and Background}\label{Priliminary}
This section consists of some basic materials needed in the rest of the paper. The proofs of well-known assertions are omitted and, if necessary, only reference is provided.

\subsection{String algebras}
$\Lambda = kQ/I$ is a \emph{string algebra} if it is special biserial and $I$ can be generated by a set of paths in $(Q,I)$. A string algebra $\Lambda=kQ/I$ is \emph{gentle} if $I$ is generated by quadratic relations and $(Q,I)$ satisfies the following condition:
\begin{enumerate}
\item[(G)] For each arrow $\alpha \in Q_1$, there is at most one arrow $\beta$ and at most one arrow $\gamma$ such that $0\ne\alpha\beta \in I$ and $0\ne\gamma\alpha \in I$.
\end{enumerate}

\begin{remark}\label{vertex-quotient}
We freely use the following simple observation: if $\Lambda=kQ/I$ is special biserial (respectively a string, or gentle), for each $x \in Q_0$ and any $\gamma \in Q_1$ the quotient algebras $\Lambda/\langle e_x \rangle$ and $\Lambda/\langle \gamma \rangle$ are again special biserial (respectively string, or gentle). 
Although an arbitrary quotient of a special biserial algebra is again special biserial, this does not necessarily hold for gentle and string algebras (for an explicit non-example, see Example \ref{Example of gentle,string,special biserial algebras}).
\end{remark}

Let $\Lambda = kQ/I$ be special biserial and $Q_1^{-1}:=\{\gamma^{-1} \,| \,\gamma \in Q_1\}$ be the set of formal inverses of arrows of $Q$ (i.e, $s(\gamma^{-1})=e(\gamma)$ and $e(\gamma^{-1})=s(\gamma)$).
A \emph{string} in $\Lambda$ is a word $w =\gamma_k^{\epsilon_k}\cdots\gamma_1^{\epsilon_1}$ with letters in $Q_1$ and $\epsilon_i \in \{\pm 1\}$, for all $1 \leq i \leq k$, such that $w$ satisfies the following properties:

\begin{enumerate}
\item[(S1)] $s(\gamma_{i+1}^{\epsilon_{i+1}})=e(\gamma_i^{\epsilon_i})$ and $ \gamma_{i+1}^{\epsilon_{i+1}} \neq \gamma_i^{-\epsilon_i}$, for all $1 \leq i \leq k-1$;
\item[(S2)] Neither $w$, nor $w^{-1} := \gamma_1^{-\epsilon_1}\cdots\gamma_k^{-\epsilon_k}$, contain a subpath in $I$.
\end{enumerate}

A string $v$ in $(Q,I)$ is called \emph{serial} if either $v$ or $v^{-1}$ is a direct path in $Q$ (i.e, $v=\gamma_k \cdots \gamma_2 \gamma_1$ or $v^{-1}=\gamma_k \cdots \gamma_2 \gamma_1$, for some arrows $\gamma_i$ in $Q_1$).
For a string $w=\gamma_k^{\epsilon_k}\cdots\gamma_1^{\epsilon_1}$, we say it starts at $s(w)=s(\gamma_1^{\epsilon_1})$, ends at $e(w)=e(\gamma_k^{\epsilon_k})$, and is of \emph{length} $l(w):=k$. Moreover, a zero-length string is associated to every vertex and by abuse of notation we denote it by $e_x$, for $x \in Q_0$. 
Suppose $\Str(\Lambda)$ is the set of all equivalence classes of strings in $\Lambda$, where for each string $w$ in $\Lambda$ the equivalence class consists of $w$ and $w^{-1}$ (i.e, set $w \sim w^{-1}$).
A string $w$ is called a \emph{band} if $l(w)>0$ and $w^m$ is a string for each $m \in \mathbb{Z}_{\geq 1}$, but $w$ itself is not a power of a string of strictly smaller length.
For a vertex $x$ of $\Lambda$, we say $w$ in $\Str(\Lambda)$ \emph{visits $x$} if it is supported by $x$. Moreover, $w$ \emph{passes through $x$} provided that there exists a nontrivial factorization of $w$ at $x$ (i.e, there exist $w_1, w_2 \in \Str(\Lambda)$ with $s(w_2)=x=e(w_1)$, such that $l(w_1), l(w_2) >0$ and $w=w_2w_1$).

If $G_Q$ denotes the underlying graph of $Q$, each string $w = \gamma_d^{\epsilon_d}\cdots \gamma_1^{\epsilon_1}$ induces a walk \xymatrix{x_{d+1} \ar@{-}^{\gamma_d}[r] & x_d \ar@{-}^{\gamma_{d-1}}[r] & \cdots & x_1 \ar@{-}_{\gamma_1}[l]} in $G_Q$. Note that a vertex or an edge may occur multiple times in a segment.
The \textit{string module} $M(w)$, determined by $w$, is the representation $ M(w) := ((V_x)_{x \in Q_0}, (\varphi_\alpha)_{\alpha\in Q_1})$ given as follows: Put a copy of $k$ at each vertex $x_i$ of the segment induced by $w$. This step gives the vector spaces $\{V_x\}_{x \in Q_0}$, where $V_{x} \simeq k^{n_x}$ and $n_x$ is the number of times $w$ visits $x$.
For the linear maps of the representation $M(w)$, if $1 \leq i \leq d$,  put the identity map in the direction of $\gamma_i$, between the two copies of $k$ associated to $s(\gamma_i^{\epsilon_i})$ and $e(\gamma_i^{\epsilon_i})$. Namely, this identity map is from $s(\gamma_i)$ to $e(\gamma_i)$ if $\epsilon_i=1$, and it goes from $e(\gamma_i)$ to $s(\gamma_i)$, if $\epsilon_i=-1$.
These two steps generate the indecomposable $\Lambda$-module $M(w)$. Note that, for every string $w$ in $\Lambda$, we have the isomorphism of $\Lambda$-modules $M(w) \simeq M(w^{-1})$. This explains why we put $w \sim  w^{-1}$, meaning that $\{w, w^{-1}\}$ is considered as an equivalence class in $\Str(\Lambda)$.

In addition to the string modules described above, provided that a special biserial algebra has a band, there exists another type of indecomposable modules over it, known as \emph{band modules}, which also have an explicit construction.
The combinatorial structure of the string and band modules will be extensively used in this paper to study brick-finiteness of such algebras. Every band in $\Str(\Lambda)$, under identification with its inverse and its cyclic permutations, gives rise to a family of indecomposable $\Lambda$-modules indexed by a parameter in $k \setminus \{0\}$ and a positive integer.
If $k$ is algebraically closed, every band gives rise to a family of pairwise non-isomorphic indecomposable modules. Over more general fields $k$, each band in $\Str(\Lambda)$ gives rise to additional band modules associated to field extensions of $k$ as well.
Recall that a $\Lambda$-module is \emph{projective-injective} if its both projective and injective over $\Lambda$.
It is well-known that every indecomposable module over special biserial algebras is either a string or a band module, or a non-serial projective-injective module (see \cite{WW} and \cite{BR}). Hence, the knowledge of string and band modules, as well as the homomorphism between them, give a complete understanding of the module category of special biserial algebras.
Since the technical details of the isomorphism classes of band modules over different fields do not play a role in this paper, we only give an example of such representations and refer to \cite{BR} for details on their properties and construction.

Since the following well-known statement is handy at various points of this paper, we state it as a lemma for the future reference.

\begin{lemma}\label{band, rep-infinite}
A string algebra $\Lambda=kQ/I$ is rep-finite if and only if it has no band.
\end{lemma}

For a string algebra $\Lambda$ and a pair of strings $v$ and $w$ in $\Str(\Lambda)$, there is a concrete basis for $\Hom_{\Lambda}(M(w),M(v))$ in terms of \emph{graph maps}, as shown in \cite{CB2}.
Furthermore, if $\Lambda$ is gentle, a combinatorial basis for $\Ext^1_{\Lambda}(M(w),M(v))$ has been recently given in \cite{BD+}.

In the following example we illustrate some of the above-mentioned facts.

\begin{example}\label{Example of gentle,string,special biserial algebras}
For $1 \leq i \leq 4$, consider $\Lambda_i=kQ/I_i$, where $Q$ is the quiver appearing in Figure $1$ and the ideals $I_i$ are as follows:
$$I_1=\langle\beta\alpha,\theta\epsilon,\delta\theta,\beta\epsilon-\gamma\delta\rangle,\,\,I_2=\langle\beta\alpha,\theta\alpha,\theta\epsilon,\delta\theta\rangle,\,\,I_3=\langle\beta\alpha,\theta\epsilon,\delta\theta\rangle\,\,\text{and}\,\,I_4=\langle\beta\alpha,\theta\epsilon,\epsilon\theta\rangle.$$

\begin{figure}
\begin{center}
\begin{tikzpicture}[scale=0.6]
\draw [->] (-2.2,0) --(-0.4,0);
\node at (-1.3,0.3) {$\alpha$};
\node at (-2.5,0) {$^1\bullet$};
\draw [->] (0.2,0) --(1.8,0);
\node at (1,0.3) {$\beta$};
\node at (-0.1,0.1) {$\bullet^2$};
\node at (-0.3,-2) {$_5\bullet$};
\draw [->] (2,-1.8) --(2,-0.2) ;
\node at (2.1,-2) {$\bullet_4$};
\node at (2.3,-1) {$\gamma$};
\draw [<-] (1.8,-2) --(0,-2);
\node at (1,-2.3) {$\delta$};
\node at (2.1,0) {$\bullet^3$};

\draw [->] (-0.25,-0.15) to [bend right=20] (-0.25,-1.8);
\node at (-0.6,-1) {$\theta$};
\draw [<-] (-0.1,-0.1) to [bend left=20] (-0.1,-1.8);
\node at (0.3,-1) {$\epsilon$};
\end{tikzpicture}
\end{center}
\caption{Consider the bound quiver $(Q,I_i)$, for $1 \leq i \leq 4$.}
\label{fig:example}
\end{figure}
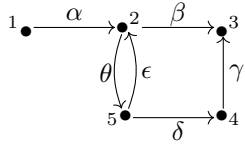

By definition, $\Lambda_1=kQ/I_1$ is special biserial, however it is not a string algebra. In particular, note that $P_5 \simeq I_3$ is a biserial module with a simple socle, which cannot occur over a string algebra.
It is easy to check $\Lambda_2=kQ/I_2$ is a string algebra, but not gentle. Moreover, $\Lambda_3=kQ/I_3$ and $\Lambda_4=kQ/I_4$ are gentle algebras and $\Str(\Lambda_3)$ and $\Str(\Lambda_4)$ respectively contain one band and two bands.

Notice that for every $i \in \{1,2,3,4\}$, the algebra $\Lambda_i$ is rep-infinite type but it is not minimal. However, both $\Lambda_4 /\langle e_1, \theta \rangle $ and $\Lambda_4 /\langle e_1, \epsilon \rangle$ are min-rep-infinite algebras.
We further remark that although all algebra quotients of the form $\Lambda_4/\langle e_x \rangle$ and $\Lambda_4/\langle \zeta \rangle$ are again gentle (for each $x\in Q_0$ and any $\zeta \in Q_1$), if we consider the ideal $J=\langle \beta\epsilon-\gamma\delta \rangle$, then $\Lambda_4/J$ is a biserial algebra, where the indecomposable projective module associated to vertex $5$ has a simple socle. Therefore, $\Lambda_4/J$ is not a string algebra (therefore it is not gentle). This clarifies the point mentioned earlier, showing that, unlike special biserial algebras, string and gentle algebras can admit quotient algebras which are not of the same type.

For $w=\alpha^{-1} \epsilon \delta^{-1} \gamma^{-1} \beta \epsilon$ in $\Str(\Lambda_2)$, the string module $M(w)$ can be depicted by the following diagram (on the left), with a copy of the field $k$ at each vertex and the identity map in the direction of arrows. Moreover, $u=\delta^{-1} \gamma^{-1} \beta \epsilon$ in $\Str(\Lambda_2)$ is a substring of $w$ and the diagram of $M(u)$ appears on the right. 
Although $u$ is a substring of $w$, the module $M(u)$ is not a submodule of $M(w)$. In fact, in this case, $M(u)$ is a quotient module of $M(w)$ and $\Hom_{\Lambda_2}(M(w),M(u))$ is one dimensional, whereas $\Hom_{\Lambda_2}(M(u),M(w))=0$. Further details on a basis of such $\Hom$-spaces of string modules is reviewed in Section \ref{Section:tau-Tilting Finite Node-free Special Biserial Algebras}.

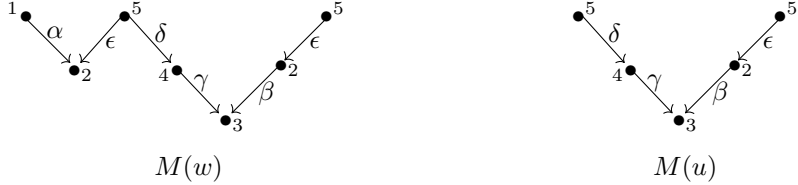
\begin{center}
\begin{tikzpicture}[scale=0.67]


-------------------------
\node at (10.1,1.1) {$\bullet^5$};
\draw [->] (10,1) -- (9.15,0.15);
\node at (9.75,0.5) {$\epsilon$};
--
\node at (9.2,0) {$\bullet_2$};
\draw [->] (9,0) -- (8.1,-0.9);
\node at (8.80,-0.5) {$\beta$};
\node at (8.1,-1.1) {$\bullet_3$};
--
\node at (8.1,-2) {$M(u)$};
--
\draw [->] (7,0) -- (7.85,-0.9);
\node at (7.5,-0.3) {$\gamma$};
\node at (6.9,-0.1) {$_4\bullet$};
--
\draw [->] (6,1.1) -- (6.9,0.1);
\node at (6.7,0.7) {$\delta$};
\node at (6.1,1.1) {$\bullet^5$};

-------------------------
\node at (1.1,1.1) {$\bullet^5$};
\draw [->] (1,1) -- (0.15,0.15);
\node at (0.75,0.5) {$\epsilon$};
--
\node at (0.2,0) {$\bullet_2$};
\draw [->] (0,0) -- (-0.9,-0.9);
\node at (-0.20,-0.5) {$\beta$};
\node at (-0.9,-1.1) {$\bullet_3$};
--
\draw [->] (-2,0) -- (-1.15,-0.9);
\node at (-1.5,-0.3) {$\gamma$};
\node at (-2.1,-0.1) {$_4\bullet$};
--
\node at (-1.75,-2) {$M(w)$};
--
\draw [->] (-3,1.1) -- (-2.1,0.1);
\node at (-2.3,0.7) {$\delta$};
\node at (-2.9,1.1) {$\bullet^5$};
--
\draw [->] (-3.1,1) -- (-3.9,0.1);
\node at (-3.3,0.5) {$\epsilon$};
\node at (-3.9,-0.1) {$\bullet_2$};
--
\draw [->] (-5,1) -- (-4.1,0.1);
\node at (-4.4,0.7) {$\alpha$};
\node at (-5.1,1.1) {$^1\bullet$};

\end{tikzpicture}
\end{center}

Note that $u=\delta^{-1} \gamma^{-1} \beta \epsilon$ is also a band in $\Lambda_2$ and $\Lambda_3$. For a $k$-vector space $V$, let $\phi: V \rightarrow V$ be an invertible $k$-map which is indecomposable (i.e, $\phi=\phi_1 \oplus \phi_2$ implies $\phi_1=0$ or $\phi_2=0$). If $k$ is algebraically closed, this holds if and only if $\phi$ is similar to an indecomposable Jordan block $J_n(\lambda)$ with eigenvalue $\lambda \in k^*$, where $n=\dim (V)$. Then, $M_u(V,\phi)$, shown in the following, is a band module over $\Lambda_2$ and $\Lambda_3$. Note that for every cyclic permutation $v$ of $u=\delta^{-1} \gamma^{-1} \beta \epsilon$ we get $M_u(V,\phi)\simeq M_{v}(V,\phi)$. 
Therefore, as in \cite{BR}, the module associated to every band $u$ in $(Q,I)$ is isomorphic to those associated to the cyclic permutations of $u$ and all of their inverses.

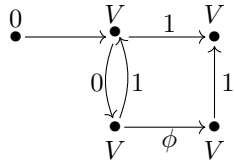
\begin{center}
\begin{tikzpicture}[scale=0.6]
\draw [->] (-2.2,0) --(-0.4,0);
\node at (-2.4,0) {$\bullet$};
\node at (-2.4,0.4) {$0$};
\draw [->] (0.1,0) --(1.8,0);
\node at (1,0.3) {$1$};
\node at (-0.2,0.1) {$\bullet$};
\node at (-0.2,0.5) {$V$};
\node at (-0.2,-2) {$\bullet$};
\node at (-0.2,-2.5) {$V$};
\draw [->] (2,-1.8) --(2,-0.2);
\node at (2,-2) {$\bullet$};
\node at (2,-2.5) {$V$};
\node at (2.3,-1) {$1$};
\draw [<-] (1.8,-2) --(0,-2);
\node at (1,-2.3) {$\phi$};
\node at (2,0) {$\bullet$};
\node at (2,0.5) {$V$};
\draw [->] (-0.25,-0.15) to [bend right=20] (-0.25,-1.8);
\node at (-0.6,-1) {$0$};
\draw [<-] (-0.1,0) to [bend left=20] (-0.1,-1.8);
\node at (0.3,-1) {$1$};

\end{tikzpicture}
\end{center}

Finally, one should note that although the cyclic permutations of two bands are identified, their associated string modules are not necessarily isomorphic. 
For instance, in the preceding example, for the band $u=\delta^{-1} \gamma^{-1} \beta \epsilon$ and its cyclic permutation $v=\epsilon\delta^{-1} \gamma^{-1} \beta $, we observe that $\Hom_{\Lambda_2}(M(v),M(w))\neq 0$, whereas, as remarked before, $\Hom_{\Lambda_2}(M(u),M(w))=0$.
\end{example}

\subsection{$\tau$-Tilting theory}
Classical tilting theory has been a fundamental concept in representation theory and an active area of research with numerous contributions to other domains. For a nice collection of investigations related to this subject, we refer to \cite{AHK}. In tilting theory and its more recent generalizations, the notion of mutation plays a prominent role. To explain this, we first need to recall some standard terminology in this subject.

A $\Lambda$-module $T$ is said to be \textit{partial tilting} if $\Ext^1_{\Lambda}(M,M)=0$ and $\pd_{\Lambda}T \leq 1$. A partial tilting module $T$ is \textit{complete tilting} (or tilting, for short) if additionally there exists a short exact sequence $$0 \rightarrow \Lambda \rightarrow T_1 \rightarrow T_2 \rightarrow 0,$$ with $T_1, T_2 \in \add(T)$. A partial tilting module $M$ with $|M|=|\Lambda|-1$ is called \emph{almost complete tilting}. This is because of a well-known fact, which asserts that a partial tilting module $T$ is tilting if and only if $|T|=|\Lambda|$. 
By Bongartz's completion theorem, for every partial tilting $\Lambda$-module $M$, there exists a $\Lambda$-module $N$ such that $M \oplus N$ is tilting. For the rudiments of tilting theory, see \cite[Chapter VI]{ASS}.

As a fundamental fact in this subject, from \cite{Un} it is known that if $M$ is a basic almost complete tilting module, it appears as a direct summand of at most two complete tilting modules $T_1=M \oplus X_1$ and $T_2= M \oplus X_2$. Moreover, by \cite{HU}, the completion occurs in exactly two ways if and only if $M$ is \emph{faithful}, meaning that $\ann_{\Lambda}(M):= \{a \in \Lambda \,| \, a.M=0 \} =0$. 

The new concept of $\tau$-tilting theory, introduced by Adachi, Iyama and Reiten \cite{AIR}, is primarily aimed at resolving the deficiency of tilting theory with respect to mutation. 
Although there have been other approaches to address the aforementioned deficiency of tilting theory in different settings, in this subsection we only collect some basic materials we need in the $\tau$-tilting theory. For the proofs and further details, we refer to \cite{AIR}, \cite{DIJ} and \cite{DI+}. 

Recall that a $\Lambda$-module $M$ is called \emph{rigid} if $\operatorname{Ext}^1_{\Lambda}(M,M)=0$. By $\rigid (\Lambda)$ we denote the set of all basic rigid $\Lambda$-modules and $\irigid (\Lambda)$ is used to specify the subset of $\rigid (\Lambda)$ consisting of indecomposable ones.
Analogously, $M$ is called \emph{$\tau$-rigid} if $\Hom_{\Lambda}(M,\tau M)=0$. Similarly, let $\tau \trig(\Lambda)$ and $\mathtt{i}\tau \trig(\Lambda)$) respectively denote the set of basic $\tau$-rigid modules and the indecomposable $\tau$-rigid modules. 
A $\tau$-rigid module $M$ is \emph{$\tau$-tilting} if $|M|=|\Lambda|$. Furthermore, $M$ is \emph{support $\tau$-tilting} if $M$ is $\tau$-tilting over $A/\langle e \rangle$, for an idempotent $e$ in $A$. 
By $\tau \ttilt(\Lambda)$ and $s\tau \ttilt(\Lambda)$ we respectively denote the set of all basic $\tau$-tilting modules and that of all basic support $\tau$-tilting modules in $\modu \Lambda$. Consequently, $\Lambda$ is called $\tau$-rigid finite (respectively $\tau$-tilting finite) if $|\tau \trig(\Lambda)| < \infty$ (respectively $|\tau \ttilt(\Lambda)| < \infty$). 
By the celebrated \emph{Auslander-Reiten duality}, namely the functorial isomorphism
$\overline{\Hom}_{\Lambda}(X,\tau_{\Lambda} Y) \simeq D \Ext^1_{\Lambda}(Y,X)$ for each pair of $\Lambda$-modules $X$ and $Y$, it is immediate that every $\tau$-rigid module is rigid. In particular, $\tau \trig(\Lambda)= \rigid(\Lambda)$, provided that $\Lambda$ is hereditary. 

Recall that $M$ is a brick if $\End_{\Lambda}(M)$ is a division algebra. It is known that over algebraically closed fields, $M$ is a brick if and only if $\End_{\Lambda}(M)\simeq k$. In general, for an arbitrary field $k$, by the combinatorial description of basis for the morphisms between string modules, which will be revisited in Section \ref{Section:tau-Tilting Finite Node-free Special Biserial Algebras}, it is easy to show that a string module $X$ is a brick if and only if $\End_{\Lambda}(X)$ is one-dimensional.
By $\Brick(\Lambda)$ we denote the set of bricks in $\modu \Lambda$.

A subcategory of $\modu \Lambda$ is a \emph{torsion class} if it is closed under quotient and extension. A torsion class $\mathcal{T}$ is {\it functorially finite}, in the sense of \cite{AR1}, if $\mathcal{T}=\Fac (M)$, for some $\Lambda$-module $M$. By $\tors(\Lambda)$ we denote the set of all torsion classes in $\modu \Lambda$, while $\ftors (\Lambda)$ is the set of all functorially finite torsion classes. Recall that in a subcategory $\mathcal{C}$ of $\modu \Lambda$, a module $X \in \mathcal{C}$ is \textit{$\Ext$-projective} if $\Ext^1_A(X,-){|_\mathcal{C}}=0$. 

The following theorem establishes a connection between the above notions.
\begin{theorem}[{\cite[Theorem 2.7]{AIR}}]\label{st-tor}
For an algebra $\Lambda$, there is a bijection between $s\tau \ttilt (\Lambda)$ and $\ftors(\Lambda)$. In particular, in one direction, each basic support $\tau$-tilting module $X$ is sent to $\Fac (X)$. In the other direction, every $\mathcal{T} \in \ftors(\Lambda)$ is sent to $X_{\mathcal{T}}=\bigoplus X_i$, where the direct sum runs over the isomorphism classes of all $\Ext$-projective indecomposable modules $X_i$'s in $\mathcal{T}$.  
\end{theorem}

Moreover, we collect some important equivalent conditions for $\tau$-tilting finiteness of algebras in the following theorem.

\begin{theorem}[\cite{DIJ, DI+}] \label{tau-finiteness}
For an algebra $\Lambda$, the following are equivalent.
\begin{enumerate}
\item $\Lambda$ is $\tau$-tilting finite;
\item $\ftors(\Lambda)= \tors(\Lambda)$;
\item $\tors(\Lambda)$ is finite;
\item $\mathtt{i} \tau \trig(\Lambda)$ is finite;
\item $\Brick(\Lambda)$ is finite.
\end{enumerate}
\end{theorem}

Let us finish this subsection by recalling the following important results in the $\tau$-tilting theory from the lattice theoretical point of view. In particular, we will extensively use a more elementary version of it to reduce our main problems to the min-rep-infinite algebras, which are significantly more tractable.

\begin{theorem} [\cite{DI+}]\label{quotient-lattice}
Every surjective morphism $\phi: \Lambda_1 \rightarrow \Lambda_2$ of algebras induces a surjective lattice map $\tilde{\phi}: \tors(\Lambda_1) \rightarrow \tors(\Lambda_2)$, defined by $\tilde{\phi}(\mathcal{T}):=\mathcal{T}\cap \modu \Lambda_2$, for each $\mathcal{T} \in \tors(\Lambda_1)$. In particular, if $\Lambda_1$ is $\tau$-tilting finite, so is $\Lambda_2$.
\end{theorem}

\subsection{Representation-infinite algebras}\label{subsection:Module category of representation-infinite algebras}

Representation-finite algebras and local algebras form well-known families of $\tau$-rigid finite algebras. 
Thus, in our comparison between the notions of $\tau$-tilting finiteness and rep-infiniteness, we will be primarily interested in rep-infinite algebras. It is known that an algebra $\Lambda$ is rep-infinite if and only if every component of $\Gamma (\modu \Lambda)$ is infinite, where $\Gamma (\modu \Lambda)$ denotes the Auslander-Reiten quiver of $\modu \Lambda$ (for a proof, see \cite[IV.5.4]{ASS}). This itself is equivalent to $\rad^{\infty}(\modu \Lambda) \neq 0$, where $\rad(\modu \Lambda)$ denotes the Jacobson radical of the category $\modu \Lambda$. Here we collect some fundamental facts about the module category of rep-infinite string algebras, because they play a crucial role in this paper.

From \cite{BR} we know that for a rep-infinite string algebra $\Lambda$, every component of $\Gamma(\modu \Lambda)$ falls into exactly one of the following three families.
\begin{enumerate}[(I)]
\item Finitely many non-periodic components whose boundaries are given by projective or injective modules.
\item Infinitely many periodic components of the form $\mathbb{ZA}^{\infty}/\langle \tau^{k} \rangle$, known as a \emph{tube of rank $k$}, for some $k \in \mathbb{Z}_{>0}$.
\item Infinitely many non-periodic components without boundary which are of the form $\mathbb{ZA}^{\infty}_{\infty}$.
\end{enumerate}

A tube of rank $1$ is often called \emph{homogeneous}. Among the components listed above, the second and third families are \emph{stable}, meaning that for every $X$ in such components, $\tau(X)$ and $\tau^{-1}(X)$ are nonzero. Moreover, any component $\mathcal{C}$ that contains a band module is of the second type and $\mathcal{C}$ consists of only band modules. In this case, $\mathcal{C}$ is a homogeneous tube and is called a \emph{band tube}. In particular, for each band module $X$, we have $\tau_\Lambda(X)=X$, implying that band modules are never $\tau$-rigid. Therefore, as long as our aim is to study $\tau$-tilting modules, we need not to consider band modules.
Note that among all tubes, only finitely many of them contain (therefore consist of) string modules, which we call \emph{string tubes}.

Before showing that $\tau$-tilting finiteness of $\Lambda$ implies strong constraints on components of $\Gamma (\modu \Lambda)$ and the Jacobson radical of $\modu \Lambda$, let us recall some definitions.

\begin{definition}
For an arbitrary algebra $\Lambda$, a (connected) component $\mathcal{P}$ of $\Gamma(\modu \Lambda)$ is called \emph{preprojective} if $\mathcal{P}$ is acyclic and for any $M$ in $\mathcal{P}$, there exists a vertex $x\in Q_0$ and some $m\in \mathbb{Z}_{>0}$ such that $M\simeq \tau^{-m}(P_x)$. An indecomposable module $X$ in $\modu \Lambda$ is called \emph{preprojective} if it belongs to a preprojective component of $\Gamma(\modu \Lambda)$. \emph{Preinjective} components of $\Gamma(\modu \Lambda)$ as well as \emph{preinjective} indecomposable modules are defined dually.
\end{definition}

The following lemma is a useful result that we need in the following sections.

\begin{lemma}\cite[VIII.2.7]{ASS}\label{Preproj/Postinj brick}
For an arbitrary algebra $\Lambda$, every indecomposable preprojective or preinjective module is a $\tau$-rigid brick. 
\end{lemma}

\begin{remark}
The above lemma often gives a handy criterion to verify $\tau$-tilting infiniteness of a large family of algebras: Due to Theorem \ref{tau-finiteness} and the fact that brick-finiteness is preserved under surjective algebra morphisms, for a given algebra $\Lambda=kQ/I$ if we find a rep-infinite quotient algebra $\Lambda'$ whose Auslander-Reiten quiver contains a preprojective or preinjective component, then $\Lambda'$, and thus $\Lambda$, are $\tau$-tilting infinite. To efficiently employ this simple reductive method and verify $\tau$-tilting infiniteness of a large family of algebras via their quotients, one needs to have an explicit characterization of a wide range of $\tau$-tilting infinite algebras which are minimal with respect to this property.

In \cite{HV}, Happel and Vossieck consider a family of algebras which closely relates to the scope of our work. In the aforementioned article, the authors study the algebras which they call minimal representation infinite, but which we call weakly minimal representation-infinite, in order to highlight the difference with the notion of minimal representation-infinite algebras we study here.
We say $\Lambda=kQ/I$ is \emph{weakly minimal representation-infinite} (or weak-min-rep-infinite, for short) if $\Lambda$ is rep-infinite, but for every $x \in Q_0$ the quotient algebra $\Lambda/\langle e_x \rangle$ is rep-finite. 
Note that every min-rep-infinite algebra is obviously weak-min-rep-infinite, whereas the converse does not hold (for example, a generalized Kronecker quiver with three arrows is weak-min-rep-infinite, but not min-rep-infinite).
In \cite{HV}, a full list of weak-min-rep-infinite algebras whose Auslander-Reiten quivers contain a preprojective component is given via a concrete description of bound quivers and certain admissible operations.
In light of such a concrete classification and some further classes we study in this paper, we can often easily show that a given algebra $\Lambda=kQ/I$ is $\tau$-tilting infinite, simply by our reductive method, which amounts to finding one of the aforementioned bound quivers inside the bound quiver $(Q,I)$. 
\end{remark}

From the above lemma and some well-known facts which we alluded to earlier in this section, it follows that we in fact want to study the $\tau$-tilting finiteness of those special biserial algebras $\Lambda$ such that $\rad^{\infty}(\modu \Lambda)\neq 0$ and $\Gamma(\modu \Lambda)$ has no preprojective and preinjective component.
That being the case, we aim to determine whether $\modu \Lambda$ contains infinitely many isomorphism classes of indecomposable modules $M$ with $\End_{\Lambda}(M)\simeq k$. In the next section, we show how to reduce the above problem to a certain family of string algebras. 

We also need the following lemma.
\begin{lemma} \cite[\textrm{X}.3.3 and X.4.5]{SS} \label{tube-brick mouth}
For an algebra $\Lambda$ and a stable tube $\mathcal{T}$ of $\Gamma(\modu \Lambda)$, the following are equivalent:
\begin{enumerate}
\item If $\{E_1,\cdots,E_r\}$ is the mouth of $\mathcal{T}$, each $E_i$ is a brick and $\Hom(E_i,E_j)=0=\Hom(E_j,E_i)$, for $1\leq i<j \leq r$.
\item $\rad^{\infty}_{\Lambda}(X,Y)=0$, for every pair $X$ and $Y$ in $\mathcal{T}$.
\item $\rad^{\infty}_{\Lambda}(X,Y)=0$, for every pair $X$ and $Y$ on the mouth of $\mathcal{T}$.
\end{enumerate}
\end{lemma}

A component $\mathcal{C}$ of $\Gamma(\modu \Lambda)$ is called \emph{generalized standard} if $\rad^{\infty}_{\Lambda}(X,Y)=0$, for every pair $X$ and $Y$ in $\mathcal{C}$. By the above lemma, for a stable tube $\mathcal{T}$, this is equivalent to the vanishing of $\rad^{\infty}_{\Lambda}(X,Y)$ for every pair of mouth modules.
Moreover, $M$ and $N$ in $\modu \Lambda$ are called \emph{$\Hom$-orthogonal} if $\Hom(M,N)=0=\Hom(N,M)$.
Hence, the previous lemma states that a tube $\mathcal{T}$ is generalized standard if and only if its mouth consists of pairwise $\Hom$-orthogonal bricks. 
Via bricks and their crucial role in the study of $\tau$-tilting theory, we also observe that not only $\tau$-tilting finiteness controls the preprojective and preinjective components of $\Gamma(\modu \Lambda)$, but also it regulates the stable tubes. Note that if $\Lambda$ is a string algebra with only finitely many isomorphism classes of bricks, almost all tubes are non-standard. For an explicit family of such algebras, see Example \ref{Major NON-Example}.

In the following proposition, we collect some interesting properties of $\tau$-tilting finite algebras in terms of the components of the Auslander-Reiten quiver and the Jacobson radical. Recall that $\Lambda$ is said to be \emph{concealed of type $Q$} if $\Lambda \simeq \End_{kQ}(T)$, where $Q$ is an acyclic quiver of non-Dynkin type and $T$ is a preprojective tilting module over $kQ$. For details on such algebras, see \cite[VIII.4]{ASS}.

\begin{proposition}\label{tau-finiteness on the components and radical}
Let $\Lambda$ be a rep-infinite algebra. If $\Lambda$ is $\tau$-tilting finite, then 
\begin{enumerate}
\item $\Gamma(\modu \Lambda)$ has no preprojective nor preinjective component. Furthermore, it has only finitely many generalized standard tubes.

\item $\rad^{\infty}_{\Lambda}(X,X) \neq 0$, for some $X \in \ind(\Lambda)$.

\item For almost all homogeneous tubes $\mathcal{T}_h$ with the mouth module $E_h$, there exists a positive integer $k_h$ with 
$\rad^{k_h}_{\Lambda}(E_h,E_h)=\rad^{k_h+1}_{\Lambda}(E_h,E_h)\not =0$. 

\end{enumerate}
\end{proposition}

\begin{proof}
\begin{enumerate}
\item This part follows from Theorem \ref{tau-finiteness}, Lemma \ref{Preproj/Postinj brick} and Lemma \ref{tube-brick mouth}.

\item If $\rad^{\infty}_{\Lambda}(X,X)=0$ for every $X \in \ind(\Lambda)$, by \cite[Theorem 4.5]{Sk}, there exists an ideal $I$ in $\Lambda$ such that $\Lambda/I$ is tame concealed. Hence, $\Lambda/I$ has preprojective and preinjective components. By the first part and Theorem \ref{quotient-lattice}, this is impossible.

\item Each $\Hom_\Lambda(X,Y)$ is finite dimensional for each pair $X$ and $Y$ in $\modu \Lambda$. Thus, a component $\mathcal{C}$ is generalized standard if and only if for every $X$ and $Y$ in $\mathcal{C}$ there exists an integer $k:=k(X,Y) \geq 0$ such that $\rad^{k}_{\Lambda}(X,Y)=0$.
Since $\Lambda$ is $\tau$-tilting finite, by Theorem \ref{tau-finiteness}, $\modu \Lambda$ contains only finitely many isomorphism classes of bricks. Hence, the result follows from the first part and Lemma \ref{tube-brick mouth}

\end{enumerate}
\end{proof}

The simplest min-rep-infinite algebras are isomorphic to $k\widetilde{\mathbb{A}}_n$ (for some $n \in \mathbb{Z}_{>0}$ and an acyclic orientation of $\widetilde{\mathbb{A}}_n$). As we will see in the following sections, these hereditary algebras are crucial in the study of min-rep-infinite special biserial algebras. 
For the sake of the intuition that it provides, the Auslander-Reiten quiver of such $k\widetilde{\mathbb{A}}_n$ over algebraically closed fields is illustrated in Figure \ref{fig:AR component of A-tilde}, where $\mathcal{R}_{\lambda}$ denote infinitely many band tubes, whereas the string tubes are shown by $\mathcal{R}_0$ and $\mathcal{R}_{\infty}$. The preprojective ($\mathcal{P}$) (respectively preinjective ($\mathcal{I}$)) component has all of the indecomposable projective (respectively injective) modules on its leftmost (respectively rightmost) boundary.

\begin{figure}
    \centering
    \begin{tikzpicture}
\node at (1,0.5) {$\mathcal{P}$};
\node at (1,-0.25) {$\text{preprojective}$};
    \draw decorate [decoration={snake}] {(0,1) -- (0,0)};
    \draw[-] (0,1) -- (1.5,1); \draw[dashed] (1.5,1) -- (2,1);
    \draw[-] (0,0) -- (1.5,0); \draw[dashed] (1.5,0) -- (2,0);

\node at (3,0.5) {$\mathcal{R}_{0}$};
    \draw[-] (2.5,0.5) -- (2.5,1.5) -- (3.5,1.5) -- (3.5,0.5) ;  \draw[dashed] (2.5,0.5) -- (2.5,-0.5); \draw[dashed] (3.5,0.5) -- (3.5,-0.5);

    \draw[-] (4,0.5) -- (4,1.5) -- (4.5,1.5) -- (4.5,0.5) ;  \draw[dashed] (4,0.5) -- (4,-0.5); \draw[dashed] (4.5,0.5) -- (4.5,-0.5);

\node at (5.25,0.75) {$\cdots$};
\node at (5.25,0.25) {$\mathcal{R}_{\lambda}$};
\node at (5.25,-0.2) {$({\lambda \in k^*})$};
    
    \draw[-] (6,0.5) -- (6,1.5) -- (6.5,1.5) -- (6.5,0.5) ;  \draw[dashed] (6,0.5) -- (6,-0.5); \draw[dashed] (6.5,0.5) -- (6.5,-0.5);

\node at (5.25,-0.8) {$\underbrace{\qquad \qquad \qquad \qquad }_{\text{band tubes}}$};
   
\node at (7.5,0.5) {$\mathcal{R}_{\infty}$};   
    \draw[-] (7,0.5) -- (7,1.5) -- (8,1.5) -- (8,0.5) ;  \draw[dashed] (7,0.5) -- (7,-0.5); \draw[dashed] (8,0.5) -- (8,-0.5); 

\node at (9.5,0.5) {$\mathcal{I}$};
\node at (9.5,-0.25) {$\text{preinjective}$};
\draw decorate [decoration={snake}] {(10.5,1) -- (10.5,0)};
\draw[-] (10.5,1)--(9,1); \draw[dashed] (9,1)--(8.5,1);
\draw[-] (10.5,0)--(9,0); \draw[dashed] (9,0)--(8.5,0);
    \end{tikzpicture}
    \caption{The Auslander-Reiten quiver of $\widetilde{\mathbb{A}}_n$.}
    \label{fig:AR component of A-tilde}
\end{figure}

\section{Reduction to Mild Special Biserial Algebras}\label{Section:Reduction to mild special biserial algebras}

To explain our reductive method, analogous to the notation we introduced for a family of algebras, for every algebra $\Lambda$ we define
$$\Mri(\Lambda):= \{\Lambda/J \, | \, \text{$J$ is an ideal of $\Lambda$ and } \Lambda/J \text{ is min-rep-infinite} \}/ \sim,$$
where $\Lambda/J \sim \Lambda/J'$ if and only if they are isomorphic as $k$-algebras. Obviously, $\Mri(\Lambda)=\emptyset$ if and only if $\Lambda$ is rep-finite.
Moreover, as mentioned before, every surjective algebra map $\phi:\Lambda_1 \rightarrow \Lambda_2$ induces an embedding $\modu \Lambda_2 \hookrightarrow \modu \Lambda_1$ of the categories. Therefore, if $\Lambda_1$ is rep-finite (similarly brick-finite) so is $\Lambda_2$.
By Theorem \ref{tau-finiteness}, this implies that surjective algebra morphisms preserve $\tau$-tilting finiteness and representation-finiteness. This, in particular, shows that if there exists $\Lambda' \in \Mri(\Lambda)$ which is $\tau$-tilting infinite, then so is $\Lambda$. However, one should note that there exist $\tau$-tilting infinite algebras $\Lambda$ such that every $\Lambda' \in \Mri(\Lambda)$ is $\tau$-tilting finite. Certain algebras of this type are systematically studied in Section \ref{Section:tau-Tilting Finite Node-free Special Biserial Algebras}, as well as Sections \ref{Section:tau-tilting finite gentle algebras are representation-finite} and \ref{section:Minimal representation-infinite algebras and more}.

As before, we call a family $\mathfrak{F}$ of algebras \emph{quotient-closed} provided that for any surjective algebra map $\phi:\Lambda_1 \rightarrow \Lambda_2$, if $\Lambda_1\in \mathfrak{F}$, then $\Lambda_2 \in \mathfrak{F}$.
By $\mathfrak{F}_{\repf}$ and $\mathfrak{F}_{\tau \mathrm{f}}$ we respectively denote the quotient-closed family of rep-finite algebras and that of $\tau$-tilting finite algebras, where $\mathfrak{F}_{\repf} \subsetneq \mathfrak{F}_{\tau \mathrm{f}}$. This is because, for example, $\mathfrak{F}_{\tau \mathrm{f}}$ also contains local algebras. 
As remarked earlier, among the well-known families of the gentle, the string and the special biserial algebras (respectively denoted by $\mathfrak{F}_{\G}$, $\mathfrak{F}_{\St}$ and $\mathfrak{F}_{\sB}$), only the last of these is quotient-closed, while we obviously have $\mathfrak{F}_{\G} \subsetneq \mathfrak{F}_{\St} \subsetneq \mathfrak{F}_{\sB}$.

As explained in Section \ref{Introduction}, the problems related to the $\tau$-tilting infiniteness of the algebras from a family $\mathfrak{F}$ can be reduced to the subfamily $\Mri(\mathfrak{F})$, which consists of rep-infinite algebras in $\mathfrak{F}$ that are minimal with respect to this property. Namely, every $\Lambda$ in $\Mri(\mathfrak{F})$ is rep-infinite and for each proper quotient $\Lambda'$ of $\Lambda$, either $\Lambda'$ is rep-finite, or $\Lambda' \notin \mathfrak{F}$. 
Note that, if $\mathfrak{F}$ is quotient-closed, we have
$$\Mri(\mathfrak{F}):= \{\Lambda \in \mathfrak{F} \, | \, \text{$\Lambda$ is min-rep-infinite} \},$$
meaning that a rep-infinite algebra $\Lambda$ which is minimal in $\mathfrak{F}$ is in fact a minimal representation-infinite algebra. This is because, unlike an arbitrary family of algebras, if $\Lambda$ is in a quotient-closed family $\mathfrak{F}$, we always have $\Mri(\Lambda) \subseteq \mathfrak{F}$. 
In such a case, a complete description of the $\tau$-tilting theory of the minimal rep-infinite algebras in $\mathfrak{F}$ could provide a new insight into the $\tau$-tilting theory of arbitrary algebras in $\mathfrak{F}$.

In this paper, we are particularly concerned with the family of special biserial algebras $\mathfrak{F}_{\sB}$ and we aim to give an explicit classification of the algebra in $\Mri(\mathfrak{F}_{\sB})$ with respect to the $\tau$-tilting finiteness. This results in concrete sufficient conditions for $\tau$-tilting-infiniteness of any rep-infinite algebra which has a (special) biserial quotient algebra.

Motivated by the new developments in the study of minimal representation-infinite algebras, in \cite{R2}, Ringel has recently studied the algebras in $\Mri(\mathfrak{F}_{\sB})$ from a point of view different from what we adopt in this paper. We recall some of his results that we use in our work.

The next lemma is crucial in the study of min-rep-infinite algebras. Although each part of the assertion is known, for self-containment we present a proof.

\begin{lemma}\label{Monomial ideal}
Let $\Lambda$ be an algebra and $X$ be a projective-injective $\Lambda$-module. There exists a nonzero ideal $J$ such that $\Lambda$ and $\Lambda/J$ are of the same representation type. In particular, if $\Lambda$ is a min-rep-infinite algebra, $\modu \Lambda$ contains no projective-injective module. Thus, each min-rep-infinite special biserial algebra is a string algebra.
\end{lemma}

\begin{proof}
Let $\Lambda=kQ/I$ be an arbitrary algebra and $M$ be an indecomposable projective-injective $\Lambda$-module. Thus, $M=P_x=I_y$, for some vertices $x$ and $y$ in $Q$.
Therefore, $M= \Lambda e_x$ and $\soc(M)$ is a subspace of $e_y \Lambda e_x$. 
Note that $\soc(M)$ is a simple module and therefore it is a one-dimensional vector space over $k$ (because $\Lambda$ is basic).
By abuse of notation, if $\soc(M)$ also denotes the element of $e_y \Lambda e_x$ which associates to the module $\soc(M)$ and we put $J:=\langle \soc(M) \rangle$ for the ideal generated by that, then $J$ is also one-dimensional and we have $e_y J e_x= \soc(M)$.

Then, $JM \neq 0$, whereas for any indecomposable $\Lambda$-module $N$ (non-isomorphic to $M$), we have $JN=0$, so $N$ is a nonzero $\Lambda/J$-module. This implies that all indecomposable $\Lambda$-modules non-isomorphic to $M$ are nonzero indecomposable $\Lambda/J$-module, hence the first assertion follows. 
Now, observe that if $\Lambda$ is min-rep-infinite and $J \neq 0$, then $\Lambda$ and $\Lambda/J$ are of the same representation-type, which is impossible (because by minimality assumption of $\Lambda$, the proper quotient algebra $\Lambda/J$ is rep-finite). This shows that there cannot exist any projective-injective $\Lambda$-module.

To prove the last assertion, from \cite{SW} we know that if $\Lambda=kQ/I$ is a special biserial algebra, $I$ can be generated by a set of paths in $Q$ and relations of the form $p_1=\lambda p_2$, where $\lambda \in k$ is a nonzero scalar and $p_1$ and $p_2$ are two distinct paths in $Q$ such that $s(p_1)=s(p_2)$ and $e(p_1)=e(p_2)$, and furthermore $p_1$ and $p_2$ share no other vertex. According to the configuration of the special biserial bound quivers (see Section \ref{Introduction} for their properties), if $\Lambda=kQ/I$ is min-rep-infinite special biserial algebra and every minimal set of generators $R$ for $I$ contains a relation of the second type, then $P_x=I_y$, where $x=s(p_1)=s(p_2)$ and $y=e(p_1)=e(p_2)$. This gives the desired contradiction.
\end{proof}

From the previous lemma, henceforth we assume that any min-rep-infinite special biserial algebra $\Lambda=kQ/I$ comes with a minimal set of monomial generators for $I$. Therefore, the concrete description of the isomorphism classes of indecomposable $\Lambda$-modules, as well as components of $\Gamma (\modu \Lambda)$, will be at our disposal (see Section \ref{Priliminary}).

To state some important properties of the algebras in $\Mri(\mathfrak{F}_{\sB})$, we need to recall some standard terminology. Analogous to what has been already recalled for the modules, a component $\mathcal{C}$ of $\Gamma (\modu \Lambda)$ is called \emph{faithful} if $\ann(\mathcal{C})=0$, where $$\ann(\mathcal{C}):= \bigcap_{X\in \mathcal{C}} \ann(X).$$
Moreover, a $\Lambda$-module $M$ is \emph{sincere} if $e_xM \neq 0$, for every vertex $x$ of $\Lambda$, meaning that $M$ is supported at every vertex. In particular, a string or band module $M$ is sincere if and only if the associated string in $(Q,I)$ visits all the vertices in $Q_0$.
A family of modules $\mathcal{X}$ in $\modu \Lambda$ is called $\emph{hereditary}$ provided that $\pd_{\Lambda}X \leq 1$ and $\id_{\Lambda} X \leq 1$, for every $X\in \mathcal{X}$.
In particular, $\Lambda$ is hereditary if and only if $\modu \Lambda$ is so, which itself is equivalent to $\Lambda \simeq kQ$, for an acyclic quiver $Q$.
$\Cogen_{\Lambda}(M)$ denotes the set of $\Lambda$-modules $Y$ such that there exists a positive integer $d$ and a monomorphism of $\Lambda$-modules $\iota: Y \rightarrow M^d$.

The next lemma establishes a useful connection between these notions.

\begin{lemma}\label{faithful-cogen}
\cite[Lemma VI.2.2]{ASS} \label{Faitful-Cogen}
Let $\Lambda$ be an algebra and $M$ in $\modu \Lambda$. Then, $\ann_{\Lambda}(M)=0$ if and only if $\Lambda \in \Cogen(M)$.
\end{lemma}

Although the following lemma holds in the more general setting, for concreteness and since we will employ the assertions in the study of $\Mri(\mathfrak{F}_{\sB})$, we restrict to the family of special biserial algebras.

\begin{lemma}\label{Almost all}
Let $\Lambda$ be a mild special biserial algebra.
\begin{enumerate}
\item $\ann(M(w))=0$, for almost every $w \in \Str(\Lambda)$. Moreover, every component $\mathcal{C}$ in $\Gamma(\modu \Lambda)$ is faithful.
\item Every generalized standard tube $\mathcal{T}$ in $\Gamma(\modu \Lambda)$ is hereditary. If $\mathcal{T}$ is homogeneous, for every $X$ in $\mathcal{T}$ and $i \in Q_0$, $\Hom_\Lambda(I_i,X)=\Hom_\Lambda(X,P_i)=0$.
\end{enumerate}
\end{lemma}

\begin{proof}
Assume $\Lambda$ is min-rep-infinite, because otherwise there is nothing to show.

By Lemma \ref{Monomial ideal}, $\Lambda$ is a string algebra, so all indecomposable $\Lambda$-modules are string or band modules. 
Since $\Lambda$ is $\mSB$, almost every indecomposable $\Lambda$-module $X$ is supported on all arrows (thus vertices) of $\Lambda$, therefore $X$ is sincere. By construction, it is straight forward to check that almost every sincere string $\Lambda$-module is faithful. To show the second part of $(1)$, note that for any component $\mathcal{C}$ of $\Gamma(\modu \Lambda)$, if $\Lambda':=\Lambda / \ann_{\Lambda}(\mathcal{C})$, then $\mathcal{C}$ is a faithful component of $\Gamma(\modu \Lambda')$. Since $\Lambda$ is min-rep-infinite, $\Lambda'$ is finite type if and only if $\ann_{\Lambda}(\mathcal{C})\neq 0$. Because every component $\mathcal{C}$ of $\Gamma(\modu \Lambda)$ is infinite, we get $\ann_{\Lambda}(\mathcal{C})= 0$.

To prove $(2)$, suppose $\mathcal{T}$ is a generalized standard tube in $\Gamma(\modu \Lambda)$. From the proof of part (1), $\mathcal{T}$ admits a faithful object $X$ (in fact, infinitely many). Hence, by Lemma \ref{Faitful-Cogen}, $\Lambda \in \Cogen(X)$. The first part of the statement follows from \cite[Theorems X.4.4]{SS}. For the second assertion, by \cite[IV.2.7]{ASS}, for each $X\in \modu \Lambda$, we know $\pd_{\Lambda}X \leq 1$ if and only if $\Hom_\Lambda(D\Lambda,\tau X)=0$ (dually, $\id_{\Lambda} X \leq 1$ if and only if $\Hom_\Lambda(\tau^{-1}X,\Lambda)=0$). Since for every $M$ in a homogeneous tube $\mathcal{T}$ we have $\tau_{\Lambda}(M)=M$, the result is immediate.

\end{proof}

By the characterization of generalized standard tubes in Lemma \ref{tube-brick mouth} and faithfulness of components of $\Gamma(\modu \Lambda)$ in Lemma \ref{Almost all}, we will give a new characterization of $\tau$-tilting finiteness of mild special biserial algebras. Before we state the next proposition, we remark that for each tilting module $T$ over $\Lambda$, the existence of the coresolution
$$0 \rightarrow \Lambda \rightarrow T_1 \rightarrow T_2 \rightarrow 0,$$ with $T_1, T_2 \in \add(T)$ implies that $\Lambda \in \Cogen(T)$, thus, by Lemma \ref{Faitful-Cogen}, every tilting module is faithful.

\begin{proposition} \label{faithful tau-rigid}
For an algebra $\Lambda$, the following hold:
\begin{enumerate}
    \item Every faithful $\tau_{\Lambda}$-rigid module is partial tilting. Therefore, each faithful $\tau$-tilting ${\Lambda}$-module is tilting. 
    
    \item A $\Lambda$-module is $\tau$-tilting if and only if it is sincere support $\tau$-tilting.
    
    \item A $\Lambda$-module is tilting if and only if it is faithful support $\tau$-tilting.
\end{enumerate}

\end{proposition} 
\begin{proof}
For the proof of the first statement, see {\cite[VIII.5.1]{ASS}}. The second and third assertions are shown in \cite[Prop. 2.2]{AIR}.
\end{proof}

The following statement establishes a further connection between the notion of rigidity and $\tau$-rigidity over the algebras in $\Mri(\mathfrak{F_{\sB}})$.

\begin{proposition}\label{tau-finite vs tilting finite}
If $\Lambda$ is a mild special biserial algebra, then $\mathtt{i} \tau \trig(\Lambda)
$ is finite if and only if $\irigid(\Lambda)$ is finite.

\end{proposition}

Before proving the above proposition, let us remark that the assertion does not necessarily hold for arbitrary algebras. Namely, for an algebra $\Lambda$, if $\mathtt{i} \tau \trig(\Lambda)$ is finite, we cannot conclude $\Lambda$ is rigid finite. The preprojective algebras of (simply laced) Dynkin quivers provide concrete non-examples. This is because for a simply laced Dynkin quiver $Q$, in \cite{GLS} the module category of the preprojective algebra $\Pi(Q)$ of $Q$ is used to give a categorification of the cluster algebra of the coordinate algebra $\mathbb{C}[\mathcal{N}]$ of the maximal unipotent subgroup $\mathcal{N}$ in the corresponding semisimple algebraic group. 
In the course of this categorification, the authors give a correspondence between the maximal rigid $\Pi(Q)$-modules and the clusters of the cluster algebra of $\mathbb{C}[\mathcal{N}]$.
It is known that for all but finitely many cases $\mathbb{C}[\mathcal{N}]$ is an infinite type cluster algebra, implying that $\Pi(Q)$ has infinitely many maximal rigid modules. However, as shown in \cite{Mi}, these preprojective algebras $\Pi(Q)$ are always $\tau$-rigid finite.

\begin{proof}
Since $\mathtt{i} \tau \trig(\Lambda) \subseteq  \irigid(\Lambda)$ always holds for any algebra, we only need to show $\irigid (\Lambda) \setminus \mathtt{i} \tau \trig(\Lambda)$ is a finite set.
Suppose $\Lambda$ is min-rep-infinite (otherwise the assertion is evident). 
Since the set of partial tilting modules is the same as faithful $\tau$-rigid modules and because, by Lemma \ref{Almost all}, almost all indecomposable modules are faithful, then almost every element of $\irigid (\Lambda)$ belongs to $\mathtt{i} \tau \trig (\Lambda)$. This finishes the proof.
\end{proof}

As a consequence of the proposition above, we obtain the following equivalent conditions on mild special biserial algebras, which are of the same nature as those given by Theorem \ref{tau-finiteness}.

\begin{theorem}\label{brick-rigid min-rep special biserial}
If $\Lambda$ is a mild special biserial algebra, the following are equivalent:
\begin{enumerate}
\item $\Brick(\Lambda)$ is finite.
\item $\Lambda$ is rigid finite.
\item $\Lambda$ is tilting finite.
\end{enumerate}
\end{theorem}

\begin{proof}
The equivalence of $(1)$ and $(2)$ is an immediate consequence of Theorem \ref{tau-finiteness} and Proposition \ref{tau-finite vs tilting finite}.
Moreover, $(2)$ always implies $(3)$. To show that every tilting finite mild algebra $\Lambda$ is brick finite, we in fact show a stronger result: that almost every $\tau$-tilting $\Lambda$-module is tilting.

To see this, consider an arbitrary $\tau$-tilting module $M=M_1 \oplus \cdots \oplus M_n$ over $\Lambda$, where each $M_i$ (for $1 \leq i\leq n$) is an indecomposable $\tau$-rigid module.
By Lemma \ref{Almost all}, almost all such $M_i$ are faithful, which implies $M$ is almost always faithful. Thus, by Proposition \ref{faithful tau-rigid}, $M$ is almost always a tilting module. 
This implies that $\Lambda$ is tilting finite if and only if it is $\tau$-tilting finite, which, by Theorem \ref{tau-finiteness}, results in the desired conclusion. 
\end{proof}

As mentioned above, $\tau \trig(\Lambda) \subseteq \rigid (\Lambda)$ and $\tilt(\Lambda) \subseteq \tau \ttilt(\Lambda)$ always hold, for any algebra $\Lambda$.
Furthermore, these two containment are related via the crucial equivalence on the level of $\tau$-tilting theory, asserting that $\tau\trig(\Lambda)$ is finite if and only of $\tau\ttilt(\Lambda)$ is finite.
On the other hand, on the level of tilting theory, although finiteness of $\rigid(\Lambda)$ implies that $\tilt(\Lambda)$ is finite, finiteness of $\tilt(\Lambda)$ does not in general guarantee finiteness of $\rigid(\Lambda)$. The preceding theorem shows that for the min-rep-infinite biserial algebras, in fact all of the four above-mentioned sets are finite if and only if one of them is finite.

\section{Bound Quivers of Mild Special Biserial Algebras}\label{Section:Bound Quivers of Mild Special Biserial Algebras}

As explained in the preceding section, to ascertain whether a special biserial algebra is $\tau$-tilting infinite, we employ our reduction to minimal representation-infinite special biserial algebras and use Theorem \ref{tau-finiteness} to translate our problem to brick infiniteness of certain string algebras. To determine which algebras admit infinitely many bricks, a good knowledge of their bound quivers is crucial. Such an elegant description, for the min-rep-infinite special biserial algebras without a node, has recently appeared in \cite{R2}. From the aforementioned work, we only recall the tools needed to prove our results and for details we refer to \cite{R2}. Unless specified otherwise, in this section $\Lambda=kQ/I$ denotes a special biserial algebra.
\vskip 0.15cm
\textbf{Barification.}
Let $v_1=\delta^{\epsilon_{m+1}} \alpha^{\epsilon_m}_m \cdots \alpha^{\epsilon_1}_1 \alpha^{\epsilon_0} $ and $v_2=\gamma^{\epsilon'_{m+1}} \beta^{\epsilon'_m}_m \cdots \beta^{\epsilon'_1}_1 \beta^{\epsilon'_0}$ be a pair of disjoint strings in $(Q,I)$ which do not revisit a vertex that they pass through (i.e, they share no substring and except possibly for the very start and end, every vertex they visit appears once). 
We further assume that none of the arrows in $v_1$ and $v_2$ are involved in a relation (except for possibly $\alpha$, $\beta$, $\gamma$ and $\delta$). 
Suppose $\{ x_0, x_1,\cdots, x_{m+1}, x_{m+2}\}$ and $\{ x'_0, x'_1, \cdots, x'_{m+1}, x'_{m+2} \}$ denote the ordered sets of vertices in $Q_0$, respectively visited by $v_1$ and $v_2$.
Let $\epsilon_i=\epsilon'_i$, for all $1\leq i\leq m$, and assume $m$ is maximal with respect to this property, meaning that it determines the longest substrings of $v_1$ and $v_2$ which are sign-compatible (i.e, chopping off the first and last arrows of $v_1$ and $v_2$ give two isomorphic copies of $\mathbb{A}_{m+1}$).
Then, by $Q(v_1,v_2)$ we denote the quiver obtained from $Q$ via the following steps:
\begin{enumerate}
    \item For every $1\leq i\leq m+1$, identify the vetices $x_i$ and $x'_i$. The new vertex is denoted by $z_i$.
    \item For every $1\leq i\leq m$, identify the arrows $\alpha_i$ and $\beta_i$. The new arrow is denoted by $\theta_i$.
\end{enumerate}

The \emph{bar} $\mathfrak{b}:=\theta_m \cdots \theta_1$ is a copy of $\mathbb{A}_{m+1}$ inside $Q(v_1,v_2)$ and $l(\mathfrak{b})=m+1 > 0$. 
Moreover, the maximality assumption on the length of $v_1$ and $v_2$ implies  $\epsilon_{m+1}= - \epsilon'_{m+1}$ and $\epsilon_{0}=-\epsilon'_{0}$. Without loss of generality, we can assume $\epsilon_{0}=1$. Then, $\beta \alpha$ is a path of length $2$ in $Q(v_1,v_2)$ which passes through $z_0$.
If we further assume $\epsilon_{m+1}=1$, then \emph{barification} of $(Q,I)$ along $v_1$ and $v_2$ is defined as the bound quiver obtained by imposing the relations $\beta \alpha$ and $\delta \gamma$, as well as the generators of $I$, on the quiver $Q(v_1,v_2)$.
Consequently, we have the algebra $kQ(v_1,v_2)/\langle \{\beta \alpha, \delta \gamma\} \cup I \rangle$, which is again a special biserial algebra. Note that if $\epsilon_{m+1}=-1$, in the barification we only replace $\delta \gamma$ by $\gamma \delta$ in the generating set of the new ideal.

\subsection{Barbell algebra}\label{subsection:Barbell Algebra}
An important class of special biserial algebras is obtained via a certain sort of barification that we review here. 
 
Let $Q=\widetilde{\mathbb{A}}_n$ (for some $n \in \mathbb{Z}_{\geq 0}$) and $u$, $u'$, $v$ and $w$ in $Q$ be substrings of positive length such that
\begin{enumerate}
    \item $v$ and $w$ are isomorphic copies of $\mathbb{A}_d$, for some $1 < d \leq n/2$.
    
    \item $\widetilde{\mathbb{A}}_n=u'vuw^{-1}$.
    
    \item $u=\alpha \beta^{\epsilon'_k}_k \cdots \beta^{\epsilon'_1}_1 \beta $ and 
    $u'= \gamma \delta^{\epsilon_m}_m \cdots \delta^{\epsilon_1}_1 \delta$, 
    where 
    $\alpha$, $\beta$, $\gamma$, $\delta$, $\delta_i$ and $\beta_j$  
    are in $Q_1$, 
    for $1\leq i \leq m$ and $1 \leq j \leq k$. 
    
\end{enumerate}

Provided that the above conditions hold, suppose $Q(\mathtt{v},\mathtt{w})$ is obtained from barification of $\mathtt{v}:=\delta v \alpha$ and $\mathtt{w}:= \beta w^{-1} \gamma $. Then, $kQ(\mathtt{v},\mathtt{w})/\langle \beta \alpha, \delta \gamma \rangle$ is called the \emph{barbell} algebra.

We draw the attention of the reader that we used the term ``generalized barbell" algebras for those introduced in the introduction whose bound quivers is slightly more general than the ``barbell" algebras. In fact, a generalized barbell quiver is barbell if and only if the bar is of positive length, while it may or may not be serial (i.e, linearly oriented).
Moreover, note that every generalized barbell algebra contains infinitely many non-isomorphic bands, and therefore it is rep-infinite.
The following lemma gives a handy criterion to verify if a barbell algebra is actually min-rep-infinite.

\begin{lemma}[\cite{R2} Proposition 5.1]\label{min-rep-infiniteinite barbell algebras}
A barbell algebra is minimal representation-infinite if and only if the bar is not serial.
\end{lemma}

The previous lemma obviously implies that every generalized barbell algebra which is min-rep-infinite has a bar of length strictly greater than one.
Before we introduce the last family, let us illustrate some of the preceding definitions and notions in the following example.

\begin{example}\label{Example Barbell}

Let $Q:= \widetilde{\mathbb{A}}_9$ be the quiver given in Figure \ref{fig:barification example}.
If $v_1= \delta \alpha_2 \alpha^{-1}_1 \alpha$ and $v_2=\gamma^{-1} \beta_2 \beta^{-1}_1 \beta^{-1}$, then $\{0,1,2,3,4 \}$ and $\{ 9,8,7,6,5\}$ are respectively the ordered sets of vertices visited by these two disjoint strings in $Q=\widetilde{\mathbb{A}}_9$. 
In this case, $\alpha_2 \alpha^{-1}_1$ and $\beta_2 \beta^{-1}_1$ are respectively the longest substrings in $v_1$ and $v_2$ which are sign compatible (i.e, they are isomorphic copies of $\mathbb{A}_3$). Identifying $v_1$ and $v_2$, as well as the ordered sets of vertices $\{1,2,3\}$ with $\{8,7,6 \}$, results in the quiver $Q(v_1,v_2)$ with seven vertices, where the new vertices created via identification are depicted with circles and the new arrows are labelled by $\theta_1$ and $\theta_2$. In this case, the bar is $\mathfrak{b}= \theta^{-1}_1 \theta_2$, which is not serial, thus the barbell algebra $\Lambda_{\mathfrak{b}}:= kQ(v_1,v_2)/\langle \beta \alpha, \delta \gamma \rangle$ is min-rep-infinite, where the new relations are shown by the dotted lines.

In the bound quiver of $kQ(v_1,v_2)/\langle \beta \alpha, \delta \gamma \rangle$, we can again barify the strings $w_1=\alpha \epsilon^{-1}$ and $w_2=\mu^{-1} \delta$ along the vertices $0$ and $4$, which results in a quiver with six vertices, where the new bar $\mathfrak{b}'$ is the vertex $z$.
If $Q(w_1,w_2)$ denotes the new quiver, the associated special biserial algebra is given by
$$\Lambda_{\mathfrak{b,b'}}:= kQ(w_1,w_2)/\langle \{\beta \alpha, \delta \gamma\} \cup \{\epsilon \mu, \alpha \delta\} \rangle.$$

We can barify the bound quiver of $\Lambda_{\mathfrak{b,b'}}$ once again, by identifying the vertices $5$ and $9$, and further create a new special biserial algebra (illustrated in the rightmost image). However, since all but one vertex of the new quiver is involved in a relation, we cannot barify further.

Note that, in the above copy of $\widetilde{\mathbb{A}}_9$, it is possible to barify $v= \delta (\alpha_2 \alpha^{-1}_1 \alpha \epsilon^{-1}) \beta$ and $w= \beta(\beta_1 \beta^{-1}_2 \gamma \mu^{-1}) \delta$ to obtain a different barbell algebra, while the substrings of $v$ and $w$ which are identified by this choice of barification are the same copy of $\mathbb{A}_5$, specified by the parenthesis. This, in particular, shows that a copy of $\widetilde{\mathbb{A}}_n$ may create different barbell algebras and also indicates that the strings $u$ and $u'$ in a (generalized) barbell algebra can be of length one. In the new case, the bar will be again nonserial and thus the barbell algebra $kQ(v,v^{-1})/\langle \beta^{2}, \delta^{2}  \rangle$ is also min-rep-infinite, where, as before, we identify $w$ and $v$ in the notation $Q(v,w^{-1})$.

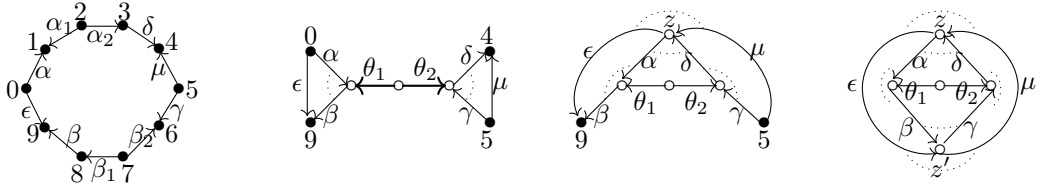
\begin{figure}
\begin{center}
    
\begin{tikzpicture} [thin, scale=0.9]

    
    \node at (0.5,2.45) {$\bullet$}; \node at (0.3,2.45) {$0$};
\draw [->] (0.5,2.5) --(0.75,3);
    \node at (0.78,3.02) {$\bullet$};\node at (0.6,3.15) {$1$};
        \node at (0.75,2.7) {$\alpha$};
\draw [<-] (0.83,3.09) --(1.3,3.37);
    \node at (1.32,3.38) {$\bullet$};\node at (1.3,3.6) {$2$};
        \node at (1,3.38) {$\alpha_1$};
\draw [->] (1.34,3.38) --(1.9,3.38);
    \node at (1.93,3.38) {$\bullet$};\node at (1.95,3.6) {$3$};
        \node at (1.6,3.2) {$\alpha_2$};
\draw [->] (1.99,3.38) --(2.42,3.09);
    \node at (2.46,3.04) {$\bullet$}; \node at (2.65,3.15) {$4$};
        \node at (2.3,3.38) {$\delta$};
\draw [<-] (2.47,3.01) --(2.73,2.53);
        \node at (2.45,2.7) {$\mu$};

\draw [->] (0.5,2.4) --(0.75,1.9);
    \node at (0.77,1.88) {$\bullet$}; \node at (0.6,1.8) {$9$}; 
        \node at (0.5,2.1) {$\epsilon$};
\draw [<-] (0.83,1.85) --(1.3,1.45);
    \node at (1.32,1.44) {$\bullet$};\node at (1.25,1.2) {$8$};
        \node at (1.2,1.75) {$\beta$};
\draw [<-] (1.37,1.45) --(1.9,1.45);
    \node at (1.93,1.45) {$\bullet$}; \node at (2,1.2) {$7$}; 
            \node at (1.65,1.25){$\beta_1$};
\draw [->] (1.99,1.45) --(2.4,1.87);
    \node at (2.46,1.91) {$\bullet$};\node at (2.65,1.8) {$6$};
        \node at (2.2,1.75) {$\beta_2$};
\draw [<-] (2.47,1.98) --(2.73,2.4);
         \node at (2.73,2.1) {$\gamma$};

\node at (2.75,2.45) {$\bullet$}; \node at (2.95,2.45) {$5$};


 \node at (4.7,3) {$\bullet$}; \node at (4.7,3.25) {$0$}; \node at (5,2.95) {$\alpha$};
 \draw [->] (4.75,3) --(5.25,2.5); \node at (5.3,2.5) {$\circ$};
 \draw [->] (5.25,2.5) --(4.75,2); \node at (4.7,1.95) {$\bullet$}; \node at (4.7,1.7) {$9$};\node at (5,2.05) {$\beta$};
 \draw [->] (4.65,3) --(4.65,2.05);\node at (4.5,2.5) {$\epsilon$};
           \draw [dotted] (5.05,2.3) to [bend left=50](5.05,2.75);

\draw [thick,<-] (5.35,2.5) --(5.95,2.5); \node at (6,2.5) {$\circ$};
\node at (5.65,2.75) {$\theta_1$};

\draw [thick,->] (6.05,2.5) --(6.7,2.5); \node at (6.75,2.5) {$\circ$};
\node at (6.4,2.75) {$\theta_2$};

\draw [->] (6.8,2.5) --(7.3,3); \node at (7.35,3) {$\bullet$}; \node at (7.3,3.25) {$4$};
\draw [->] (7.38,2) --(7.38,2.95); \node at (7.35,1.95) {$\bullet$}; \node at (7.3,1.7) {$5$};\node at (7.5,2.5) {$\mu$};\node at (7,2.95) {$\delta$};
\draw [->] (7.35,1.95) --(6.8,2.45); \node at (7,2.05) {$\gamma$};

             \draw [dotted] (7,2.3) to [bend right=50](7,2.75);


\draw [->] (9.9,3.28) to [bend right=70](8.65,2.05);
\node at (8.8,3) {$\epsilon$};

 \draw [->] (9.95,3.24) --(9.3,2.6); \node at (9.3,2.5) {$\circ$};
 \node at (9.7,2.85) {$\alpha$};
 
 \draw [->] (9.25,2.5) --(8.75,2); \node at (8.7,1.95) {$\bullet$}; \node at (8.7,1.7) {$9$};\node at (9,2.05) {$\beta$};

           \draw [dotted] (9.05,2.3) to [bend left=50](9.5,2.8);

\draw [<-] (9.35,2.5) --(9.95,2.5); \node at (10,2.5) {$\circ$};
\node at (9.65,2.25) {$\theta_1$};

\draw [->] (10.05,2.5) --(10.7,2.5); \node at (10.75,2.5) {$\circ$};
\node at (10.4,2.25) {$\theta_2$};

\draw [->] (10.7,2.55) --(10.05,3.25);
\node at (10.25,2.85) {$\delta$};

\draw [->] (11.45,2) to [bend right=70](10.1,3.28);
\node at (11.3,3) {$\mu$};
\node at (11.4,1.95) {$\bullet$}; \node at (11.4,1.7) {$5$};

\draw [->] (11.35,1.95) --(10.8,2.45); \node at (11,2.05) {$\gamma$};
         \draw [dotted] (11,2.3) to [bend right=50](10.5,2.75);

         \draw [dotted] (10.5,3.35) to [bend right=60](9.5,3.35);
         
         \draw [dotted] (10.25,3) to [bend left=10](9.75,3);

\node at (10,3.25) {$\circ$}; \node at (10,3.45) {$z$};


\draw [->] (13.9,3.28) to [bend right=50](12.85,2.5) to [bend right=50](13.95,1.5);
\node at (12.7,2.5) {$\epsilon$};

 \draw [->] (13.95,3.24) --(13.3,2.6); \node at (13.3,2.5) {$\circ$};
 \node at (13.7,2.85) {$\alpha$};
 
 \draw [->] (13.25,2.45) --(13.95,1.65); 
 \node at (13.5,1.85) {$\beta$};

           \draw [dotted] (13.25,2.75) to [bend right=80](13.3,2.25);

\draw [<-] (13.35,2.5) --(13.95,2.5); \node at (14,2.5) {$\circ$};
\node at (13.65,2.4) {$\theta_1$};

\draw [->] (14.05,2.5) --(14.7,2.5); \node at (14.75,2.5) {$\circ$};
\node at (14.4,2.4) {$\theta_2$};

\draw [->] (14.7,2.55) --(14.05,3.25);
\node at (14.25,2.85) {$\delta$};

\draw [->] (13.95,1.5) to [bend right=50] (15.15,2.5) to [bend right=50] (14.1,3.28);
\node at (15.3,2.5) {$\mu$};

\draw [->] (14.05,1.65) --(14.8,2.45); \node at (14.45,1.85) {$\gamma$};
         \draw [dotted] (14.75,2.7) to [bend left=80](14.75,2.2);

         \draw [dotted] (14.5,3.35) to [bend right=60](13.5,3.35);
         
         \draw [dotted] (14.25,3) to [bend left=10](13.75,3);
         
         \draw [dotted] (14.5,1.5) to [bend left=60](13.5,1.5);
         
         \draw [dotted] (14.25,1.9) to [bend left=10](13.75,1.9);

\node at (14,3.25) {$\circ$}; \node at (14,3.45) {$z$};

\node at (14,1.55) {$\circ$}; \node at (14,1.3) {$z'$};

\end{tikzpicture}
\end{center}    
   \caption{Barification of $\widetilde{\mathbb{A}}_9$ via bars $\mathfrak{b}= \theta^{-1}_1 \theta_2$, $\mathfrak{b}'=z$ and $\mathfrak{b}''=z'$.}
    \label{fig:barification example}
\end{figure}
\end{example}

\subsection{Wind wheel algebra}\label{subsection:construction of wind wheel algebras}
The other family of algebras that we need to recall are very similar to the barbell algebras. However, the difference is that we only barify along some linearly oriented copies of $\mathbb{A}_d$'s, and may also barify multiple pairs of such strings at the same time. We impose some additional relations to address the serial bars that are created in the course of barifications of $\widetilde{\mathbb{A}}_n$. Before we describe the construction via an algorithm, let us fix some terminology which comes handy in the remainder of the paper.

\begin{definition}
Let $kQ/I$ be a string algebra and $u$ and $w$ be strings in $(Q,I)$.
If $\gamma \in Q_1$ is an arrow, \emph{$w$ supports $\gamma$} provided $\gamma$ or $\gamma^{-1}$ occurs in $w$. Let $u=\alpha^{\epsilon_m}_m \cdots \alpha^{\epsilon_1}_2 \alpha^{\epsilon_1}_1$ and $v=\beta^{\eta_n}_n \cdots \beta^{\eta_2}_2 \beta^{\eta_1}_1$, for some $\alpha_i, \beta_j \in Q_1$ and $\epsilon_i, \eta_j \in \{\pm 1\}$, for every $1 \leq i \leq m$ and $1 \leq j \leq n$. 
    We say \emph{tail of $u$ collides with head of $w$} if $e(u)= s(v)$ and $\epsilon_m= -\eta_1$, but $\alpha^{\epsilon_m}_m \neq \beta^{-\eta_1}_1$. Then, the ordered pair $(u,v)$ is called \emph{tail-head collision}.

The configuration of a tail-head collision pair $(u,v)$ at the vertex $x=e(u)= s(v)$ could be illustrated as follows:
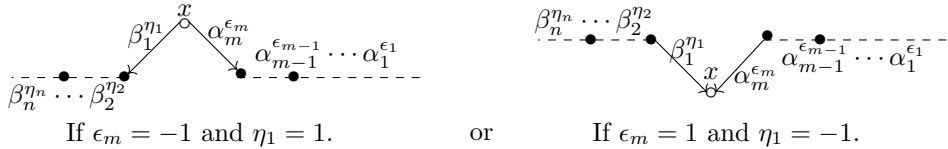
\begin{center}
    \begin{tikzpicture}
    

 \draw [->] (1.35,0.75) --(2,0.1);
    \node at (1.9,0.6) {$\alpha^{\epsilon_{m}}_{m}$};
 \draw [<-] (0.55,0.05) --(1.25,0.75);

     \node at (2.05,0.05) {$\bullet$};
    \node at (1.3,0.71) {$\circ$};

 \draw [dashed] (0.5,0) --(-1,0);
    \node at (0.5,0) {$\bullet $};
     \node at (-0.3,0) {$\bullet $};

 \draw [dashed] (2,0) --(4.5,0);
 \node at (2.75,0) {$\bullet$};
 \node at (3.2,0.3) {$\alpha^{\epsilon_{m-1}}_{m-1} \cdots \alpha^{\epsilon_1}_1$};
 
\node at (0.8,0.55) {$\beta^{\eta_1}_1$};
 \node at (-0.25,-0.25) {$\beta^{\eta_n}_n \cdots \beta^{\eta_2}_2$};

  \node at (1.3,0.9) {$x$};
  
  \node at (1.5,-0.75) {If $\epsilon_m=-1$ and $\eta_1=1$.};
  
 
 \draw [<-] (8.35,-0.2) --(9,0.5);
    \node at (8.9,0) {$\alpha^{\epsilon_{m}}_{m}$};
    \node at (10.2,0.25) {$\alpha^{\epsilon_{m-1}}_{m-1} \cdots \alpha^{\epsilon_1}_1$};
 \draw [->] (7.55,0.5) --(8.25,-0.2);

     \node at (9.05,0.55) {$\bullet$};
    \node at (8.3,-0.21) {$\circ$};

 \draw [dashed] (7.5,0.5) --(6,0.5);
    \node at (7.5,0.5) {$\bullet $};
     \node at (6.7,0.5) {$\bullet $};

 \draw [dashed] (9,0.5) --(11.5,0.5);
 \node at (9.75,0.5) {$\bullet$};
 
\node at (8,0.35) {$\beta^{\eta_1}_1$};
 \node at (6.75,0.75) {$\beta^{\eta_n}_n \cdots \beta^{\eta_2}_2$};

  \node at (8.3,0.05) {$x$};
  
  \node at (8.5,-0.75) {If $\epsilon_m=1$ and $\eta_1=-1$.};
 
\node at (5.25,-0.75) {or};
\end{tikzpicture}
\end{center}
where the length and orientation of the dashed segments could be anything and they may also share some arrows or vertices.
\end{definition}

Now we are ready to recall the construction of the wind wheel algebras.
As in the case of barbell algebras, start with $Q:= \widetilde{\mathbb{A}}_n$, for some $n \in \mathbb{Z}_{>0}$. Suppose for each $1 \leq i \leq 2k$, $u_i$ and $v_i$ are substrings of $Q$ of positive lengths with $\widetilde{\mathbb{A}}_n= v_{2k}u_{2k} \cdots v_1u_1$, while $\sigma$ is an involution on $\{1,2, \cdots,2k\}$ such that $\sigma(i) \neq i$ and $v^{-1}_i$ and $v_{\sigma(i)}$ are identical copies of some linearly oriented $\mathbb{A}_{d_i}$, for some $1\leq d_i \leq n$. Namely, $v^{-1}_i$ and $v_{\sigma(i)}$ are the same after relabelling the arrows of $v_{\sigma(i)}$ by those of $v_i$, as we assume henceforth.
Moreover, suppose that the following conditions hold:
\begin{enumerate}
    \item Every $v_i$ is serial with $v^{-1}_i=v_{\sigma(i)}$.
    \item Every arrow $\gamma \in Q_1$ supported by a $u_i$ is supported exactly once, whereas if $\theta \in Q_1$ is supported by some $v_i$, then $\theta^{\pm}$ is supported (once by $v_i$ and the second time by $v_{\sigma(i)}$).
    \item For all $1\leq i \leq 2k$, where $u_{2k+1}$ is interpreted as $u_1$, every pair $(v_i,u_{i})$ is a tail-head collision, whereas $(u_{i+1},v_i)$ is not.
\end{enumerate}

Provided the above properties hold, suppose
$\mathtt{v}_i$ (respectively $\mathtt{v}^{-1}_{\sigma(i)}$) denotes the substrings of $\widetilde{\mathbb{A}}_n$ which is obtained from $v_i$ (respectively $v^{-1}_{\sigma(i)}$) by extending in $\widetilde{\mathbb{A}}_n$ from each side by one arrow.
Then, $Q_{\mathfrak{w}}$ denotes the quiver obtained by barification of all pairs of $\mathtt{v}_i$ and $\mathtt{v}^{-1}_{\sigma(i)}$, for every $1 \leq i \leq 2k$.
In order to create a special biserial quotient of the path algebra $kQ_{\mathfrak{w}}$, assume $v_i$ is a direct string and $\sigma(i)=j$, then fix the expression $u_i=\alpha^{\epsilon_{m_i}}_{m_i}\cdots\alpha^{\epsilon_{1_i}}_{1_i}$ (for some $m_i \in \mathbb{Z}_{>0}$).
Then, we impose the following monomial relations:

\begin{itemize}
    \item C-Relation: The quadratic relations $\alpha_{1_{j+1}} \alpha_{m_i}$ and $ \alpha^{-1}_{m_j}\alpha^{-1}_{1_{i+1}}$.
    
    \item B-Relation: The path $\alpha^{-1}_{m_{j}} v_i \alpha_{m_i}$.
\end{itemize}
The \emph{wind wheel algebra} is defined as 
$$\Lambda_{\mathfrak{w}}:= kQ_{\mathfrak{w}}/ I_{\mathfrak{w}},$$
where $I_{\mathfrak{w}}:=\langle \{ \alpha_{1_{j+1}} \alpha_{m_i}, \alpha^{-1}_{m_j}\alpha^{-1}_{1_{i+1}}, \alpha^{-1}_{m_{j}} v_i \alpha_{m_i} \,|\, 1 \leq i \leq 2k \} \rangle$.

Let us mention a few remarks before we present an example.

\begin{remark}
Note that the C-relations are the same as those in the barbell algebras, added after barification of two identical copies of an $\mathbb{A}_d$ inside $\widetilde{\mathbb{A}}_n$. However, because in the wind wheel algebra we typically barify multiple bars, $2k$-many C-relations are added such that each one kills the pair of consecutive arrows that form paths of length two and pass through a $3$-vertex at one end of each bar $v_i$. We called these relations the C-relations, because any barification of $\widetilde{\mathbb{A}}_n$ gives rise to two cyclic strings (one on each side of the bar), and the corresponding quadratic relation kills a band of type $\widetilde{\mathbb{A}}_{t_i}$ formed by $u_i$'s.
In contrast, the B-relations are those arising from the paths of length $l(v_i)+2$ which begin at the start of the only non-vanishing arrow entering the serial bar $v_i$, go along it and continue until the end of the single non-vanishing arrow leaving $v_i$. Since each barification of $\widetilde{\mathbb{A}}_n$ gives rise to exactly one such path, we annihilate them by such B-relations.
From this, it is obvious that in the generating set of the \emph{wind wheel ideal} $I_{\mathfrak{w}}$, there are $2k$ C-relations and $k$-many B-relations.
\end{remark}

The following example is meant to illustrate the technical construction above.

\begin{example} \label{Example Wind wheel}
Let $Q=\widetilde{\mathbb{A}}_{12}$, with $Q=  \gamma^{-1}_{12}\gamma_{11}\gamma_{10}\gamma_9 \gamma^{-1}_8 \gamma^{-1}_7\gamma^{-1}_6\gamma_5\gamma_4\gamma^{-1}_3\gamma^{-1}_2 \gamma_1 \gamma^{-1}_0$.
In this case, we consider 
$u_1=\gamma^{-1}_3\gamma^{-1}_2 \gamma_1 \gamma^{-1}_0$, $v_1=\gamma_4$, $u_2=\gamma_5$, $v_2=\gamma^{-1}_7\gamma^{-1}_6$, $u_3=\gamma^{-1}_8$, $v_3=\gamma_{10}\gamma_9$, $u_4=\gamma_{11}$ and $v_4=\gamma^{-1}_{12}$. 
The following is the illustration of $\widetilde{\mathbb{A}}_{12}$ where the two vertices labelled by $0$ must be identified. Starting from the right end, every inverse arrow is oriented to the right, whereas direct arrows are towards left. Moreover, the subdivision of $Q$ into $u_i$'s and $v_i$'s is as follows:

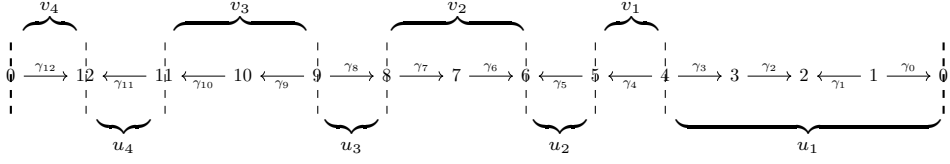
\begin{center}
\begin{tikzpicture}
\node[scale=0.7] (a) at (0,0){
    \begin{tikzcd}
 0 \arrow{r} {\gamma_{12}}
  & 12
  & 11 \arrow{l}{\gamma_{11}}
  & 10 \arrow{l}{\gamma_{10}}
  & 9 \arrow{l}{\gamma_{9}} \arrow{r}{\gamma_{8}}
  & 8 \arrow{r}{\gamma_{7}}
   & 7 \arrow{r}{\gamma_{6}}
   & 6 
   & 5 \arrow{l}{\gamma_{5}}
   & 4 \arrow{l}{\gamma_{4}} \arrow{r}{\gamma_{3}}
   & 3 \arrow{r}{\gamma_{2}}
    & 2 
    & 1 \arrow{l}{\gamma_{1}} \arrow{r}{\gamma_{0}}
    & 0 
\end{tikzcd}};

\draw[thick,dashed] (6.2,-0.5) -- (6.2,0.5);
\draw[dashed] (2.5,-0.5) -- (2.5,0.5);
\draw[dashed] (1.57,-0.5) -- (1.57,0.5);
\draw[dashed] (0.65,-0.5) -- (0.65,0.5);
\draw[dashed] (-1.2,-0.5) -- (-1.2,0.5);
\draw[dashed] (-2.12,-0.5) -- (-2.12,0.5);
\draw[dashed] (-4.15,-0.5) -- (-4.15,0.5);
\draw[dashed] (-5.2,-0.5) -- (-5.2,0.5);
\draw[thick,dashed] (-6.2,-0.5) -- (-6.2,0.5);


\node at (4.4,-0.8) {$\underbrace{\qquad \qquad \qquad \qquad \qquad }_{u_1}$};

\node at (2.05,0.8) {$\overbrace{\qquad }^{v_1}$};
\node at (1.1,-0.8) {$\underbrace{\qquad }_{u_2}$};
\node at (-0.26,0.8) {$\overbrace{\qquad \qquad \quad }^{v_2}$};
\node at (-1.67,-0.8) {$\underbrace{\qquad}_{u_3}$};
\node at (-3.15,0.8) {$\overbrace{\qquad \qquad \quad }^{v_3}$};
\node at (-4.7,-0.8) {$\underbrace{\qquad }_{u_4}$};
\node at (-5.68,0.8) {$\overbrace{\qquad }^{v_4}$};
\end{tikzpicture}
\end{center}

It could be easily verified that for any $1 \leq i \leq 4$, the pair $(u_i,v_i)$ is a tail-head collision, whereas $(v_i,u_{i+1})$ is not (when $u_5$ is interpreted as $u_1$). 
Moreover, the involution $\sigma$ on $\{ 1, 2, 3, 4\}$ is given by $\sigma(1)=4$, $\sigma(2)=3$, $\sigma(3)=2$ and $\sigma(4)=1$, implying that $v_1$ and $v_4$, as well as $v_2$ and $v_3$, shall be identified. Hence, after relabelling the arrows of $v_4$ and $v_3$ respectively by those of $v_1$ and $v_2$, we observe that their arrows are supported in both directions, whereas those in $u_i$'s are supported exactly once.
Via barification of the substrings of $\widetilde{\mathbb{A}}_{12}$ determined by the above pairs we obtain the following quiver, where $\theta^{p}_{q}$ denotes the result of identification of $\alpha_p$ and $\alpha_q$, and the new vertices are depicted by circles, in contrast with the original vertices that are solid. In addition to the C-relations illustrated by the dotted lines, the B-relations are given by the paths $\gamma_5 \theta^{12}_{4} \gamma_0$ and $\gamma_{11} \theta^{10}_{6} \theta^{9}_{7} \gamma_8$, illustrated by the dashed segments.

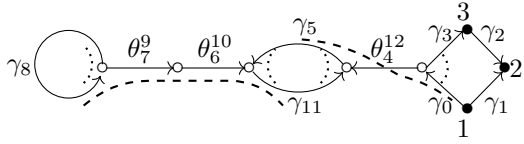
\begin{center}
    \begin{tikzpicture}

\node at (6.75,2.5) {$\circ$};
\draw [->] (7.38,2) --(7.82,2.45); 
\node at (7.85,2.5) {$\bullet$};
\draw [->] (7.38,3) --(7.82,2.55);
\node at (7.35,1.95) {$\bullet$};
\node at (8,2.5) {$2$};
\node at (7.3,1.7) {$1$};
\node at (7.3,3.25) {$3$};
\node at (7,2.95) {$\gamma_3$};
\node at (7.7,2.95) {$\gamma_2$};
\draw [->] (7.35,1.95) --(6.8,2.45); 
\node at (7,2) {$\gamma_0$};
\node at (7.75,2) {$\gamma_1$};
 \draw [thick,dotted] (7,2.3) to [bend right=50](7,2.7);

\draw [<-] (5.8,2.5) --(6.7,2.5); 
\node at (6.3,2.75) {$\theta^{12}_{4}$};
 \node at (5.75,2.5) {$\circ$};
 

 \draw [->] (5.7,2.55) to [bend right=50](4.5,2.55);
         \node at (4.45,2.5) {$\circ$};
 \draw [<-] (5.7,2.45) to [bend left=50](4.5,2.45);
            \node at (5.2,3.05) {$\gamma_5$};
 
\draw [->] (6.8,2.5) --(7.3,3); \node at (7.35,3) {$\bullet$}; 
            \node at (5.2,2) {$\gamma_{11}$};

\draw [thick,dotted] (5.4,2.3) to [bend right=50](5.4,2.75);
 \draw [thick,dotted] (4.8,2.3) to [bend left=50](4.8,2.75);
     

   \draw [<-] (4.4,2.5) --(3.55,2.5);   
        \node at (3.5,2.5) {$\circ$};
   \draw [<-] (3.45,2.5) --(2.55,2.5); 

\node at (4,2.75) {$\theta^{10}_{6}$};
\node at (3,2.75) {$\theta^{9}_{7}$};


\node at (2.5, 2.5) {$\circ$};   
    \draw[->] (2.5,2.55) arc (0:340:0.45cm);
\node at (1.4,2.5) {$\gamma_8$};    
\draw [thick, dotted] (2.25,2.3) to [bend right=50](2.25,2.8);


\draw [thick,dashed] (2.25,2) to [bend left=20] (3.2,2.33)-- (3.2,2.33) -- (4.25,2.33) -- (4.25,2.33) to [bend left=20] (4.9,2); 

\draw [thick,dashed] (5.15,2.88) to [bend left=10] (6,2.65)-- (6.45,2.4) to [bend left=10] (7.15,2.07); 

    \end{tikzpicture}
\end{center}

\end{example}

\subsection{Resolution of nodes}
In an arbitrary algebra $\Lambda=kQ/I$, a vertex $x$ is a \emph{node} if it is neither a sink nor a source, and for any arrow $\alpha$ entering $x$ and any arrow $\beta$ leaving $x$, we have $\beta \alpha \in I$. We define
$\node(\Lambda):=\{x \in Q_0 \,|\, x \, \text{is a node} \}$.

Suppose $x \in \node(\Lambda)$ and $\{\alpha_1, \cdots, \alpha_m\}$ and $\{\beta_1, \cdots, \beta_n\}$ are respectively the set of arrows entering and leaving $x$. By \emph{resolving} a node $x$, we replace it by two new vertices $x_+$ and $x_-$ such that $x_+$ is a sink (which receives $\alpha_i$, for $1\leq i \leq m$) and $x_-$ is a source (from which every $\beta_j$ leaves, for $1\leq j \leq n$). Let $\nn(\Lambda)$ be the algebra obtained by resolving all nodes in $\Lambda$. 
The bound quiver $(Q',I')$ of $\nn(\Lambda)$ is such that $Q'$ is given by node-resolving process at every node of $Q$ and the set of generators for $I'$ consists of all those generators of $I$ whose terms do not include any path passing through a node in $(Q,I)$. In particular, if $\Lambda$ is min-rep-infinite special biserial, by Lemma \ref{Monomial ideal}, we can assume all generators of $I$ are monomial, from which only those paths that pass through no node in $\Lambda$ generate $I'$.

The following fundamental result by Ringel \cite{R2} demonstrates the close connection between the module categories of $\Lambda$ and $\nn(\Lambda)$. 
Recall that for an algebra $\Lambda$, the stable module category of $\Lambda$ is the additive category $\underline{\modu} \Lambda$ given as follows: The objects of $\underline{\modu}\Lambda$ and $\modu \Lambda$ are the same and for two $\Lambda$-modules $X$ and $Y$, the space of morphism $\Hom_{\underline{\modu} \Lambda}(X,Y)$ is a quotient space of $\Hom_{\Lambda}(X,Y)$ by the subspace generated by all those morphisms in $\Hom_{\Lambda}(X,Y)$ which factor through a projective. Namely, the morphisms of $\underline{\modu} \Lambda$ are the classes of the morphisms in $\modu \Lambda$ under the equivalence relation of differing by a map which factors through a projective $\Lambda$-module. If two algebras have equivalent stable module categories, they are called \emph{stably equivalent}.

\begin{theorem}[{\cite[Sections 7 and 8]{R1}\label{node-free-algebra}}]
For an algebra $\Lambda$, we have the following:
\begin{enumerate}
\item $\Lambda$ and $\nn(\Lambda)$ are stably equivalent. Moreover, $\Lambda$ is minimal representation-infinite if and only if $\nn (\Lambda)$ is so.
\item If $\Lambda$ is a special biserial and minimal representation-infinite, every $4$-vertex is a node.
\end{enumerate}
\end{theorem}

The first part of the above theorem implies that, up to stable equivalence, the study of min-rep-infinite algebras could be done via the node-free bound quivers. In particular, for the family $\mathfrak{F}_{\sB}$, both parts could be employed together to investigate min-rep-infinite special biserial algebras via the bound quivers $(Q,I)$ such that $kQ/I$ is a finite dimensional algebra and each vertex is of degree $2$ or $3$, respectively with exactly zero and one relation. 
The full classification of the node-free min-rep-infinite special biserial algebras is as follows.

\begin{theorem}[Theorem 1.2 \cite{R2}]\label{Ringel's Classification Thm}
Let $\Lambda=kQ/I$ be a minimal representation-infinite special biserial algebra and $\node (\Lambda)=\emptyset$. Then, $\Lambda$ is either a cycle algebra, or a barbell algebra with no serial bar or a wind wheel algebra.
\end{theorem}

The Theorems \ref{node-free-algebra} and \ref{Ringel's Classification Thm} together result in the following corollary.

\begin{corollary}\label{Nod-free Sp.biserial}
If $\Lambda$ is a special biserial algebra, it is minimal representation-infinite if and only if $\nn(\Lambda)$ is a cycle, barbell or wind wheel algebra.
\end{corollary}

Before we finish this section let us remark that although barification could be defined for other type of special biserial algebras, it does not necessarily preserve  representation-finiteness (or $\tau$-tilting finiteness). In particular, if $\Lambda'$ is obtained from $\Lambda$ via barification, it is important to realize that it is not a quotient algebra of $\Lambda$.
The following example illustrates some of these points. 

\begin{example}\label{barification does not preserve rep-finiteness}
Let $Q=\mathbb{A}_8$ with the following orientation:

\begin{center}
\begin{tikzpicture}
\node[scale=1] (a) at (0,0){
    \begin{tikzcd}
 1 
  & 2 \arrow{l} {\alpha} \arrow{r}{\beta}
  & 3 \arrow{r}{\gamma_{1}}
  & 4 
  & 5 \arrow{l}{\delta} \arrow{r}{\gamma_{2}}
  & 6 \arrow{r}{\epsilon}
   & 7 
   & 8 \arrow{l}{\mu}
\end{tikzcd}};
\end{tikzpicture}
\end{center}
Then, barification of $Q$ along the arrows $\gamma_1$ and $\gamma_2$ and further along vertices $2$ and $7$ result in the following bound quivers.

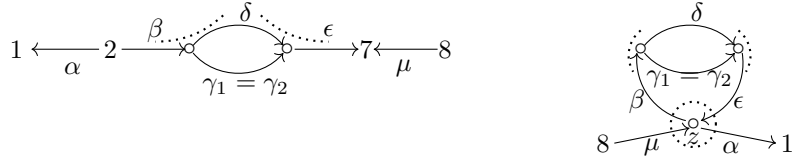
\begin{center}
    \begin{tikzpicture}
 \node at (5.75,2.5) {$\circ$};
\draw [->] (5.85,2.5) --(6.7,2.5); 
\node at (6.3,2.75) {$\epsilon$};
 \node at (6.8,2.5) {$7$};
 \draw [<-] (6.9,2.5) --(7.75,2.5); 
 \node at (7.3,2.25) {$\mu$};
  \node at (7.85,2.5) {$8$};
 \draw [<-] (5.7,2.55) to [bend right=50](4.5,2.55);
         \node at (4.45,2.5) {$\circ$};
 \draw [<-] (5.7,2.45) to [bend left=50](4.5,2.45);
     \node at (5.2,3) {$\delta$};

    \node at (5.2,2) {$\gamma_{1}=\gamma_{2}$};

\draw [thick,dotted] (5.4,2.95) to [bend right=20](6.3,2.6);
 \draw [thick,dotted] (4.9,2.95) to [bend left=20](4,2.6);

   \draw [<-] (4.4,2.5) --(3.55,2.5);   
        \node at (3.4,2.5) {$2$};
   \draw [->] (3.25,2.5) --(2.35,2.5); 
          \node at (2.15,2.5) {$1$};
\node at (4,2.75) {$\beta$};
\node at (2.9,2.25) {$\alpha$};


 \draw [<-] (11.7,2.55) to [bend right=50](10.5,2.55);
         \node at (10.45,2.5) {$\circ$};
 \draw [<-] (11.7,2.45) to [bend left=50](10.5,2.45);
        \node at (11.75,2.5) {$\circ$};
     \node at (11.2,3) {$\delta$};

    \node at (11.1,2.1) {$\gamma_{1}=\gamma_{2}$};

\draw [thick,dotted] (11.75,2.75) to [bend left=40](11.85,2.2);
 \draw [thick,dotted] (10.45,2.75) to [bend right=40](10.35,2.2);


\draw [->] (11.8,2.45) to [bend left=40](11.25,1.55);
    \node at (11.75,1.8) {$\epsilon$};

\draw [<-] (10.4,2.45) to [bend right=40](11.05,1.55);
     \node at (10.4,1.8) {$\beta$};
    
    \node at (11.15,1.5) {$\circ$};     
         \node at (11.15,1.32) {$z$};     
\draw [<-] (12.25,1.25) -- (11.25,1.45);
        \node at (11.65,1.2) {$\alpha$};
       \node at (12.4,1.25) {$1$};
\draw [->] (10.1,1.25) -- (11.1,1.45);
        \node at (10.6,1.2) {$\mu$};
         \node at (9.95,1.25) {$8$};

\draw [dotted,thick] (11.15,1.5) circle (.3cm);
         
    \end{tikzpicture}
\end{center}
where the vertex $z$ in the right bound quiver is a node.

Since there exists a copy of the Kronecker quiver inside each of the above bound quivers, it is immediate that neither of them is rep-finite, nor $\tau$-tilting finite. However, the algebra $kQ$ is obviously representation-finite.
\end{example}

\section{$\tau$-Tilting Finiteness of Node-free Special Biserial Algebras}\label{Section:tau-Tilting Finite Node-free Special Biserial Algebras}

In this section, building upon the explicit description of min-rep-infinite special biserial algebras without nodes recalled in the preceding section, we determine which ones are $\tau$-tilting finite and which ones are not. 
For an acyclic orientation of $\widetilde{\mathbb{A}}_n$ (for some $n\in \mathbb{Z}_{> 0}$), the Auslander-Reiten quiver of $k \widetilde{\mathbb{A}}_n$ is illustrated in Figure \ref{fig:AR component of A-tilde}, which has preprojective and preinjective components. Therefore, by Lemma \ref{Preproj/Postinj brick} the algebra $k\widetilde{\mathbb{A}}_n$ is always $\tau$-tilting infinite. 
Hence, in the rest of this section, unless specified otherwise, we always assume $\Lambda=kQ/I$ is a min-rep-infinite special biserial algebra and $I\neq 0$.

Since our method relies on giving an explicit description of an infinite family of bricks in certain string algebras, we briefly recall the notion of graph maps, which form a concrete basis for the space $\Hom(M(w),M(v))$, for every pair of strings $v$ and $w$ in a string algebra.
Thanks to the investigation of other algebras with similar bound quivers, as those studied by Crawley-Boevey in \cite{CB1} and \cite{CB3}, most of these results apply to a more general setting, but here we restrict to min-rep-infinite special biserial algebras and follow the notation used by Schröer \cite{S}.

\begin{definition}\label{factorization def}
If $\Lambda$ is a string algebra, for $u \in \Str(\Lambda)$, the set of \emph{factorizations of $u$} is defined as 
$$\mathtt{F}(u):=\{(u_3,u_2,u_1) \mid u_3,u_2,u_1 \in \Str(\Lambda) \text{ and } u=u_3u_2u_1\},$$
and for $(u_3,u_2,u_1) \in \mathtt{F}(u)$, we define  $(u_3,u_2,u_1)^{-1}:=(u_1^{-1},u_2^{-1},u_3^{-1})\in \mathtt{F}(u^{-1})$.
We say $(u_3,u_2,u_1) \in \mathtt{F}(u)$ is a \emph{quotient factorization} of $u$ if

\begin{enumerate}[(i)]			

\item $u_1=e_{s(u_2)}$ or $u_1=\gamma^{-1} u_1'$ with $\gamma$ in $Q$;	

\item $u_3=e_{e(u_2)}$ or $u_3=u_3'\theta$ with $\theta$ in $Q$.

\end{enumerate}

Every quotient factorization $(u_3,u_2,u_1)$ induces a quotient map from
$M(u)$ to $M(u_2)$.
Let $\mathtt{F_q}(u)$ denote the set of all quotient factorizations of $u$.
Each element of $\mathtt{F_q}(u)$ can be visualized as follows, where $u_1$ and $u_3$ can be of length zero.

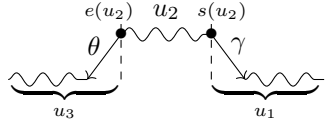
\begin{center}
\begin{tikzpicture}[scale=0.6]
\draw  [-,decorate,decoration=snake] (0,0) --(2,0);
----
\draw [->] (0,0) -- (-0.75,-1);
\node at (-0.6,-0.25) {$\theta$};
\draw [-,decorate,decoration=snake] (-2.5, -1) -- (-0.75,-1);
----
\draw [->] (2,0) -- (2.75,-1);
\node at (2.6,-0.25) {$\gamma$};
\draw [-,decorate,decoration=snake] (2.75,-1) -- (4.5,-1);
----
\node at (1,0.5) {$u_2$};
\node at (2.25,0.5) {$_{s(u_2)}$};
\node at (-0.25,0.5) {$_{e(u_2)}$};
\node at (2,0.0) {$\bullet$};
\node at (0,0.0) {$\bullet$};

\node at (3.2,-1.5) {$\underbrace{\qquad \qquad }_{u_1}$};
\draw[dashed] (2,-1) -- (2,0.25);
\node at (-1.25,-1.5) {$\underbrace{\qquad \qquad }_{u_3}$};
\draw[dashed] (0,-1) -- (0,0.25);

\end{tikzpicture}
\end{center}

As the dual notion, $(u_3,u_2,u_1) \in \mathtt{F}(u)$ is called a \emph{submodule factorization} of $u$ if
\begin{enumerate}[(i)]
\item $u_1=e_{s(u_2)}$ or $u_1=\gamma u_1'$ with $\gamma$ in $Q$;
\item $u_3=e_{e(u_2)}$, or $u_3=u_3'\theta^{-1}$ with $\theta$ in $Q$.
\end{enumerate}
$\mathtt{F_s}(u)$ denotes the set of all submodule factorizations of $u$, and any $(u_3,u_2,u_1)$ in $\mathtt{F_s}(u)$ induces an inclusion of $M(u_2)$ into $M(u)$. This inclusion can be visualised by a pair of diagrams dual to the one shown above.

\end{definition}

\begin{definition}\label{admissible_pair}
Let $\Lambda$ be a string algebra and $u, v \in \Str(\Lambda)$. Then, a pair $((u_3,u_2,u_1),(v_3,v_2,v_1)) \in \mathtt{F_q}(u) \times \mathtt{F_s}(v)$ is called \emph{admissible} if $u_2 \sim v_2$ (i.e, $u_2=v_2$ or $u_2=v_2^{-1}$). For $u$ and $v$ in $\Str(\Lambda)$, the collection of all admissible  pairs in $\mathtt{F_q}(u) \times \mathtt{F_s}(v)$ is denoted by $\mathtt{A}(u,v)$.
For every $T=((u_3,u_2,u_1),(v_3,v_2,v_1))$ in $\mathtt{A}(u,v)$ suppose $f_T$ is defined as: composition of the projection
from $M(u)$ to $M(u_2)$, followed by the identification of $M(u_2)$ with $M(v_2)$, followed by the inclusion of 
$M(v_2)$ into $M(v)$. The result is a morphism $f_T: M(u) \rightarrow M(v)$ in $\modu \Lambda$, called the {\em graph map} induced by $T$. For $w \in \Str(\Lambda)$, an admissible pair $T$ in $\mathtt{A}(w,w)$ is \emph{trivial} if $f_T$ is the identity morphism $id: M(w) \rightarrow M(w)$.
\end{definition}

In the study of string algebras, the next theorem is crucial.

\begin{theorem}\cite{CB1}\label{Thm-graph-maps}
Let $\Lambda$ be a string algebra. If $u$ and $v$ are in $\Str(\Lambda)$, the set of graph maps $\{f_T\mid T\in \mathtt{A}(u,v) \}$ forms a basis for $\Hom_{\Lambda}(M(u), M(v))$.
\end{theorem}

From the previous theorem, the following corollary is immediate.

\begin{corollary}
If $\Lambda$ is a string algebra and $w\in \Str(\Lambda)$, the following are equivalent:

\begin{enumerate}
    \item $T=((u_3,u_2,u_1),(v_3,v_2,v_1))$ is a non-trivial admissible pair in $\mathtt{A}(w,w)$;
    \item There exist two copies of a substring of $w$, say $v_2\sim u_2$, such that $w=u_3u_2u_1$ and $w=v_3v_2v_1$ with $l(u_1)+l(u_3)>0$ and $l(v_1)+l(v_3)>0$, where $u_1$ and $u_3$ (respectively $v_1$ and $v_3$) leave $u_2$ (respectively enter $v_2$).
\end{enumerate}
\end{corollary}

The above corollary states that for a non-trivial endomorphism of a string module $M(w)$, there must exist a proper (non-trivial) substring of $w$ which occurs in $w$ at least twice, once on the top and another time at the bottom of $w$. By the \emph{top} and \emph{bottom} of $w$ we respectively refer to the local configuration of $u_2$ and $v_2$ in an admissible pair $T=((u_3,u_2,u_1),(v_3,v_2,v_1))$ in $\mathtt{A}(w,w)$, illustrated as follows, which gives rise to a (non-identity) graph map in $\End_{\Lambda}(M(w))$:

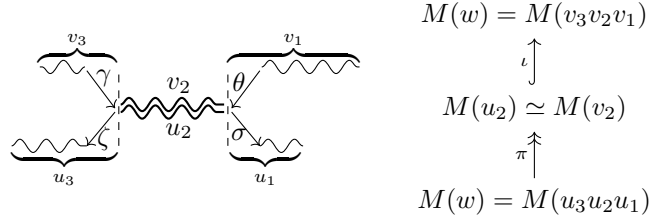
\begin{center}
\begin{tikzpicture}[scale=0.50]
\draw  [-,decorate,decoration=snake][thick] (0.15,0.1) --(2.9,0.1);
\draw  [-,decorate,decoration=snake][thick] (0.15,-0.1) --(2.9,-0.1);
\node at (1.7,0.6) {$v_2$};
\node at (1.7,-0.6) {$u_2$};
------
\draw [<-] (0,0) -- (-0.75,1);
\draw [-,decorate,decoration=snake] (-0.8,1.1) -- (-2,1.1);
\node at (-0.3,0.8) {$\gamma$};
\node at (-1,1.6) {$\overbrace{\qquad \quad }^{v_3}$};
\node at (4.75,1.6) {$\overbrace{\qquad \qquad \quad}^{v_1}$};
-----
\draw [<-] (3.1,0) -- (3.85,1);
\draw [-,decorate,decoration=snake] (3.9,1.1) -- (6.5,1.1);
\node at (3.3,0.7) {$\theta$};

------
\draw [<-] (-0.75,-1) -- (0,-0.1);
\draw [-,decorate,decoration=snake] (-2.75,-1) -- (-0.75,-1);
\node at (-0.3,-0.8) {$\zeta$};
\node at (-1.35,-1.6) {$\underbrace{\qquad \qquad }_{u_3}$};

------
\draw [<-] (3.85,-1) -- (3.1,-0.1);
\draw [-,decorate,decoration=snake] (3.85,-1) -- (5,-1);
\node at (3.3,-0.7) {$\sigma$};
\node at (4,-1.6) {$\underbrace{\qquad \,\,\,\, }_{u_1}$};
\draw[dashed] (0.1,-1) -- (0.1,1);
\draw[dashed] (3,-1) -- (3,1);


\node at (10,0) {
\begin{tikzcd}

  & M(w)=M(v_3v_2v_1)
\\
    & M(u_2)\simeq M(v_2) \arrow[hookrightarrow]{u}{\iota}
\\
  & M(w)=M(u_3u_2u_1) \arrow[->>]{u}{\pi} 
\end{tikzcd}
};

\end{tikzpicture}
\end{center}
Here the substrings $u_i$'s and $v_i$'s, for every $1\leq i \leq 3$, could be of any length, provided that $l(u_1)+l(u_3)>0$ and $l(v_1)+l(v_3)>0$. These inequalities guarantee that at least one of the two arrows $\sigma$ or $\zeta$ (respectively $\theta$ or $\gamma$) that leave $u_2$ (respectively which enter $v_2$) is actually present in the configuration.
Moreover, the squiggly parts of the substrings (except for the specified arrows), can be of any orientation.

\begin{remark}\label{induced graph maps}
Let $\Lambda$ be a string algebra and $u$ and $v$ be in $\Str(\Lambda)$. For each admissible pair $T=((u_3,u_2,u_1),(v_3,v_2,v_1))$ in $\mathtt{A}(u,v)$, observe that if we remove all the arrows in $u_1$ and $u_3$ (respectively $v_1$ and $v_3$) which are not connected to $u_2$ (respectively to $v_2$), we obtain the admissible pair $((u'_3, u_2,u'_1), (v'_3,v_2, v'_1))$ in $\mathtt{A}(u'_3 u_2 u'_1 , v'_3 v_2 v'_1)$, where $1 \leq l(u'_1)+l(u'_3) \leq 2$ and similarly $1 \leq l(v'_1)+l(v'_3) \leq 2$.
Suppose $f_{T'}$ is the graph map given by the pair $u':=u'_3u_2u'_1$ and $v':=v'_3v_2v'_1$, which are respectively substrings of $u$ and $v$. Then, for $w \in \Str(\Lambda)$, the graph map $f_T$ in $\End_{\Lambda}(M(w))$
can be recovered from a graph map $f_{T'}$ in $\Hom_{\Lambda}(M(u'),M(v'))$, for some $T'$ in $\mathtt{A}(u',v')$, where $u'=u'_3u_2u'_1$ and $v'=v'_3v_2v'_1$ are two distinct substrings of $w$ with $l(u_2)+1 \leq l(u') \leq l(u_2)+2$ and $l(v_2)+1 \leq l(v') \leq l(v_2)+2$, such that $u_2 \sim v_2$, while $u'_3$ and $v'_1$ (respectively $u'_1$ and $v'_3$) are direct (respectively inverse) arrows, provided they are strings of positive length. 
For $w \in \Str(\Lambda)$, we often use this simple observation and find non-invertible graph maps in $\End_{\Lambda}(M(w))$ to verify $M(w)$ is not a brick.
\end{remark}

\begin{proposition}\label{Barbells tau-infinite}
If $\Lambda$ is a barbell algebra, $\modu \Lambda$ contains infinitely many isoclasses of bricks.
\end{proposition}
Before we prove the proposition, let us recall that the general configuration of the bound quiver of any barbell algebra $kQ(v,v^{-1})/\langle \beta \alpha, \delta \gamma \rangle$ is as in Figure \ref{fig:bound quiver of Barbell algebra}, where $C_L$ (respectively $C_R$) denote the cyclic string which starts and ends at $x$ (respectively at $y$), and $\mathfrak{b}$, $C_L$ and $C_R$ could be of any positive length (in particular, we may have $\alpha=\beta$ or $\gamma=\delta$).
Although the arrows $\alpha$, $\beta$, $\gamma$ and $\delta$ must appear with the indicated orientation, the other arrows can have any orientation. This freedom is depicted by the dashed segments. From this description, it follows that every band supports all arrows of $Q(v,w^{-1})$ and one can also verify Lemma \ref{min-rep-infiniteinite barbell algebras}, asserting that the barbell algebra $kQ(v,v^{-1})/\langle \beta \alpha, \delta \gamma \rangle$ is min-rep-infinite if and only if $\mathfrak{b}$ is not serial. 
\begin{figure}
    \centering
\begin{tikzpicture}
 \draw [->] (1.25,0.75) --(2,0.1);
    \node at (1.7,0.55) {$\alpha$};
 \draw [<-] (1.25,-0.75) --(2,-0.05);
    \node at (1.7,-0.5) {$\beta$};
  \draw [dashed] (1.25,0.75) to [bend right=100] (1.25,-0.75);
   \node at (1.3,0) {$C_L$};
    \node at (2,0) {$\circ $};
    \node at (2.1,-0.2) {$x$};
    
    \draw [dotted,thick] (1.65,-0.25) to [bend right=50](1.65,0.35);
 \draw [dashed] (2.05,0) --(2.75,0);
 \node at (2.95,0) {$\bullet$};
 \draw [-] (2.95,0) --(3.8,0);
 \node at (3.8,0) {$\bullet$};
 \draw [dashed] (4,0) --(4.7,0);
 \node at (3.4,0.3) {$\overbrace{\qquad \qquad\qquad \qquad}^{\mathfrak{b}}$};
  \node at (3.5,-0.25) {$\theta^{\epsilon_i}_i$};
 
 \node at (4.75,0.0) {$\circ$};
 \draw [<-] (5.5,0.75) --(4.75,0.05);
 \node at (5,0.45) {$\delta$};
 \draw [->] (5.50,-0.8) --(4.8,-0.05);
    \node at (5,-0.55) {$\gamma$};
  \draw [dashed] (5.55,0.8) to [bend left=100] (5.55,-0.8);
   \node at (5.5,0) {$C_R$};
   \node at (4.65,-0.2) {$y$};
   
\draw [dotted,thick] (5.15,-0.3) to [bend left=50](5.15,0.3);
\end{tikzpicture}
    \caption{Bound quiver of a barbell algebra.}
    \label{fig:bound quiver of Barbell algebra}
\end{figure}
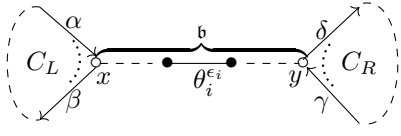
    
\begin{proof} [Proof of Proposition \ref{Barbells tau-infinite}]
Let $\mathfrak{b}=\theta^{\epsilon_d}_d \cdots \theta^{\epsilon_2}_2 \theta^{\epsilon_1}_1$ with $\theta_i \in Q_1$, for all $1 \leq i \leq d$, and $s(\mathfrak{b})=x$ and $e(\mathfrak{b})=y$. We know that $l(\mathfrak{b}) \geq 1$, thus depending on $\epsilon_1$, the arrow $\theta_1$ is either incoming to $x$ (if $\epsilon_1=-1$) or outgoing from $x$ (if $\epsilon_1=1$). We show that in either case, there are infinitely many non-isomorphic bricks in $\Str(\Lambda)$.

For the left and right cyclic strings $C_L$ and $C_R$, we respectively fix the walk $C_L=\alpha \cdots \beta$ and $C_R=\gamma \cdots \delta$. i.e, $C_L$ starts with $\beta$ (respectively $C_R$ starts with $\delta$) and ends with $\alpha$ (respectively ends with $\gamma$).

If $\epsilon_1=1$, then consider the string $w:= \mathfrak{b}^{-1} C_R \mathfrak{b} C_L$. We claim that $M(w^m)$ is a brick, for every $m \in \mathbb{Z}_{>0}$. For $m=1$, notice that the only substrings of $w$ that appear more than once in $w$ are the substrings of the bar $\mathfrak{b}$. In fact every such a substring appears exactly twice, where the second copy is the inverse of the first copy. Obviously, the two copies of $\mathfrak{b}$ (being $\mathfrak{b}$ and $\mathfrak{b}^{-1}$) do not induce a graph map, because by the remark preceding the proposition, the local configuration of $w$ around vertices $x$ and $y$ do not satisfy the required conditions for a graph map. 
Moreover, for any proper substring $v$ of $\mathfrak{b}$ (which also appears exactly twice in $w$, say $v'$ and $v''$), at least one arrow of $w$ entering (respectively leaving) $v'$ is repeated in the reverse direction, leaving (respectively entering) the second copy $v''$, hence the local configuration of a graph map never occurs at the ends of the copies of $v$. This shows that the only non-zero maps in $\End_{\Lambda}(M(w))$ are scalar multiplications of the identity morphism. 

Now we prove that $M(w^m)$ is a brick for every $m \in \mathbb{Z}_{>1}$. Assuming otherwise, there exists $t \in \mathbb{Z}_{>1}$ such that $M(w^{t-1})$ is a brick, but $M(w^{t})$ is not. For simplicity, we first derive the desired contradiction in the case $t=2$ and then generalize the argument to any $t \in \mathbb{Z}_{>0}$.

By Theorem \ref{Thm-graph-maps}, $M(w^{2})$ admits a non-invertible graph map if and only if there exists a non-trivial admissible pair $T=((u_3,u_2,u_1),(v_3,v_2,v_1))$ in $\mathfrak{A}(w^{2},w^{2})$. 
Note that if $u_2$ (and therefore $v_2$) is a proper substring of $w$, we can generate a non-invertible graph map of $M(w)$ by chopping off the extra arrows, as explained in the remark preceding the proposition. Moreover, if $u_2\sim w$, one copy of $w$ (being $u_2$) appears on the top and the second copy (being $v_2$) at the bottom of $w^{2}$. This, under our assumption $\epsilon_1=1$, can never occur for $w= \mathfrak{b}^{-1} C_R \mathfrak{b} C_L$ because $w$ ends with an inverse arrow and starts with a direct arrow, which rules out the possibility of an admissible pair with $u_2 \sim w$ and $v_2 \sim w$ in $w^2$. This implies that $u_2$ (and therefore $v_2$) is not a substring of $w$. Hence, as proper substring of $w^2$, $u_2$ and $v_2$ must contain the substring $\beta\theta^{-1}_1$ (where the first copy of $w$ terminates with inverse arrow $\theta^{-1}_1$ and the second copy starts with the direct arrow $\beta$). This is obviously a contradiction because $\beta\theta^{-1}_1$ in $w^2$ occurs exactly once, while $u_2$ and $v_2$ are two distinct substrings of $w^2$ with $u_2 \sim v_2$.

Now, assume $t > 2$ and let $T=((u_3,u_2,u_1),(v_3,v_2,v_1))$ be an admissible pair in $\mathfrak{A}(w^{t},w^{t})$
which induces a non-invertible graph map in $\End(M(w^t))$. For a similar reason, as in case $t=2$, we conclude that $u_2$ and therefore $v_2$ cannot be substrings of $w^{t-1}$. This, as in the previous case, implies that both $u_2$ and $v_2$ must contain exactly $(t-1)$-many copies of $\beta\theta^{-1}_1$ as a substring, which is impossible because $\beta\theta^{-1}_1$ appears as a substring of $w^t$ exactly $(t-1)$-many times (once for the connection of each copy of $w$ with the next copy of $w$ in $w^t$). From this, we get the desired contradiction and conclude that $M(w^t)$ is a brick for every $t \in \mathbb{Z}_{>0}$.

As the only other possibility, if $\epsilon_1=-1$, we consider the string $w= \mathfrak{b}^{-1} C_R^{-1} \mathfrak{b} C_L^{-1}$ and a similar argument shows that every $M(w^m)$ is a brick over the barbell algebra.
\end{proof}

\begin{remark}\label{powers of bricks may not be brick}
In order to highlight the significance of the orientations of the arrows in the construction of the band $w$ in the preceding theorem, recall that a barbell quiver admits infinitely many bands. However, among them, there is a unique band that gives rise to infinitely many non-isomorphic bricks. We use the following example to clarify this point.

Consider the barbell algebra given by the following bound quiver, which is obtained from an acyclic copy of $\widetilde{\mathbb{A}}_3$ which has a string of the form $\delta \gamma \beta^{-1} \alpha$. In particular, $\Lambda_{\mathfrak{b}}$ is the result of barification of $\gamma \beta^{-1} \alpha$ and $\alpha \delta \gamma$, where $\beta$ and $\delta$ are identified. 
Then we have
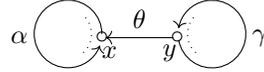
\begin{center}
    \begin{tikzpicture}

    \draw[->] (3.52,2.47) arc (190:530:0.45cm);
\node at (4.6,2.5) {$\gamma$};    
\draw [dotted] (3.75,2.3) to [bend left=50](3.75,2.8);
        \node at (3.5,2.5) {$\circ$};
        \node at (3.4,2.3) {$y$};
   \draw [->] (3.45,2.5) --(2.55,2.5); 
\node at (3,2.75) {$\theta$};
    \node at (2.5, 2.5) {$\circ$}; 
    \node at (2.6, 2.3) {$x$}; 

    \draw[->] (2.5,2.55) arc (0:340:0.45cm);
\node at (1.4,2.5) {$\alpha$};    
\draw [dotted] (2.25,2.3) to [bend right=50](2.25,2.8);

    \end{tikzpicture}
\end{center}

Because bands are invariant under cyclic permutations and their inverses, we can always assume that a band $v$ in $\Str(\Lambda_{\mathfrak{b}})$ starts at vertex $x$ and its first two arrows are either $\theta^{-1}\alpha$ or $\theta^{-1}\alpha^{-1}$.
If $v$ starts with $\theta^{-1}\alpha$, one can easily check that $S_x$ belongs to the top and socle of $M(v)$, implying that $M(v^m)$ is never a brick, for any $m \in \mathbb{Z}_{>0}$. 
Moreover, if $v$ starts with $\theta^{-1}\alpha^{-1}$ and has a substring of the form $\gamma \theta^{-1}\alpha^{-1}$, then $M(v^m)$ is never a brick, for any $m \in \mathbb{Z}_{>1}$. This is because if $\gamma \theta^{-1}\alpha^{-1}$ follows via $\alpha \theta$, then we have $\theta^{-1}\alpha^{-1}$ is both on the top and in the bottom of $v^2$, whereas if $\gamma \theta^{-1}\alpha^{-1}$ follows via $\alpha^{-1} \theta$, then $\theta$ is both on the top and bottom of $v^2$. 
This, in particular, shows that to create a band which is brick and its powers do not admit a substring appearing both in the top and bottom, we must start with $\alpha^{-1}$, go through $\theta^{-1}$, continue via $\gamma^{-1}$, which forces that the substring of $v$ which appears at the beginning is $\theta \gamma^{-1}\theta^{-1} \alpha^{-1}$. 

Due to an argument similar to the above, this band must continue via $\alpha^{-1}$. Otherwise, $\theta^{-1} \alpha \theta \gamma^{-1}\theta^{-1} \alpha^{-1}$ is the substring appearing at the beginning of $v$.
This implies that the initial substring of $v$ is either $\cdots \theta \gamma^{-1} \theta^{-1} \alpha \theta \gamma^{-1}\theta^{-1} \alpha^{-1}$ or $\cdots\theta^{-1} \alpha \theta \gamma \theta^{-1} \alpha \theta \gamma^{-1}\theta^{-1} \alpha^{-1}$ or $\cdots \alpha^{-1} \theta \gamma \theta^{-1} \alpha \theta \gamma^{-1}\theta^{-1} \alpha^{-1}$. It is easy to check that in all cases we have a substring which appear both on the top and in the bottom of $v$. In particular, for the three aforementioned cases, we respectively have the substrings $\theta \gamma^{-1} \theta^{-1}$, and $\theta \gamma^{-1} \theta^{-1} \alpha^{-1}$ and $\theta^{-1} \gamma^{-1} \theta$ appear on the top and in the bottom. This shows that the only band which generates infinitely many non-isomorphic bricks in given by $v=\theta \gamma^{-1} \theta^{-1} \alpha^{-1}$.
\end{remark}

The following corollary is a direct consequence of Theorem \ref{tau-finiteness}, Theorem \ref{quotient-lattice} and Proposition \ref{Barbells tau-infinite}.

\begin{corollary}
For an arbitrary algebra $\Lambda$, if $\Lambda$ has a quotient isomorphic to a barbell algebra, then $\Lambda$ is $\tau$-tilting infinite.
\end{corollary}

Before we state the next proposition, let us fix some notation to simplify our proof. Recall that if $\mathfrak{b}_1, \cdots, \mathfrak{b}_d$ are the bars of $(Q^{\mathfrak{w}}(v_i,v^{-1}_{\sigma(i)}), I_{\mathfrak{w}})$, then $l(\mathfrak{b}_t) >0$, for every $1 \leq t \leq d$, where $\mathfrak{b}_t= \theta^{\epsilon_{t}}_{t_m} \cdots \theta^{\epsilon_{t}}_{t_1}$, for some $m_t \in \mathbb{Z}_{>0}$ and $\theta_{t_s} \in Q_1$. Recall that each $\mathfrak{b}_t$ comes from barification of two identical copies of a linearly oriented $\mathbb{A}_{q_i}$, given by the strings $v_i$ and $v^{-1}_{\sigma(i)}$ in $\widetilde{\mathbb{A}}_n= v_{2k}u_{2k} \cdots v_1u_1$ (see Section \ref{Section:Bound Quivers of Mild Special Biserial Algebras}). Thus, every bar is serial and $\epsilon_t \in \{\pm 1\}$ is the same for all arrows in the bar $\mathfrak{b}_t$.
Although the bound quiver of a wind wheel algebra can be complicated, the local configuration of each bar and the incoming and outgoing arrows at its start and end, as shown in Figure \ref{fig: Bar in wind wheel alge}, is often enough to analyze the graph maps between those strings supported by $\mathfrak{b}_t$ and its neighboring arrows.

\begin{figure}

\begin{center}
\begin{tikzpicture}[scale=1]

\draw [dashed] (1.4,-0.8) to (1,-0.8);
\draw [->] (1,-0.8) --(0.3,-0.05);
    \node at (0.5,-0.65) {$\gamma_{t}$};
\draw [<-] (1,0.75) --(0.3,0.05);
     \node at (0.5,0.6) {$\delta_{t}$};
\draw [dashed] (1.4,0.75) to (1,0.75);
    \draw [dotted,thick] (0.7,-0.3) to [bend left=50](0.7,0.3);
   \node at (1.25,0) {$C_{R_{t}}$};

     \node at (0.25,0) {$\circ$};
       \node at (0.25,-0.3) {$y_{t}$};
\draw [dashed] (0.2,0) --(-0.35,0); 
    \node at (-0.4,0) {$\bullet$};
    \node at (-0.9,0) {$\bullet$};
\draw [-] (-0.9,0) --(-0.4,0); 
\draw [dashed] (-1.4,0) --(-0.85,0); 
    \node at (-1.45,0) {$\circ$};
    \node at (-1.4,-0.3) {$x_{t}$};
    \node at (-0.6,0.3) {$\overbrace{\qquad  \qquad \quad}^{\mathfrak{b}_{t}}$};

     \node at (-0.57,-0.3) {$\theta^{\epsilon_t}_j$};

\draw [->] (-2.1,0.8) --(-1.5,0.05);
 \node at (-1.7,0.6) {$\alpha_{t}$};

\draw [<-] (-2.1,-0.8) --(-1.5,0);
    \node at (-1.75,-0.65) {$\beta_{t}$};
    \draw [dotted,thick] (-1.85,-0.3) to [bend right=50](-1.85,0.3);

  \draw [dashed] (-2.1,0.8) to (-2.7,0.8);
  \draw [dashed] (-2.15,-0.8) to (-2.7,-0.8);
   \node at (-2.3,0) {$C_{L_{t}}$};
   
    \end{tikzpicture}
\end{center}

    \caption{Local configuration of each bar in a wind wheel algebra}
    \label{fig: Bar in wind wheel alge}
\end{figure}
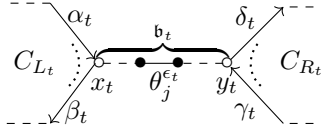

If $v$ is a band in the wind wheel algebra with a fixed presentation and starting point, each bar $\mathfrak{b}_t$ can be locally visualized as in Figure \ref{fig: Bar in wind wheel alge}, where we let $x_t$ denote the first vertex of $\mathfrak{b}_t$ that $v$ visits, and $y_t$ be the other end of $\mathfrak{b}_t$. Since every band in a wind wheel algebra supports all arrows and visits every vertex, once the presentation and starting point of $w$ is fixed, as we assume henceforth, $x_t$ and $y_t$ are well defined and because $l(\mathfrak{b}_t) >0$ for every $t$, all the $x_t$ and $y_t$ are distinct. 
Considering every bar $\mathfrak{b}_t$ with a local configuration as in Figure \ref{fig: Bar in wind wheel alge}, we further assume $\mathfrak{b}_t$ is oriented away from $e(\alpha_t)$ if $\epsilon_t=1$, and towards it if $\epsilon_t=-1$. Every internal vertex of a bar is a $2$-vertex with exactly one incoming and one outgoing arrow, whereas $x_t:=s(\mathfrak{b}_t)$ and $y_t:=e(\mathfrak{b}_t)$ are the $3$-vertices involved in the C-relations, as shown in Figure \ref{fig: Bar in wind wheel alge}. 
In addition to the C-relations, every B-relation along $\mathfrak{b}_t$ kills the unique path $p(\mathfrak{b}_t)$ of length $l(\mathfrak{b}_t)+2$ in $(Q^{\mathfrak{w}}(v_i,v^{-1}_{\sigma(i)}), I_{\mathfrak{w}})$ determined by $\epsilon_t$. In particular, $p(\mathfrak{b}_t)=\delta_{t+1} \mathfrak{b}_t \alpha_{t}$, if $\epsilon_t=1$, and $p(\mathfrak{b}_t)=\beta_{t} \mathfrak{b}_t \gamma_{t+1}$, if $\epsilon_t=-1$.

To state the next result, we need to recall that an infinite word $q$ is called \emph{biperiodic} if there exist finite words $p_1, p_2$ and $p_3$ such that $q={}^{\infty}(p_3)p_2(p_1)^{\infty}$. Morever, a string algebra $\Lambda=kQ/I$ is \emph{domestic} if for each arrow $\gamma$ in $Q$, there is at most one band (up to cyclic permutation and inverse) which starts with $\gamma$.

The following statement gives a description of strings in a wind wheel algebra.

\begin{def-prop}[\cite{R2} Prop. 6.2] \label{strings of wind wheel algebras}
Every wind wheel algebra $\Lambda_{\mathfrak{w}}$ is domestic and has a unique band. 
Let $v$ be the unique band in $\Lambda_{\mathfrak{w}}$ and for a bar $\mathfrak{b}_t$ in $\Lambda_{\mathfrak{w}}$, consider
$v=\mathfrak{b}^{-1}_t w_2 \mathfrak{b}_t w_1$. Then, put $v^{-1}_{\mathfrak{b}_t}:=\mathfrak{b}_t w_1^{-1} \mathfrak{b}_t^{-1} w_2^{-1}$, as the cyclic permutation of $v^{-1}$ which ends with the bar $\mathfrak{b}_t$. The \emph{$\mathfrak{b}_t$-infinite word} is the biperiodic infinite word defined as:
$$w(\mathfrak{b}_t):= {}^{\infty}(v) \mathfrak{b}_t^{-1} w_2 ({v^{-1}_{\mathfrak{b}_t}})^{\infty}.$$

Then, every string $w$ in $\Str(\Lambda_{\mathfrak{w}})$ with $l(w)>l(v)$ is either
isomorphic to a finite subword of $v^{\infty}$ or a finite subword of the $\mathfrak{b}_t$-infinite biperiodic word $w(\mathfrak{b}_t)$, for some bar $\mathfrak{b}_t$ in $\Lambda_{\mathfrak{w}}$.
\end{def-prop}

\begin{proposition}\label{Wind wheel tau-finite}
Every wind wheel algebra is brick finite.
\end{proposition}

\begin{proof}

Let $\Lambda_{\mathfrak{w}}$ be a wind wheel algebra and $v$ be the band in $\Str(\Lambda_{\mathfrak{w}})$, which is unique up to cyclic permutations and their inverse.
First we show that for every $w\in \Str(\Lambda_{\mathfrak{w}})$ whose length is large enough (say for example $l(w)>2l(v)$), the string module $M(w)$ admits a non-invertible graph map.

The lower bound $2l(v)$ on the length of the string $w$, 
along with a simple consequence of Proposition \ref{strings of wind wheel algebras}, guarantee that any such string in $\Str(\Lambda)$ contains $v$ as a substring. Therefore, $w$ supports a bar $\mathfrak{b_t}$ in both directions and, as we will show, it contains a pair of substrings around $\mathfrak{b}_t$ which satisfies the configuration of an admissible pair in $\mathtt{A}(w,w)$ and give rise to a nontrivial graph map. Therefore, we will obtain a nonzero non-invertible endomorphism of $M(w)$ and conclude that $M(w)$ is not a brick.

In particular, if we use the notation and convention of Figure \ref{fig: Bar in wind wheel alge} to explicitly determine this non-invertible graph maps in $\End(M(w))$, then depending on $\epsilon_t=1$ or $\epsilon_t=-1$, one of the following local configurations appears in $w$:

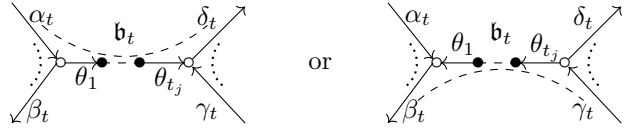
\begin{center}
\begin{tikzpicture}[scale=1]

\draw [->] (1,-0.8) --(0.3,-0.05);
    \node at (0.5,-0.65) {$\gamma_{t}$};
\draw [<-] (1,0.75) --(0.3,0.05);
     \node at (0.5,0.6) {$\delta_{t}$};
    \draw [dotted,thick] (0.7,-0.3) to [bend left=50](0.7,0.3);

     \node at (0.25,0) {$\circ$};
       \node at (0.,-0.25) {$\theta_{t_j}$};
\draw [<-] (0.2,0) --(-0.35,0); 
    \node at (-0.4,0) {$\bullet$};
    \node at (-0.9,0) {$\bullet$};
\draw [dashed] (-0.9,0) --(-0.4,0); 
\draw [->] (-1.4,0) --(-0.95,0); 
    \node at (-1.45,0) {$\circ$};
    \node at (-1.1,-0.2) {$\theta_1$};
    \node at (-0.6,0.4) {${\mathfrak{b}_{t}}$};

\draw [->] (-2.1,0.8) --(-1.5,0.05);
 \node at (-1.7,0.6) {$\alpha_{t}$};

\draw [<-] (-2.1,-0.8) --(-1.5,0);
    \node at (-1.75,-0.65) {$\beta_{t}$};
    \draw [dotted,thick] (-1.85,-0.3) to [bend right=50](-1.85,0.3);

    \draw [dashed] (-1.7,0.5) to [bend right=40](0.5,0.5);

 \node at (2,0) {$\text{or}$};

\draw [->] (6,-0.8) --(5.3,-0.05);
    \node at (5.5,-0.65) {$\gamma_{t}$};
\draw [<-] (6,0.75) --(5.3,0.05);
     \node at (5.5,0.6) {$\delta_{t}$};
    \draw [dotted,thick] (5.7,-0.3) to [bend left=50](5.7,0.3);

     \node at (5.25,0) {$\circ$};
       \node at (5,0.2) {$\theta_{t_j}$};
\draw [->] (5.2,0) --(4.65,0); 
    \node at (4.6,0) {$\bullet$};
    \node at (4.1,0) {$\bullet$};
\draw [dashed] (4.1,0) --(4.6,0); 
\draw [<-] (3.6,0) --(4.05,0); 
    \node at (3.55,0) {$\circ$};
    \node at (3.9,0.25) {$\theta_1$};
    \node at (4.4,0.4) {${\mathfrak{b}_{t}}$};

\draw [->] (2.9,0.8) --(3.5,0.05);
 \node at (3.3,0.6) {$\alpha_{t}$};

\draw [<-] (2.9,-0.8) --(3.5,0);
    \node at (3.25,-0.65) {$\beta_{t}$};
    \draw [dotted,thick] (3.15,-0.3) to [bend right=50](3.15,0.3);

    \draw [dashed] (3.3,-0.5) to [bend left=40](5.5,-0.5);
    \end{tikzpicture}
\end{center}

If we ignore the B-relations, up to symmetry we can assume that the local configuration around a bar $\mathfrak{b}_t$ that occurs in $w$ is as in the diagram on the left. 
Therefore, for a bar $\mathfrak{b}_t$, the string $w$ could contain either the pair $\gamma^{-1}_t\mathfrak b\alpha_t$ and $\beta_t\mathfrak b^{-1}\delta^{-1}_t$ or the pair $\delta_t\mathfrak b\alpha_t$ and $\beta_t\mathfrak b^{-1}\gamma_t$. However, depending on $\epsilon_t$, one of these is ruled out by the corresponding B-relation, while the other one gives rise to a non-invertible nonzero endomorphism. 
For example, in the case of the left diagram, where $\epsilon_t=1$, the B-relation rules out the pair $\delta_t \mathfrak{b}\alpha_t$ and 
$\beta_t \mathfrak{b}^{-1}\gamma_t^{-1}$ (because $\delta_t \mathfrak{b}\alpha_t$ belongs to $I$). 
Thus, $w$ contains the substrings $\gamma^{-1}_t\mathfrak b\alpha_t$ and $\beta_t\mathfrak b^{-1}\delta^{-1}_t$.
Now, it is easy to see that  this pair of substrings induces a nonzero non-invertible endomorphism of $M(w)$ by the associated graph map along $\mathfrak{b}_t$.

For every $\lambda \in k \setminus \{0\}$, in order to find a non-invertible nonzero endomorphism in $\End_{\Lambda}(M(v,\lambda))$, we can similarly use the interaction of two substrings of $v$ around a bar $\mathfrak{b}_t$. Namely, analogous to the graph map we found for any string modules $M(w)$ with $l(w)>2l(v)$, one can use the local interaction of two specific substrings of $v$ around a bar $\mathfrak{b}_t$ to find a nonzero endomorphism of the band module $M(v,\lambda)$. In particular, this endomorphism is not onto, and therefore non-invertible. Hence, $M(v,\lambda)$ is not a brick.
In fact, as shown in \cite{R2}, the radical of $\End(M(v,\lambda))$ is $d$-dimensional with a basis given by the above endomorphisms, where $d$ is the number of bars in the wind wheel algebra $\Lambda_{\mathfrak{w}}$. 

These cases together show that $\modu \Lambda_{\mathfrak{w}}$ has only finitely many isomorphism classes of bricks and we are done by Theorem \ref{tau-finiteness}.

\end{proof}

\begin{remark}\label{examples of min-tau-infinite string alg}
One should note that the $\tau$-tilting finiteness of wind wheel algebra heavily depends on the additional B-relations we added after barifications. Namely, from the same analysis made in the case of barbell algebras, it easily follows that a sequence of barifications of an acyclic $\widetilde{\mathbb{A}}_n$ always results in a brick infinite algebra. However, the result may not be necessarily min-rep-infinite.
\end{remark}

By the above results, we can now fully determine whether a given $\Lambda$ in $\Mri(\mathfrak{F}_{\sB})$ is $\tau$-tilting finite, provided the bound quiver of $\Lambda$ has no node. 
To complete this characterization of $\Mri(\mathfrak{F}_{\sB})$ with respect to $\tau$-tilting finiteness, in the next section we separately treat those min-rep-infinite special biserial algebras whose bound quivers have a node. 
\begin{theorem}\label{tau-finite node-free mSB Thm}
Let $\Lambda$ be a node-free minimal representation-infinite special biseiral algebra. $\Lambda$ is $\tau$-tilting infinite if and only if $\Lambda=k \widetilde{\mathbb{A}}_n$ or $\Lambda$ is a barbell algebra.
\end{theorem}

\begin{proof}
As stated in Theorem \ref{tau-finiteness}, an algebra $\Lambda$ is $\tau$-tilting infinite if and only if $\modu \Lambda$ contains infinitely many non-isomorphic bricks. If $\Lambda=k \widetilde{\mathbb{A}}_n$, the Auslander-Reiten quiver of $\Lambda$ has preprojective and preinjective components, thus, by Lemma \ref{Preproj/Postinj brick}, we are done. If $\Lambda$ is non-hereditary, then by Theorem \ref{Ringel's Classification Thm}, it is a barbell or a wind wheel algebra, where Propositions \ref{Barbells tau-infinite} and \ref{Wind wheel tau-finite} imply the desired result.
\end{proof}

\section{$\tau$-Tilting Finiteness of Nody Special Biserial Algebras}{\label{section: Nody algebras}}

As shown by Ringel \cite{R2} and reviewed in Section \ref{Section:Bound Quivers of Mild Special Biserial Algebras}, the study of minimal representation-infinite special biserial algebras up to stable equivalence can be reduced to the study of node-free algebras $\Lambda=kQ/I$. As a result, any minimal representation-infinite algebra is stably equivalent to a cycle, barbell, wind wheel algebra (see Theorem \ref{Ringel's Classification Thm}).
However, unlike the representation type of algebras, $\tau$-tilting finiteness is not invariant under stable equivalence. Thus, for our primary objective in this paper we still need to consider those algebras $\Lambda$ in $\Mri(\mathfrak{F}_{\sB})$ for which $\nn(\Lambda)\neq \emptyset$.
To complete our classification of min-rep-infinite special biserial algebras from the $\tau$-tilting finiteness viewpoint, in this section we treat such algebras. In the following we call $\Lambda=kQ/I$ a \emph{nody} algebra if it is min-rep-infinite special biserial algebras and $\nn(\Lambda)\neq \emptyset$.

As the main goal of this section, we prove that every nody algebra is $\tau$-tilting finite. This, along with our results from the preceding section, fully determines which algebras in $\Mri(\mathfrak{F}_{\sB})$ are $\tau$-tilting finite and which ones are $\tau$-tilting infinite.

Before we study $\tau$-tilting finiteness of the nody algebras, let us emphasize a useful observation we made through the construction of barbell and wind wheel algebras: if $(Q,I)$ is a string bound quiver obtained solely via a sequence of barifications of an acyclic $\widetilde{\mathbb{A}}_n$, then $\Lambda=kQ/I$ is again $\tau$-tilting infinite. 
However, since $(Q,I)$ may not be necessarily min-rep-infinite, in order to result in a min-rep-infinite algebra, we must add new relations to the set of relations imposed by the barifications. 
These new relations, as we observed in the wind wheel algebras, can destroy $\tau$-tilting infiniteness, such that the quotient algebra is rep-infinite but $\tau$-tilting finite.

To illustrate some of the above points and general features of nody special biserial algebras, we first consider an interesting family of string algebras in the following example.
As before, $\overline{Q}$ denotes the double quiver of $Q$, obtained by adding the reverse arrow $\gamma^*$ for each $\gamma \in Q_1$. 

\begin{example}\label{Major NON-Example}
For every $n \in \mathbb{Z}_{\geq 0}$, suppose $Q=\widetilde{\mathbb{A}}_n$ and $\Lambda_n=k\overline{Q}/I$, where $I=\langle \beta\alpha \,|\, \forall \alpha,\beta \in \overline{Q} \rangle $. 
Thus, $\Lambda_n$ is a radical-square zero string algebra and
all vertices are $4$-vertex and $\node(\Lambda_n)=Q_0$. Since there exists a band which alternates between direct and reverse arrows at every vertex, $\Lambda_n$ is obviously rep-infinite.
Moreover, $\Lambda_n$ is min-rep-infinite if and only if $n$ is even. This is because for odd $n$, in the double quiver of $\widetilde{\mathbb{A}}_n$ there are two distinct copies of $\widetilde{\mathbb{A}}_n$ with alternating arrows. Therefore, if $n$ is odd, $\Lambda_n$ is a $\tau$-tilting infinite string algebra. We remark that, by definition, for any odd $n$ the algebra $\Lambda_n$ is not a nody algebra.

In contrast, if $n$ is even, $\Lambda_n$ has a unique band (of length $2n+2$) of the form
$$ w= \gamma_{2}^{-1} \gamma_{3}^{*}\gamma_{4}^{-1} \cdots \gamma_{n-1}^{*} \gamma_{n}^{-1}\gamma_{n+1}^{*} \gamma_{1}^{-1} \cdots \gamma_{n-1}^{-1} \gamma_{n}^{*} \gamma_{n+1}^{-1} \gamma_1^{*}$$ where $\gamma_i$ and $\gamma^{*}_i$ respectively denote the original and reverse arrows in $\overline{Q}$.
Note that this band supports every arrow once and visits each vertex twice, except for $s(w)=e(w)$, visited three times. Thus, minimality of $\Lambda_n$ is immediate. Moreover, in this case (if $n$ is even), $\modu \Lambda_n$ has only finitely many bricks, because any string of $\Lambda_n$ is a finite portion of the infinite word ${}^{\infty}w$, whose diagram is as follows:
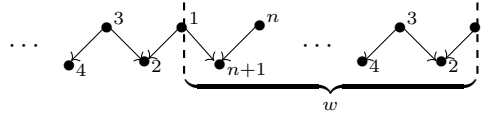
\begin{center}
\begin{tikzpicture}[scale=0.5]
-------------------------
---
\node at (1.1,1.1) {$\bullet^1$};
\draw [->] (1,1) -- (0.15,0.15);
--
\node at (0.2,0) {$\bullet_2$};
\draw [->] (-1,1) -- (-0.1,0.1);
\node at (-0.9,1.1) {$\bullet^3$};
--
\draw [->] (-1.1,1) -- (-2,0.1);
\node at (-1.9,0) {$\bullet_4$};
--
\node at (-3.2,0.5) {$\cdots$};
--
\node at (-4.6,1.1) {$\bullet^{n}$};
\draw [->] (-4.9,1) -- (-5.8,0.1);
\node at (-5.3,-0.1) {$\bullet_{n+1}$};
--
\draw [->] (-6.8,1) -- (-6,0.1);
\node at (-6.7,1.1) {$\bullet^1$};
--
\draw [->] (-6.9,1) -- (-7.75,0.15);
\node at (-7.7,0) {$\bullet_2$};
\draw [->] (-8.8,1) -- (-7.9,0.1);
--
\node at (-8.7,1.1) {$\bullet^{3}$};
\draw [->] (-8.9,1) -- (-9.8,0.1);
\node at (-9.7,-0.1) {$\bullet_{4}$};
--
\node at (-11,0.5) {$\cdots$};
\draw [dashed,thick] (1,1.65) -- (1,-0.4);
\draw [dashed,thick] (-6.8,1.65) -- (-6.8,-0.4);

 \node at (-2.9,-0.8) {$\underbrace{\qquad \qquad \qquad \qquad \qquad \quad}_w$};

\end{tikzpicture}
\end{center}
Hence, if $l(v)>n$, there exists at least one vertex $x$ such that the simple module $S(x)$ is both a summand of $\topm(M(v))$ and $\soc(M(v))$. This admits a non-invertible graph map in $\End_{\Lambda_n}(M(v))$. 
A similar argument shows that for each band module $M(w,\lambda)$, where $\lambda \in k \setminus \{0\}$, the algebra $\End(M(w,\lambda))$ has non-invertible elements.
Thus, there are only finitely many $\Lambda_n$-modules which are brick and for every even $n$ the nody algebra $\Lambda_n$ is $\tau$-tilting finite.

Thus, in the above infinite family $\{\Lambda_n \,|\, n \in \mathbb{Z}_{>0} \}$ of string algebras, $\Lambda_n$ is $\tau$-tilting finite if and only if $\Lambda_n$ is a nody algebra (i.e, if and only if $n$ is even). Note that if $n$ is even, we have $\nn(\Lambda_n)=\widetilde{\mathbb{A}}_{2n+1}$, which is $\tau$-tilting infinite. 
This is an instance of the phenomenon we previously alluded to, where $\tau$-tilting finiteness is not preserved under resolving the nodes.
Finally note that if $n$ is odd, $\Mri(\mathcal{P}(\widetilde{\mathbb{A}}_n))$ contains $\tau$-tilting infinite algebra of type $\widetilde{\mathbb{A}}_n$.
\end{example}

\begin{remark} 
\begin{enumerate}

\item The family of string algebras in Example \ref{Major NON-Example} could be in fact generalized to a larger family with similar properties, where all vertices are not necessarily nodes.
In particular, for every $n \in \mathbb{Z}_{>0}$, let $\widetilde{\Lambda}_n$ be obtained from $\Lambda_n$ as follows: replace every arrow $\gamma$ in $\Lambda_n$ by a copy of $\mathbb{A}_{m(\gamma)}$, for some $m(\gamma) \in \mathbb{Z}_{>0}$, such that the first and last arrows of $\mathbb{A}_{m(\gamma)}$ have the same direction as $\gamma$. Obvisously, $\widetilde{\Lambda}_n$ is again a string algebra and it is nody if and only if $n$ is even. 
The general configuration of $\widetilde{\Lambda}_n$ is illustrated in Figure \ref{fig:nody algebra}, where $n=2$. 
We remark that every vertex in $\widetilde{\Lambda}_n$ is of degree $2$ or $4$. As we will see in the following, such algebras play an important role in the study of nody algebras.

\item    In the above example we showed that $\Lambda_n$ is $\tau$-tilting finite if and only if $n$ is even. This result could be alternatively shown by the classification of $\tau$-tilting finiteness of radical-square zero algebras \cite{Ad}.
Moreover, note that although every $\Lambda_n$ is a string algebra with quadratic relations, it is never a quotient of a gentle algebra. In the next section we show that this in fact stems from a significant difference between string and gentle algebras from the viewpoint of $\tau$-tilting theory. Example \ref{Major NON-Example} and Proposition \ref{Wind wheel tau-finite} give two concrete families of rep-infinite string algebras which are $\tau$-tilting finite.
However these two infinite families are significantly different in terms of their bound quivers. This is because every wind wheel algebra is node-free, whereas in every algebra from Example \ref{Major NON-Example} all vertices are nodes. 
In the following we further enlarge this family of string algebras where $\tau$-tilting finiteness and rep-infiniteness disagree. 

\end{enumerate}
\end{remark}

\begin{figure}
    \centering
\begin{tikzpicture}

    
 \draw [<-] (2.7,3) to [bend right=20](0.5,3);
         \node at (0.45,2.95) {$\circ$};
          \node at (0.25,2.95) {$1$};
 \draw [->] (2.7,2.9) to [bend left=20](0.5,2.9);
        \node at (2.75,3) {$\circ$};
        \node at (2.95,3) {$2$};
     \node at (1.7,3.35) {$\alpha$};
      \node at (1.75,2.87) {$\alpha^*$};

\draw [->] (2.83,2.9) to [bend left=20](1.75,1.2);
     \node at (2.5,1.75) {$\beta$};
        \node at (2.25,2.1) {$\beta^*$};
    \node at (1.7,0.95) {$3$};
\draw [<-] (2.77,2.85) to [bend right=20](1.7,1.2);

\draw [<-] (0.37,2.9) to [bend right=20](1.6,1.2);
    \node at (1.65,1.2) {$\circ$};     
 \node at (0.7,1.75) {$\gamma$};
  \node at (1.2,2.1) {$\gamma^*$};

\draw [->] (0.37,2.9) to [bend left=20](1.62,1.25);    


 \draw [<-] (8,3.05) to [bend right=40](7.4,3.65);
 \draw [dashed] (7.5,3.65) to [bend right=15](6,3.65);
 \draw [<-] (6,3.65) to [bend right=40](5.5,3);
         \node at (5.45,2.95) {$\circ$};
          \node at (5.25,3.05) {$1$};
 
        \node at (8.05,3) {$\circ$};
        \node at (8.25,3.1) {$2$};
     \node at (5.6,3.65) {$\alpha_1$};
     \node at (8.1,3.65) {$\alpha_{m(\alpha)}$};

 \draw [->] (8,2.95) to [bend left=0](7.4,2.95);
  \draw [dashed] (7.5,2.95) to [bend left=0](6,2.95); 
  \draw [->] (6.1,2.95) to [bend left=0](5.55,2.95);
 
      \node at (7.7,3.15) {$\alpha^*_1$}; 
      \node at (6.2,3.15) {$\alpha^*_{m(\alpha^*)}$};


\draw [->] (8.1,2.95) to [bend left=40](8.65,2.1);
\draw [dashed] (8.65,2.1) to [bend left=0](8.65,1.7);
\draw [->] (8.65,1.7) to [bend left=40](8.1,1.05);

     \node at (8.85,2.55) {$\beta_1$};
     \node at (9.05,1.45) {$\beta_{m(\beta)}$};

            \node at (8.05,1.05) {$\circ$};
       
\draw [<-] (8.05,1.7) to [bend right=0](8.05,1.1);       
\draw [dashed] (8.05,1.7) to [bend left=0](8.05,2.3);       
\draw [->] (8.05,2.35) to [bend left=0](8.05,2.9);       
       
        \node at (7.8,1.6) {$\beta^*_1$};
         \node at (8,2.5) {$\beta^*_{m(\beta^*)}$};
    \node at (8.05,0.75) {$3$};


\draw [<-] (5.45,2.9) to [bend right=40](5.7,1.9);
\draw [dashed] (7,1) to [bend left=0](5.7,1.9);
\draw [<-] (7,1) to [bend right=40](8,1);
     \node at (7.25,0.7) {$\gamma_1$};
          \node at (5.12,2) {$\gamma_{m(\gamma)}$};

\draw [->] (5.47,2.9) -- (6,2.5);
\draw[dashed] (6,2.5) -- (7.55,1.4);
\draw[->](7.55,1.4) -- (8,1.05);    

      \node at (5.7,2.45) {$\gamma^*_1$};
  \node at (7.45,1.35) {$\gamma^*_{m(\gamma^*)}$};

\end{tikzpicture}
    \caption{In a nody algebra $\widetilde{\Lambda}_2$, nodes are all of $4$-vertices. The internal part, indicated by the dashed lines, could be of any length and orientation.  In particular, $\Lambda_2$ is a special case of $\widetilde{\Lambda}_2$.}
    \label{fig:nody algebra}
\end{figure}
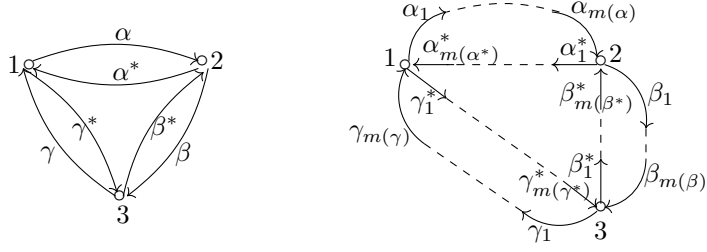

The following lemma provides a good intuition into the bound quiver of nody algebras.

\begin{lemma}\label{nodes of mSB}
If $\Lambda=kQ/I$ is a min-rep-infinite special biserial algebra, then $\node(\Lambda)= \{x \in Q_0 \, | \, \deg(x)=4 \}$ and the number of $3$-vertices in $\Lambda$ is even.
\end{lemma}

\begin{proof}
By Lemma \ref{Monomial ideal}, $\Lambda$ is a string algebra, so fix a minimal set of monomial relations that generate $I$. 

Since $\Lambda$ is rep-infinite, there exists a band $v$ in $(Q,I)$, and, because $\Lambda$ is minimal, $v$ must support all vertices and arrows. Hence, $(Q,I)$ obviously has no $1$-vertex or a node of degree two.

Suppose $\node(\Lambda) \neq \emptyset$ and, for the sake of contradiction and without loss of generality, assume that there exists a $3$-vertex $x$ in $\node(\Lambda)$ such that $\alpha_1$ and $\alpha_2$ are incoming to $x$ and $\beta$ is outgoing from it, such that $\beta \alpha_1 \in I$ and $\beta \alpha_2 \in I$.
Since $v$ is a band, it should have a presentation that starts (and ends) at $x$ and leaves it via $\beta$. This is, however, impossible and gives the desired contradiction (because such a presentation of $v$ must end with $\beta^{-1}$). Hence, every node in $\Lambda$ is a $4$-vertex.
For the converse, see the second assertion of Theorem \ref{node-free-algebra}, which shows that every $4$-vertex of $Q_0$ is a node.

For the second assertion, by the second part of Theorem \ref{node-free-algebra}, the algebra $\nn(\Lambda)$ is a min-rep-infinite special biserial algebra and by Corollary \ref{Nod-free Sp.biserial} it falls into one of the three types of algebras described in Section \ref{Section:Bound Quivers of Mild Special Biserial Algebras}. 
From the explicit configuration of the barbell and wind wheel algebras, where $\nn(\Lambda) \neq \emptyset$, it is easy to verify the claim.
Moreover, it is easy to check that resolving a node of $\Lambda$ does not interfere with the number of $3$-vertices appearing in $\nn(\Lambda)$, as it converts a $4$-vertex in $\Lambda$ into a pair of $2$-vertices in $\nn(\Lambda)$.
\end{proof}

Before showing the main result of this section, let us recall the dual to the notion of resolving a node. 
In a bound quiver $(Q,I)$, if $x$ is a sink (with the incoming arrows $\{\alpha_k, \cdots, \alpha_1\}$ and $y$ is a source (with outgoing arrows $\{\beta_m, \cdots ,\beta_1\}$), we say $(Q',I')$ is obtained from $(Q,I)$ via \emph{gluing $x$ and $y$} as follows: identify vertices $x$ and $y$ and call the new vertex $z$, then put all composition of arrows at $z$ equal to zero (i.e, $\beta_j \alpha_i=0$, for all $1 \leq i \leq k$ and all $1 \leq j \leq m$
).

\begin{proposition}{\label{gluing min-rep-inf sp.biserial}}
Let $\Lambda$ be a minimal representation-infinite special biserial algebra. If $\node(\Lambda) \neq \emptyset$, then $\Lambda$ is obtained via a sequence of gluing of barbell algebras, wind wheel algebras or a copy of $\widetilde{\mathbb{A}}_n$.
\end{proposition}

\begin{proof}
By Lemma \ref{nodes of mSB}, if $\Lambda=kQ/I$, then $\node(\Lambda)=\{x \in Q_0 \,|\, \deg(x)=4 \}$. 
If $\Lambda$ has no $3$-vertex, then $\nn(\Lambda) \simeq k \widetilde{\mathbb{A}}_d$ (for some $d \in \mathbb{Z}_{>0}$). 
In this case $\Lambda$ is the result of a sequence of gluing of a copy of $\widetilde{\mathbb{A}}_d$.

If $\Lambda$ has a $3$-vertex, then it cannot be a node and therefore, by Theorem \ref{node-free-algebra}, $\nn(\Lambda)$ is either a barbell or a wind wheel algebra.
\end{proof}

Now, we are ready to prove the following result.

\begin{proposition}\label{nody algebras are tau-finite}
Every nody algebra is $\tau$-tilting finite.
\end{proposition}

\begin{proof}
Since $\node({\Lambda}) \neq \emptyset$, the same argument as in the case of wind wheel algebras finishes the proof.
In particular, almost every string $w$ in $\Str(\Lambda)$ passes through a vertex which appears on the top and bottom of $w$. This gives rise to non-trivial graph maps in the string modules $M(w)$ and band modules $M(w,\lambda)$. Thus, there are only finitely many bricks in $\modu \Lambda$ and by Theorem \ref{tau-finiteness} we are done.
\end{proof}

The following theorem summarizes our results on min-rep-infinite special biserial algebras from the viewpoint of $\tau$-tilting finiteness.

\begin{theorem}\label{tau-finiteness of min-rep-inf special biseiral}
A minimal representation-infinite special biserial algebra $\Lambda$ is $\tau \text{-}$tilting finite if and only if $\nn(\Lambda)\neq \emptyset$ or $\Lambda$ is a wind wheel algebra. 
\end{theorem}
\begin{proof}
This follows from Proposition \ref{faithful tau-rigid}, Theorem \ref{tau-finite node-free mSB Thm} and Proposition \ref{nody algebras are tau-finite}.
\end{proof}

From the previous example one can also observe that not every min-rep-infinite algebra is minimal with respect to $\tau$-tilting infiniteness. Analogous to the classical notion of minimal representation-infinite, we call $\Lambda$ 
\emph{minimal $\tau$-tilting infinite} if $\Lambda$ is $\tau$-tilting infinite, but every proper quotient algebra of $\Lambda$ is $\tau$-tilting finite. 
In a sequel to this work, we extensively study the notion of minimal $\tau$-tilting infinite algebras for different families of algebras.
Recently, in \cite{W}, Wang also studies this notion for the two-point algebras. As a consequence of the preceding theorem, we have the following corollary.

\begin{corollary}\label{tau-infinite min-rep-sp.biserial}
If $\Lambda$ is a minimal representation-infinite special biserial algebra, it is minimal $\tau$-tilting infinite if and only if $\Lambda$ is a cycle or barbell algebra. 
\end{corollary}

Before we end this section, let us remark on the new insight that our results provide into min-rep-inf special biserial algebras and also present a useful consequence of our investigation from the viewpoint of $\tau$-tilting finiteness. 

By Corollary \ref{tau-infinite min-rep-sp.biserial}, we can observe that the intersection of the family of minimal rep-infinite special biserial algebras with the family of minimal $\tau$-tilting infinite special biserial algebras forms a relatively simple class of string algebras. In Remark \ref{examples of min-tau-infinite string alg}, we alluded to the fact that there could be an infinite family of minimal $\tau$-tilting infinite special biserial algebras which are not min-rep-infinite.
This gap between these two important families of special biserial algebras is treated in our future work.

With respect to the above results, roughly speaking, one could say that most min-rep-infinite special biserial algebras are $\tau$-tilting finite. 
In particular, for a fixed number of vertices $n$, the number of isomorphism classes of algebras in $\Mri(\mathfrak{F}_{\sB})$ with $n$ non-isomorphic simple modules which are $\tau$-tilting finite grows much faster that than the $\tau$-tilting infinite classes. 
This is because the latter classes of special biserial algebras, which consists of barbell and cycle algebras, have simple and restrictive configurations for their bound quivers in comparison with their $\tau$-tilting finite counterparts in $\Mri(\mathfrak{F}_{\sB})$, which consists of wind wheel and nody algebras.
In fact, Proposition \ref{gluing min-rep-inf sp.biserial} gives a simple algorithm to convert many algebras $\Lambda$ in $\Mri(\mathfrak{F}_{\sB})$ with $|\Lambda|>n$ to nody algebras with exactly $n$ vertices. 

For any $n$ in $\mathbb{Z}_{>0}$, finding the exact number of isomorphism classes of $\tau$-tilting finite and $\tau$-tilting infinite algebras $\Lambda$ in $\Mri(\mathfrak{F}_{\sB})$, with $|\Lambda|=n$, could be an interesting problem to investigate. We do not address this problem in this paper, however, the explicit description of the bound quivers of the algebras in $\Mri(\mathfrak{F}_{\sB})$ and their complete classification with respect to $\tau$-tilting finiteness should provide impetus to finding closed formulas for each $n$. 
One can also ask the same question for the number of isomorphism classes of minimal $\tau$-tilting infinite special biserial algebras. The latter problem, however, seems less tractable at this stage.
Any progress in this direction should shed a new light on the fundamental differences between the concepts of $\tau$-tilting finiteness and representation-finiteness, where, as explained in the introduction, the former one is a natural generalization of the latter concept. 

\section{$\tau$-Tilting Finite Gentle Algebras are Representation-finite}\label{Section:tau-tilting finite gentle algebras are representation-finite}

As shown in Section \ref{Section:Reduction to mild special biserial algebras}, a minimal representation-infinite algebra cannot have a projective-injective module. This, in particular, implies that every min-rep-infinite special biserial algebra is a string algebra (see Lemma \ref{Monomial ideal}).
From the explicit description of min-rep-infinite special biserial algebras in terms of their bound quivers (see Sections \ref{Section:Bound Quivers of Mild Special Biserial Algebras} and \ref{section: Nody algebras}), one can further observe that two of the four classes of min-rep-infinite special biserial algebras are in fact gentle algebras (being the cycle and barbell algebras).
Inspired by this observation, in this section we prove an important corollary of the methodology that we developed in the previous sections. In particular, we consider the problem of $\tau$-tilting finiteness of gentle algebras through a similar analysis.
However, to make a complete comparison between rep-finiteness and $\tau$-tilting finiteness of gentle algebras, we need to also consider some new bound quivers which do not appear on the list of those studied in the previous sections.

Let us first highlight an important difference between the family of gentle algebras and the family of special biserial algebras which we must consider in our treatment of the former class.
As before, by $\mathfrak{F}_{S}$ and $\mathfrak{F}_{G}$ we denote the family of string and gentle algebras. They are, by definition, proper subfamilies of $\mathfrak{F}_{\sB}$ which consists of special biserial algebras.
As shown in Sections \ref{Section:Reduction to mild special biserial algebras}, if $\mathfrak{F}$ is quotient-closed, to compare the concepts of rep-finiteness and $\tau$-tilting finiteness of the algebras in $\mathfrak{F}$, we can reduce the problem to that of the minimal rep-infinite algebras in $\mathfrak{F}$, denoted by $\Mri(\mathfrak{F})$.
This was the key point in our study of the $\tau$-tilting finiteness of special biserial algebras, because the family $\mathfrak{F}_{\sB}$ is quotient-closed. However, as remarked previously (and it is easy to check), neither the family of string algebras $\mathfrak{F}_{\St}$ nor that of gentle algebras $\mathfrak{F}_{\G}$ is quotient-closed. 
Hence, to compare the $\tau$-tilting finiteness and rep-finiteness of string and gentle algebras, one cannot directly employ the above-mentioned reductive method.

Regardless of the failure of the families $\mathfrak{F}_{\St}$ and $\mathfrak{F}_{\G}$ with respect to the quotient-closedness, one can still substantially benefit from our characterization of $\tau$-tilting finite and infinite algebras in $\Mri(\mathfrak{F}_{\sB})$ to treat the algebras from the aforementioned families.
To see this, as in Section \ref{Section:Reduction to mild special biserial algebras}, for a family of algebra $\mathfrak{F}$ we define
$$\Mri(\mathfrak{F}):= \{\Lambda \in \mathfrak{F} \, | \, \text{$\Lambda$ is min-rep-infinite} \}.$$

By Lemma \ref{Monomial ideal}, for the family of string algebras $\mathfrak{F}_{\St}$ we in fact have $\Mri(\mathfrak{F}_{\St})=\Mri(\mathfrak{F}_{\sB})$. 
Namely, two different notions of minimality coincide in the family of string algebras: if $\Lambda$ is a rep-infinite string algebra  and it is minimal in $\mathfrak{F}_{\St}$ with respect to this property (i.e, if $\Lambda'$ is any proper quotient of $\Lambda$, then either $\Lambda'$ is rep-finite or $\Lambda'$ does not belong to $\mathfrak{F}_{\St}$), then $\Lambda$ is minimal representation-infinite. 
Hence, $\Lambda$ falls into exactly one of the four different classes of min-rep-infinite special biserial algebras we studied in Section \ref{Section:Bound Quivers of Mild Special Biserial Algebras}, \ref{Section:tau-Tilting Finite Node-free Special Biserial Algebras} and \ref{section: Nody algebras}. We remark that for an arbitrary family $\mathfrak{F}$ of algebras, \emph{a priori}, neither of these two notions of minimality implies the other one.

In contrast to the family of string algebras, our results from the preceding sections do not fully settle the same problem for the family of gentle algebras.
This is because the above-mentioned coincidence of the notions of minimality does not occur in the family $\mathfrak{F}_{\G}$.
In particular, not all of the rep-infinte gentle algebras which are minimal in $\mathfrak{F}_{\G}$ appear among the list of algebras in $\Mri(\mathfrak{F}_{\sB})$ (Namely, there exist gentle algebras $A$ such that $A$ is rep-infinite and every proper quotient of $A$ is either rep-finite or not a gentle algebra, but $A$ is not min-rep-infinite.)

To achieve the primary goal of this section, which is to compare $\tau$-tilting finiteness and representation-finiteness of algebras in $\mathfrak{F}_{\G}$, we develop a simple reductive method for string algebras such that for any rep-infinite algebra in $\mathfrak{F}_{\G}$, through a sequence of reductions which preserves $\tau$-tilting finiteness, we end up in an explicit list of rep-infinite gentle algebras which we call fully reduced. One can abstractly define the fully reduced algebras as follows: a gentle algebra $A$ is \emph{fully reduced} if it is representation-infinite and for every proper quotient $B$ of $A$ we have either $B$ is rep-finite or $B$ is not a gentle algebra. 
Note that the fully reduced gentle algebras are in fact defined based upon the notion of minimality among the rep-infinite algebras in $\mathfrak{F}_{\G}$, rather than in the family of all algebras. 

This way, we only need to verify the $\tau$-tilting infiniteness of the fully reduced gentle algebras in order to prove the main result of this section. Based upon the concrete characterization of the fully reduced gentle algebras in terms of their bound quivers and the similar methods that we used in the previous sections, we show that every fully reduced gentle algebra admits infinitely many (non-isomorphic) bricks. This, by Theorem \ref{tau-finiteness}, implies the following theorem.
\begin{theorem}\label{tau-tilting finiteness of gentle algebras}
A gentle algebra is $\tau$-tilting finite if and only if it representation-finite.
\end{theorem}

As mentioned before, in a paper with the same title as the current section, and via a similar approach, Plamondon independently shows the above theorem in his recent work (see \cite{P}).

Before we restrict to the family of gentle algebras and specifically explore the $\tau$-tilting finiteness of this ubiquitous family of algebras, let us first look at some of the relevant consequences of our results from the preceding section.
The following theorem is in fact a rewording of Corollary \ref{tau-infinite min-rep-sp.biserial}, which highlights one of the fundamental differences between two subclasses of $\Mri(\mathfrak{F}_{\sB})$ in terms of gentle algebras.

\begin{theorem}
Let $\Lambda$ be a minimal representation-infinite special biserial algebra. Then, $\Lambda$ is $\tau$-tilting infinite if and only if it is a gentle algebra.
\end{theorem}

\begin{proof}
For $\Lambda$ in $\Mri(\mathfrak{F}_{\sB})$, by Theorem \ref{tau-finiteness of min-rep-inf special biseiral}, $\Lambda$ is $\tau$-tilting infinite if and only if it is a barbell or cycle algebra. Since $\Lambda$ is min-rep-infinite and every hereditary algebra of type $\widetilde{\mathbb{A}}_n$ is obviously gentle, one only need to verify that among the other three classes of min-rep-infinite special biserial algebras, only barbell algebras are gentle.

From the configuration of barbell algebras (see Section \ref{subsection:Barbell Algebra}), it is immediate that they are gentle.
In contrast, nody algebras are never gentle, because their bound quivers contain a $4$-vertex with more than two quadratic relations. Furthermore, each wind wheel algebra has at least one non-quadratic monomial relation, therefore it cannot be gentle. Hence, we are done.
\end{proof}

Because a minimal representation-infinite algebra which is $\tau$-tilting infinite is necessarily minimal $\tau$-tilting infinite, we can rephrase our classification of minimal $\tau$-tilting algebras in $\Mri( \mathfrak{F}_{\sB})$.

\begin{corollary}
A minimal representation-infinite special biserial algebra is minimal $\tau$-tilting infinite if and only if it is a gentle algebra.
\end{corollary}

The following is alternative articulation of the remark that followed Corollary \ref{tau-infinite min-rep-sp.biserial}.

\begin{remark}\label{size of barbell and affine vs nody and wind wheel}
In the previous section, we observed that due to the more flexible nature of the bound quivers of the nody and wind wheel algebras (in comparison to the barbell and cycle algebras), for almost all positive integers $d$, the number of isomorphism classes of wind wheel and nody algebras with $d$ simple modules is larger than those of barbell algebras and hereditary of type $\widetilde{\mathbb{A}}_d$. 
Starting from min-rep-infinite special biserial algebras $\Lambda=kQ/I$ with $d'=d+c$ simple modules, provided the bound quiver $(Q,I)$ has $c$ sinks and $c$ sources, one can apply Proposition \ref{gluing min-rep-inf sp.biserial} to create nody algebras with $d$ simple modules. 
Thus, it is clear that as $d$ grows, the difference between the number of nody and wind wheel algebras with the number of barbell algebras and those of type $\widetilde{\mathbb{A}}_n$ becomes more drastically different. This can be phrased in terms of gentle algebras: Although two of the four distinct classes of min-rep-infinite special biserial algebras consist of gentle algebras, roughly speaking, one can say that most min-rep-infinite string algebras are not minimal $\tau$-tilting infinite (equivalently, are not gentle).
 
This observation is in fact one of the major motivations in a sequel to this work, where we characterize minimal $\tau$-tilting infinite string algebras in terms of quivers and relations and further explore the similarities and differences between the new concept of minimal $\tau$-tilting infinite algebras that we defined above and their classical counterparts, minimal representation-infinite algebras.
Table \ref{Table} summarizes our results on the $\tau$-tilting finiteness of the minimal representation-infinite special biserial algebras and presents them in terms of the minimal $\tau$-tilting infinite algebas and gentle algebras. 
Note that the quiver of $\widetilde{\mathbb{A}}_n$ is acyclic. Moreover, the (left and right) cycles in the barbell and wind wheel algebras could be of any positive length and orientation.
The bars in the barbell and wind wheel algebras are always of positive length. However, in the former case, the bar is always unique and non-serial, whereas a wind wheel algebra may have several bars and all of them are serial. A nody algebra has at least one node. The dashed segment means the orientation and configuration of the bound quiver in that part is free, provided they satisfy the general conditions of a biserial algebra.

\begin{table}
\begin{center}
\caption{$\tau$-tilting finiteness of minimal rep-infinite special biserial algebras.}\label{Table}

\begin{tabular}{|c|c||c|c|}
 \hline
 \multicolumn{4}{|c|}{\qquad Minimal representation-infinite special biserial algebras \qquad\,} \\
 \hline
 Cycle & Barbell  & Wind wheel & Nody \\
 \hline
{\begin{tikzpicture}[scale=0.75]

\node at (0.08,0.75) {$\bullet$};
\draw [-] (0,0.75)-- (-0.75,0.05);
\node at (-0.5,0.6) {$\alpha_n$};
\node at (-0.75,-0.05) {$\bullet$};

\draw [dashed,-] (-0.75,0)-- (-0.75,-0.9);

\node at (-0.75,-1) {$\bullet$};
\draw [-] (-0.7,-1.05)-- (0,-1.8);
\node at (-0.5,-1.55) {$\alpha_{p+1}$};

\draw [-] (0.15,0.75)-- (0.8,0.05);
\node at (0.6,0.6) {$\alpha_1$};
    \node at (0.8,-0.05) {$\bullet$};  

\draw [dashed,-] (0.85,0)-- (0.85,-0.9);
\node at (0.8,-1) {$\bullet$};
\draw [-] (0.8,-1.05)--(0.1,-1.8);
\node at (0.7,-1.55) {$\alpha_p$};
\node at (0.05,-1.85) {$\bullet$};

\end{tikzpicture} }
 & 
 
\begin{tikzpicture}[scale=0.8]

\node at (8.7,-1.5) {$(r \geq 1)$};
\node at (8.75,-2) {non-serial bar};
\node at (7.75,-0.5) {$\circ$}; \node at (7.57,-0.5) {$x$};
\draw [->] (7.75,-0.55) to [bend left=15](7.35,-1.8);

\draw [dashed] (7.35,-1.8)to [bend left=50] (7.35,0.77);

\draw [->] (7.35,0.77) to [bend left=15](7.75,-0.43);

\draw [thick,dotted] (7.63,-0.05) to [bend right=70](7.63,-0.95);

\node at (7.45,-1.25) {$\beta$};
\node at (7.45,0.25) {$\alpha$};

\draw [-] (7.8,-0.5)-- (8.3,-0.5);
\node at (8.1,-0.75) {$\theta_1$};
\node at (8.35,-0.5) {$\bullet$};
\draw [dashed] (8.45,-0.5)-- (9,-0.5);
\node at (9.05,-0.5) {$\bullet$};
\draw [-] (9.1,-0.5)-- (9.6,-0.5);
\node at (9.3,-0.75) {$\theta_r$};
\node at (9.65,-0.5) {$\circ$};
\node at (9.85,-0.5) {$y$};

\draw [->]  (9.65,-0.4) to [bend left=15] (10.13,0.74);
\draw [dashed] (10.13,0.74)to [bend left=50]  (10.15,-1.87);
\draw [->]  (10.15,-1.85) to [bend left=15] (9.65,-0.55);

\draw [thick,dotted] (9.77,-0.05) to [bend left=70](9.77,-0.95);

\node at (10,0.2) {$\delta$};
\node at (10,-1.25) {$\gamma$};
\end{tikzpicture}
 
 & 

\begin{tikzpicture}[scale=0.8]

\node at (8.7,-1.5) {$(r \geq 1)$};
\node at (8.75,-2) {serial bar};
\node at (7.75,-0.5) {$\circ$}; \node at (7.57,-0.5) {$x$};
\draw [->] (7.75,-0.55) to [bend left=15](7.35,-1.8);

\draw [dashed] (7.35,-1.8)to [bend left=50] (7.35,0.77);

\draw [->] (7.35,0.77) to [bend left=15](7.75,-0.43);

\draw [thick,dotted] (7.63,-0.05) to [bend right=70](7.63,-0.95);

\node at (7.45,-1.25) {$\beta$};
\node at (7.45,0.25) {$\alpha$};

\draw [->] (7.8,-0.5)-- (8.3,-0.5);
\node at (8.1,-0.75) {$\theta_1$};
\node at (8.35,-0.5) {$\bullet$};
\draw [dashed,->] (8.45,-0.5)-- (9,-0.5);
\node at (9.05,-0.5) {$\bullet$};
\draw [->] (9.1,-0.5)-- (9.6,-0.5);
\node at (9.3,-0.75) {$\theta_r$};
\node at (9.65,-0.5) {$\circ$};
\node at (9.85,-0.5) {$y$};

\draw [->]  (9.65,-0.4) to [bend left=15] (10.13,0.74);
\draw [dashed] (10.13,0.74)to [bend left=50]  (10.15,-1.87);
\draw [->]  (10.15,-1.85) to [bend left=15] (9.65,-0.55);

\draw [thick,dotted] (9.77,-0.05) to [bend left=70](9.77,-0.95);

\node at (10,0.2) {$\delta$};
\node at (10,-1.25) {$\gamma$};

\draw [thick,dotted] (7.78,0.25) to [bend left=10] (7.95,-0.32)-- (9.45,-0.32) to [bend left=10] (9.68,0.25); 
\end{tikzpicture}

& 
\begin{tikzpicture}[scale=0.75]
\node at (3.55,0.25) {$\circ$};
\draw [dotted,thick] (3.55,0.25) circle (0.35cm);
\draw [->] (3.5,0.25)-- (2.76,1.22);
\node at (2.86,0.75) {$\alpha_1$};
\node at (2.7,1.25) {$\bullet$};

\draw [dashed] (2.65,1.25)to [bend right=50] (2.65,-0.9);

\node at (2.75,-1) {$\bullet$};
\draw [->] (2.8,-1)-- (3.5,0.2);
\node at (2.86,-0.4) {$\alpha_p$};
\draw [<-] (3.6,0.3)-- (4.35,1.25);
\node at (4.3,0.75) {$\beta_q$};
    \node at (4.35,1.25) {$\bullet$}; 
\draw [dashed] (4.35,1.25) to [bend left=50] (4.35,-0.9);

\node at (4.3,-1) {$\bullet$};
\draw [<-] (4.25,-0.98)--(3.55,0.2);
\node at (4.2,-0.5) {$\beta_1$};
\end{tikzpicture}
 
 \\
 \hline
\multicolumn{2}{|c||}
  {minimal $\tau$-tilting infinite} & \multicolumn{2}{|c|}{$\tau$-tilting finite}\\ \hline
\end{tabular}
\end{center}
\end{table}
\end{remark}

In what follows, we describe our reductive method for gentle algebras and explicitly describe the fully reduced gentle algebras in terms of their bound quivers.
In fact, we reduce any rep-infinite gentle algebra $A=kQ/I$ to another rep-infinite gentle algebra $A'=kQ'/I'$ such that $(Q',I')$ contains the minimum number of bands and every subquiver of it is supported by every band in $\Str(A')$. 
We remark that the new algebra $A'$ is not necessarily unique or min-rep-infinite.
Through this reduction, all such algebras $A'$ will be characterized by their bound quivers $(Q',I')$, as they belong to a finite list of explicit families analogous to those we analyzed in Sections \ref{Section:Bound Quivers of Mild Special Biserial Algebras}, \ref{Section:tau-Tilting Finite Node-free Special Biserial Algebras} and \ref{section: Nody algebras}. 
However, in the case of gentle algebras, since we aim to obtain another gentle algebra $A'$ with some minimality condition among rep-infinite algebras in $\mathfrak{F}_{\G}$, we should avoid non-quadratic monomial relations in $A'$. 
To describe this procedure precisely, we need to introduce some notations. 

To distinguish the gentle algebras from other types of algebras studied in this paper, in the remainder of the paper we often denote a gentle algebras by $A$, in contrast with $\Lambda$ which is used to denote an arbitrary algebra.
As before, the set of nodes in a fixed bound quiver $(Q,I)$ of an algebra $\Lambda=kQ/I$ is denoted by $\node(\Lambda)$. Moreover, a vertex $x$ of $\Lambda$ is \textit{secluded} if it is of degree one, and the set of all such vertices is denoted by $\secl(\Lambda)$. Let us write
$$e_{N(\Lambda)}:= \sum_{x\in \node(\Lambda)}e_x \qquad \text{and} \qquad e_{S(\Lambda)}:= \sum_{x\in \secl(\Lambda)}e_x$$
as the idempotents in $\Lambda$, respectively determined by the set of all nodes and the secluded vertices. 

From Remark \ref{vertex-quotient}, it is immediate that if $\Lambda$ is gentle, then so are the quotient algebras $\Lambda/e_{N(\Lambda)}$ and $\Lambda/e_{S(\Lambda)}$. However, we note that $\Lambda/e_{N(\Lambda)}$ may be direct product of more than one connected algebra. 

\begin{def-lem}\label{trimming}
Let $A$ be a gentle algebra and $e_{N(A)}$ and $e_{S(A)}$ be as above.
put $A_0=A$ and for every $i \in \mathbb{Z}_{\geq0}$, define 
$$A_{2i+1}:= A_{2i}/e_{N(A_{2i})} \qquad \text{and} \qquad A_{2i+2}:= A_{2i+1}/e_{S(A_{2i+1})}.$$
Then, there exists $m \in \mathbb{Z}_{\geq0}$, such that $A_{m}=A_{m+j}$, for every $j \in \mathbb{Z}_{\geq0}$, and $A$ and $A_m$ are of the same representation type. We call $A_m$ a \emph{trimmed} algebra.
\end{def-lem}

\begin{proof}
Since $A$ has finitely many vertices, there obviously exists $m$ in the first assertion. For the second part, assume $A=kQ/I$ is rep-infinite (otherwise the statement is trivial). A vertex $x$ in $Q$ is a node if and only $x$ is a $2$-vertex with a unique incoming arrow $\alpha$ and a unique outgoing arrow $\beta$ such that $\beta \alpha \in I$. Because we always quotient out by the vertices of degree $1$ or nodes, neither of which could be on a band in $(Q,I)$, these quotients do not destroy a band in $A$. Hence, $A$ and $A_m$ are of the same representation type.
\end{proof}

\begin{rem-convention}
In the previous lemma, if $A_m$ is disconnected, it is obviously a product of finitely many gentle algebras and we can treat each one separately.
Since we are interested in connected rep-infinite algebras, for simplicity, we henceforth assume the trimmed gentle algebra $A_m$ from the previous lemma is always connected and rep-infinite.
\end{rem-convention}

The following lemma collects some easy facts for future reference.
\begin{lemma}\label{trimmed properties}
Let $A=kQ/I$ be a trimmed gentle algebra. Then, we have
\begin{enumerate}
    \item $A$ is node-free (i.e, every $2$-vertex is free of relations).
    \item Every vertex of $Q$ is of degree strictly larger than one, and there are $2d$ vertices of degree $3$, for some $d \in \mathbb{Z}_{\geq 0}$.
\end{enumerate}
\end{lemma}

\begin{proof}
By construction in Lemma \ref{trimming}, a trimmed algebra has no secluded vertex or node. Thus, all vertices of $(Q,I)$ are of degree $3$ or $4$, respectively with exactly one and two quadratic relations, or they are $2$-vertices and free of relation. For the second part, note that the number of vertices of odd degree in a graph is always even, because each edge contributes $2$ to the sum of the degrees of all the vertices. Since $Q$ has no vertex of degree one, the assertion follows by parity.
\end{proof}

\begin{definition}\label{reduced gentle} Let $\Lambda= kQ/I$ be a string algebra and $w \in \Str(\Lambda)$.
\begin{enumerate}
    \item Put $w(Q_0):=\{x\in Q_0|\, w \text{ visits } x\}$. 
    Then, $e_w:=\sum_{x \in w(Q_0)} e_x$ is the \emph{idempotent associated to $w$}. Define $\Lambda(w(Q_0)):= \Lambda/ \langle 1- e_w \rangle$ as the \emph{weak reduction of $\Lambda$ via $w$}.  
    If $\Lambda$ is rep-infinite, we call it \emph{reduced} if $\Lambda(w(Q_0))=\Lambda$, for any band $w$ in $\Str(\Lambda)$.
    
    \item For $\gamma \in Q_1$, \emph{$w$ supports $\gamma$} provided that $\gamma$ (or $\gamma^{-1}$) occurs in $w$ and define $w(Q_1):= \{\gamma \in Q_1|\, w \text{ supports } \gamma\}$.
    Moreover, \emph{$w$ supports $\gamma^{\pm}$} if both $\gamma$ and $\gamma^{-1}$ occur in $w$, and put $w(Q_1^{\pm}):= \{\gamma \in Q_1|\, w \text{ supports } \gamma^{\pm} \}$.
    For each $\gamma \in Q_1 \setminus w(Q_1^{\pm})$, we say \emph{$w$ supports $\gamma$ in at most one direction}.
    A string is \emph{single-support} if it supports every arrow of $Q$ in at most one direction.
    
    \item If $w^c(Q_1):= Q_1 \setminus w(Q_1)$ is the set of arrows not supported by $w$, we call the quotient algebra $\Lambda(w(Q_1)):= \Lambda/\langle w^c(Q_1) \rangle$ the \emph{reduction via $w$}. 
    A reduced algebra $\Lambda$ is \emph{fully reduced} if $\Lambda=\Lambda(w(Q_1))$, for every band $w$ in $\Lambda$.
\end{enumerate}
\end{definition}

\textbf{Convention:}
Because in this section we are primarily interested in the study of $\tau$-tilting finiteness of connected rep-infinite gentle algebras, in order to avoid repetition of conditions, henceforth by a trimmed (respectively fully reduced) algebra we always mean a rep-infinite trimmed (respectively fully reduced) gentle algebra which is connected.

Some quick remarks will help to clarify the preceding definitions:
Every fully reduced algebra is reduced, and thus trimmed, but the converse is not true.
Every single-support band $w$ supports every arrow in at most one direction. However, in addition to $s(w)$, it can visit a vertex multiple times. Eventually, note that a (fully) reduced algebra can have multiple bands of different lengths and types.

The following example illustrates some of the foregoing definitions and remarks.

\begin{example}\label{Second Example}
Let $(Q,I)$ be the following bound quiver, where all the relations are quadratic and illustrated by the dotted lines.
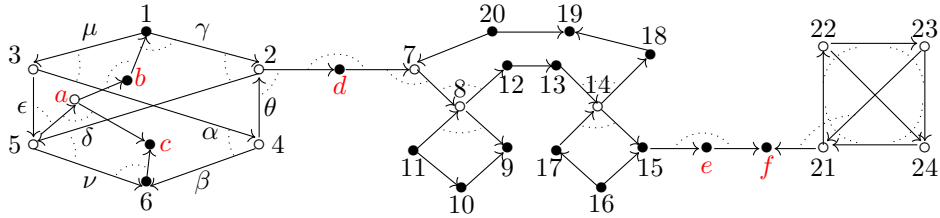
\begin{center}
\begin{tikzpicture}
    \node at (0.5,2.5) {$\bullet$};  \node at (0.5,2.75) {$1$};
     \node at (-1,2) {$\circ$};   \node at (-1.25,2.2) {$3$};
     \node at (2,2) {$\circ$};  \node at (2.15,2.2) {$2$};
     \node at (-1,1) {$\circ$};   \node at (-1.25,1) {$5$};
     \node at (2,1) {$\circ$};   \node at (2.25,1){$4$};
     \node at (0.5,0.5) {$\bullet$};    \node at (0.5,0.25) {$6$};
     
     \node at (-0.45,1.6) {$\circ$}; \node[red] at (-0.63,1.65) {$a$};
     \node at (0.25,1.85) {$\bullet$};\node[red] at (0.42,1.9) {$b$};
     \node at (0.55,1) {$\bullet$};\node[red] at (0.75,1) {$c$};
 
 \draw [->] (0.5,2.5) --(2,2.1); \node at (1.25,2.5) {$\gamma$};
 \draw [->] (0.5,2.5) --(-1,2.1);  \node at (-0.25,2.5) {$\mu$};
    \draw [<-] (2,1.9) --(2,1.1);  \node at (2.15,1.5) {$\theta$};
    \draw [->] (-1,1.9) --(-1,1.1);  \node at (-1.15,1.5) {$\epsilon$};
         \draw [->] (1.95,0.95) --(0.55,0.45); \node at (1.25,0.5) {$\beta$};
         \draw [->] (-1,0.95) --(0.45,0.45); \node at (-0.25,0.5) {$\nu$};
         
                \draw [->] (2,1.95) --(-0.95,1.05); \node at (-0.28,1.1) {$\delta$};
                \draw [->] (-0.95,2) --(1.95,1.1);  \node at (1.35,1.15) {$\alpha$};
                
          \draw [->] (-0.9,1.14) --(-0.44,1.53);
          \draw [->] (-0.4,1.6) --(0.2,1.85);
          \draw [->] (0.3,1.95) --(0.5,2.45);
          \draw [->] (-0.4,1.55) --(0.48,1.05); 
          \draw [<-] (0.55,0.94) --(0.5,0.5); 
                
\draw [dotted] (1.55,2.2) to [bend right=20](1.65,1.8);
\draw [dotted] (-0.4,2.2) to [bend left=20](-0.45,1.8);
\draw [dotted] (1.6,1.2) to [bend right=20](1.7,0.8);
\draw [dotted] (-0.5,1.1) to [bend left=20](-0.45,0.8);
\draw [dotted] (0.1,0.6) to [bend left=40](0.5,0.7);

\draw [dotted] (0.42,2.3) to [bend right=40](0.9,2.42);
\draw [dotted] (0,1.75) to [bend left=70](0.3,2.1);
\draw [dotted] (-0.6,1.38) to [bend right=70](-0.2,1.5);
\draw [dotted] (-0.95,1.5) to [bend left=20](-0.75,1.25);
\draw [dotted] (2.25,1.95) to [bend left=20](2,1.65);


\draw [->] (2.05,2) --(3,2); \node at (3.07,2){$\bullet$};
\draw [->] (3.15,2) --(4,2); \node at (4.07,2){$\circ$}; \node at (4,2.2){$7$};
        \draw [dotted] (2.65,2) to [bend left=70](3.4,2);
        \draw [dotted] (3.7,2) to [bend right=40](4.25,1.85);

\node[red] at (3.07,1.8){$d$};
        

\draw [->] (4.1,2) --(4.6,1.52); \node at (4.67,1.5){$\circ$}; \node at (4.67,1.75){$8$};
\draw [->] (4.75,1.45) --(5.25,1); \node at (5.3,0.95){$\bullet$};  \node at (5.3,0.7){$9$};
\draw [->] (4.75,0.45) --(5.25,0.9); \node at (4.69,0.42){$\bullet$}; \node at (4.69,0.2){$10$};
\draw [<-] (4.65,0.45) --(4.1,0.9); \node at (4.05,0.92){$\bullet$}; \node at (4.05,0.7){$11$};
\draw [->] (4.1,0.95) --(4.65,1.45);
        \draw [dotted] (4.45,1.75) to [bend left=20](4.95,1.75);
        \draw [dotted] (4.45,1.2) to [bend right=20](4.95,1.22);

 \draw [->] (4.75,1.55) --(5.25,2); \node at (5.3,2.05){$\bullet$};   \node at (5.35,1.8){$12$};      
\draw [->] (5.35,2.05) --(5.9,2.05); \node at (5.95,2.05){$\bullet$};   \node at (5.9,1.8){$13$};      
\draw [->] (6,2.05) --(6.5,1.55); \node at (6.5,1.5){$\circ$};   \node at (6.5,1.8){$14$};   

\draw [->] (6.55,1.47) --(7.05,1); \node at (7.1,0.95){$\bullet$}; \node at (7.2,0.7){$15$};
\draw [->] (6.55,0.45) --(7.05,0.9); \node at (6.55,0.42){$\bullet$}; \node at (6.55,0.22){$16$}; 
\draw [->] (6.5,0.45) --(6,0.9); \node at (5.95,0.92){$\bullet$}; \node at (5.85,0.7){$17$};
\draw [->] (6,1) --(6.45,1.45);
         \draw [dotted] (6.3,1.8) to [bend left=20](6.8,1.8);
         \draw [dotted] (6.3,1.25) to [bend right=20](6.75,1.3);

\draw [->] (6.55,1.55) --(7.15,2.15); \node at (7.2,2.2){$\bullet$}; \node at (7.25,2.45){$18$};
\draw [->] (7.2,2.25) --(6.2,2.5); \node at (6.13,2.5){$\bullet$}; \node at (6.13,2.75){$19$};
\draw [->] (5.1,2.5) --(6.05,2.5); \node at (5.1,2.5){$\bullet$}; \node at (5.1,2.75){$20$};
\draw [->] (5.1,2.5) --(4.1,2.1);


\draw [->] (7.2,0.95) --(7.85,0.95); \node at (7.95,0.95){$\bullet$};
\draw [->] (8.05,0.95) --(8.65,0.95); 
             \draw [dotted] (7.55,1) to [bend left=50](8.25,1);
\node[red] at (7.95,0.7){$e$};

\node at (8.75,0.95){$\bullet$};
\draw [<-] (8.85,0.95) --(9.4,0.95); 
\node [red] at (8.75,0.7){$f$};

\node at (9.5,0.95){$\circ$}; \node at (9.5,0.7){$21$};

\draw [->] (9.5,1) --(9.5,2.25); \node at (9.5,2.30){$\circ$}; \node at (9.5,2.55){$22$};

\draw [dotted] (9.2,1) to [bend left=25](9.95,1.4);

\draw [dotted] (9.5,1.83) to [bend left=40](10,2.3);
\draw [->] (9.6,2.35) --(10.75,2.35); \node at (10.85,2.30){$\circ$}; \node at (10.85,2.55){$23$};
\draw [dotted] (10.83,1.83) to [bend right=40](10.33,2.3);
\draw [->] (10.85,2.25) --(10.85,1.05); \node at (10.85,0.95){$\circ$}; \node at (10.85,0.7){$24$};
\draw [dotted] (10.82,1.55) to [bend left=40](10.38,1.03);

\draw [->] (10.75,1) --(9.62,1);
\draw [dotted] (9.95,1.03) to [bend left=40](9.53,1.55);

\draw [->] (9.6,2.25) to (10.8,1.05);
\draw [->] (10.8,2.25) to (9.55,1.05);

\end{tikzpicture} \label{fully reduced example}
\end{center}

One can easily check that $A=kQ/I$ is a gentle algebra which is not trimmed (the vertices $b$, $d$ and $e$ are nodes). If we consider $A/ \langle e_b+e_d+e_e \rangle$, we end up with a new node $a$ and a secluded vertex $f$. After trimming those, vertices $c$ becomes secluded. Hence, if we put $\overline{ \mathrm{e}}:=e_a+e_b+e_c+e_c+e_e+e_f$, then $A/\langle \overline{\mathrm{e}} \rangle$ has no nodes nor secluded vertex. The bound quiver of $A/\langle \overline{\mathrm{e}} \rangle$ has three disconnected components, each of which gives rise to a trimmed algebra. We denote them by $A_L$, $A_M$ and $A_R$, respectively for the left (L), middle (M) and right (R) component of the bound quiver of $A/\langle \overline{\mathrm{e}} \rangle$.

The vertex set of $A_L$ consists of $\{1,2,\cdots,6 \}$. It is easy to verify that $A_L$ is a reduced gentle algebra and contains a single-support band $w= \mu^{-1} \epsilon^{-1} \nu^{-1} \beta \theta^{-1} \gamma$, which is a copy of $\widetilde{\mathbb{A}}_5$ that passes through each vertex in $\{1,2,\cdots,6 \}$, but does not support $\alpha$ and $\delta$. Note that $A_L$ also has two other bands, given by $w'=\theta^{-1} (\gamma \mu^{-1} \epsilon^{-1} \delta) \theta (\alpha \epsilon^{-1} \nu^{-1} \beta)$ and $w''= \theta^{-1} ( \delta^{-1} \epsilon \mu \gamma^{-1}) \theta (\alpha \epsilon^{-1} \nu^{-1} \beta)$, which are not single-support and the bracketing makes it clear that they are non-isomorphic. The full reductions of $A_L$ via $w'$ and $w''$ result in the same algebra, being $A_L$ itself. However, the reduction via $w$ is a proper quotient of $A_L$, which is the path algebra $k\widetilde{\mathbb{A}}_5$, and this reduction destroys both $w'$ and $w''$. Hence, by definition, the reduction of $A_L$ is isomorphic to $k \widetilde{\mathbb{A}}_5$.

Unlike $A_L$, the algebra $A_M$ is not reduced. This is because there exist bands that only pass through vertices $\{8, 9, \cdots, 17\}$. 
Reduction of $A_M$ via such bands result in the following bound quiver, which is a barbell algebra. However, since the bar is serial, by Lemma \ref{min-rep-infiniteinite barbell algebras}, it is not min-rep-infinite.

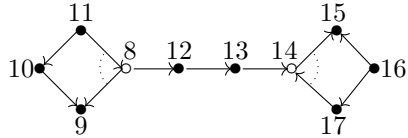
\begin{center}
\begin{tikzpicture}
 \node at (1.95,3) {$\bullet$}; \node at (1.95,3.25) {$11$}; 
 \draw [->] (2,3) --(2.5,2.5); \node at (2.55,2.5) {$\circ$}; \node at (2.6,2.75) {$8$};
 \draw [->] (2.5,2.5) --(2,2); \node at (1.95,1.95) {$\bullet$}; \node at (1.95,1.75) {$9$};
     \draw [<-] (1.9,2) --(1.4,2.5); \node at (1.4,2.5) {$\bullet$}; \node at (1.15,2.5) {$10$};
    \draw [<-] (1.45,2.55) --(1.95,3.05);
    
            \draw [dotted] (2.3,2.3) to [bend left=50](2.3,2.75);

\draw [->] (2.65,2.5) --(3.2,2.5); \node at (3.25,2.5) {$\bullet$}; \node at (3.25,2.75) {$12$};
\draw [->] (3.35,2.5) --(3.95,2.5); \node at (4,2.5) {$\bullet$}; \node at (4,2.75) {$13$};
\draw [->] (4.1,2.5) --(4.7,2.5); \node at (4.75,2.5) {$\circ$}; \node at (4.65,2.75) {$14$};

\draw [->] (4.8,2.5) --(5.3,3); \node at (5.35,3) {$\bullet$}; \node at (5.3,3.25) {$15$};
\draw [<-] (5.4,2.97) --(5.8,2.55); \node at (5.85,2.5) {$\bullet$}; \node at (6.1,2.5) {$16$};
\draw [<-] (5.4,2.0) --(5.8,2.5); \node at (5.35,1.95) {$\bullet$}; \node at (5.3,1.75) {$17$};
\draw [->] (5.35,1.95) --(4.8,2.45);

             \draw [dotted] (5,2.3) to [bend right=50](5,2.75);

\end{tikzpicture}
\end{center}

Finally, the trimmed gentle algebra $A_R$, associated to the right component of $A/\langle \overline{\mathrm{e}} \rangle$, is supported on the vertex set $\{21, 22, 23, 24 \}$ and contains a single-support band $w: 22 \rightarrow 24 \rightarrow 21 \leftarrow 23 \leftarrow 22$. There also exists another single-support band $w': 22 \rightarrow 24 \leftarrow 23 \rightarrow 21 \rightarrow 22$. Obviously, $w$ and $w'$ are non-isomorphic copies of $\widetilde{\mathbb{A}}_3$, as they are different acyclic orientation of $\widetilde{\mathbb{A}}_3$. The reduction of $A_R$ via either of these two bands produces non-isomorphic (finite dimensional) algebras of type $\widetilde{\mathbb{A}}_3$.
\end{example}

The next proposition shows that the bound quivers of the fully reduced algebras that were considered in Example \ref{Second Example} are instances of the general type of bound quivers that we need to consider in order to classify $\tau$-tilting finite gentle algebras. This plays a decisive role in the rest of the section.

\begin{proposition}\label{quiver of fully-reduced gentle algebras}
$A=kQ/I$ is a fully reduced gentle algebra if and only if one of the following holds:

\begin{enumerate}

\item $A=k\widetilde{\mathbb{A}}_n$, for an acyclic orientation of $\widetilde{ A}_m$ and some $m \in \mathbin{Z}_{\geq0}$.

\item $(Q,I)$ is a generalized barbell quiver. Namely, it is of the following form:
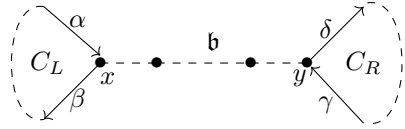
\begin{center}
\begin{tikzpicture}

 \draw [->] (1.25,0.75) --(2,0.1);
    \node at (1.7,0.55) {$\alpha$};
 \draw [<-] (1.25,-0.75) --(2,0);
    \node at (1.7,-0.5) {$\beta$};
  \draw [dashed] (1.25,0.75) to [bend right=100] (1.25,-0.75);
   \node at (1.3,0) {$C_L$};
    \node at (2,0) {$\bullet $};
    \node at (2.1,-0.2) {$x$};
 \draw [dashed] (2,0) --(2.75,0);
 \node at (2.75,0) {$\bullet$};
 \draw [dashed] (2.75,0) --(4,0);
 \node at (4,0) {$\bullet$};
 \draw [dashed] (4,0) --(4.75,0);
 \node at (3.5,0.3) {$\mathfrak{b}$};
 
 \node at (4.75,0) {$\bullet$};
 \draw [<-] (5.5,0.75) --(4.75,0);
 \node at (5,0.45) {$\delta$};
 \draw [->] (5.50,-0.8) --(4.8,-0.05);
    \node at (5,-0.55) {$\gamma$};
  \draw [dashed] (5.55,0.8) to [bend left=100] (5.55,-0.8);
   \node at (5.5,0) {$C_R$};
   \node at (4.65,-0.2) {$y$};
\end{tikzpicture}
\end{center}
where $ I= \langle \beta \alpha , \delta \gamma \rangle$, $C_L= \alpha \cdots \beta$ and $C_R=\gamma \cdots \delta$ are cyclic strings with no common arrow, and $\mathfrak{b}$ (respectively $C_L$ and $C_R$) can have any length (respectively any positive length) and orientation of their arrows, provided $C_RC_L$ is not a uniserial string in $(Q,I)$.
\end{enumerate}
\end{proposition}

Before we prove the proposition, let us elaborate on the sufficient conditions of the second case above: 
Note that the cyclic strings $C_L$ and $C_R$ could share at most one vertex (which will be $x=y$, provided $l(\mathfrak{b})=0$). Although $C_L$ and $C_R$ can be of any positive length, if they are simultaneously uniserial, then $l(\mathfrak{b})>0$ (otherwise, $C_RC_L$ and $C_LC_R$ are oriented cycles in $(Q,I)$ and $A$ becomes infinite dimensional). This is particularly the case if $\beta=\alpha$ and $\delta=\gamma$. Other than the specified arrows in $C_L$ and $C_R$, they can have any configuration of arrows and be of any positive length. As for the bar $\mathfrak{b}$, it is a copy of $\mathbb{A}_m$ (for some $m \in \mathbb{Z}_{\geq 0}$) of any orientation  and length (including zero, provided that $C_RC_L$ is not a uniserial cyclic string).
The dashed segments in the second bound quiver indicate the freedom of length and orientation.

\begin{proof}[Proof of Proposition \ref{quiver of fully-reduced gentle algebras}]
Let $w=\mu^{\epsilon_m}_d \cdots \mu^{\epsilon_1}_1$ be a band in $\Str(A)$. Without loss of generality, we assume $\epsilon_1=1$. Thus, if $s(\mu_1)=x$, then $e(\mu^{\epsilon_d}_d)=x$. 

We consider two mutually exclusive cases:
\begin{itemize}
    \item Case(1): $x$ is the only vertex that $w$ revisits.
    \item Case(2): $w$ revisits a vertex $y$ different that $x$.
\end{itemize}

First we argue Case(1): Every vertex of $(Q,I)$ is of degree at most $4$ and no band supports an arrow more than once in the same direction. By the assumption, $w$ does not revisit any vertex except for $x$. Hence, it visits $x$ either two or three times. 
If $x$ is visited exactly twice, $w$ is of type $\widetilde{\mathbb{A}}_m$ (for some $m \in \mathbb{Z}_{>0}$), and since $A$ is fully reduced, we have $Q=\widetilde{\mathbb{A}}_m$ and $I=0$. This is the first type of bound quivers asserted by the proposition. 
If $Q$ is not of type $\widetilde{\mathbb{A}}_m$, there exists the smallest $1\leq i\leq d-1$ such that $e(\mu^{\epsilon_i}_i )=x$. Thus, $\epsilon_i=1$, and $\mu_1 \mu_i \in I$ (otherwise $w'=\mu_i^{-1} \mu^{\epsilon_{i-1}}_{i-1} \cdots \mu^{\epsilon_2}_2 \mu_1$ becomes a copy of $\widetilde{\mathbb{A}}_{i-1}$, which contradicts the assumption that $A$ is fully reduced).
Since $w$ visits $x$ exactly three times, $x$ is a $4$-vertex and of $\mu^{\epsilon_{i+1}}_{i+1}$ and $\mu^{\epsilon_{d}}_{d}$, one is an incoming arrow to $x$ while the other one is outgoing arrow from $x$. Moreover, since $A$ is gentle and $\mu_1 \mu_i \in I$, either $\mu_{i+1}\mu_d \in I$ or $\mu_d \mu_{i+1} \in I$. Note that by the uniqueness of $x$ as the only revisited vertex by $w$, we have $u=\mu^{\epsilon_{i}}_{i} \cdots \mu_1$ and $v=\mu_d \cdots \mu^{\epsilon_{i+1}}_{i+1}$ are cylic strings (they both start and end at $x$) and except for $x$, there is no common vertex between $u$ and $v$.
This shows that in this case we get the bound quiver of the second type with $l(\mathfrak{b})=0$, where, up to relabelling, we have $\beta=\mu_1$, $\alpha=\mu_{i}$, $\delta=\mu_{i+1}$ and $\gamma=\mu_d$. Note that in this case $u$ and $v$ cannot be uniserial simultaneously (otherwise, the algebra becomes infinite dimensional).

Now we consider Case(2): Since there are more than one vertex revisited by $w$, without loss of generality, we assume $x$ is the vertex revisited by $w$ such that the cyclic substring $u$ of $w$ which starts and ends at $x$ is minimal among all cylic substrings of $w$.
Let $i$ be as in the previous case (i.e, the smallest $1\leq i\leq d-1$ such that $e(\mu^{\epsilon_i}_i )=x$. Thus, $\epsilon_i=1$, and $\mu_1 \mu_i \in I$). 
If $u:=\mu_i \mu^{\epsilon_{i-1}}_{i-1} \cdots \mu^{\epsilon_2}_2 \mu_1$, then $w=v u $, where $v:=\mu^{\epsilon_{d}}_{d} \cdots \mu^{\epsilon_{i+1}}_{i+1}$, and by our assumption, there exists $i+1 \leq j \leq d-1$ such that $e(\mu^{\epsilon_{j}}_{j})$ is visited by $w$ more than once and $e(\mu^{\epsilon_{j}}_{j}) \neq x$.
We assume $j$ is chosen such that $w$ revisits $e(\mu^{\epsilon_{j}}_{j})$, say at $e(\mu^{\epsilon_{k}}_{k})$, and for every $i<j'<j$, if $w$ revisits $e(\mu^{\epsilon_{j'}}_{j'})$, say at $e(\mu^{\epsilon_{k'}}_{k'})$, then $k<k'$.
Then, we put $y=e(\mu^{\epsilon_{j}}_{j})$. 

Thus, by minimality assumptions on $i$ and $j$, there exists $j+1 \leq k \leq d-1$ such that $e(\mu^{\epsilon_{k}}_{k})=y$. Let $\mathfrak{b}:= \mu^{\epsilon_{j}}_{j} \cdots \mu^{\epsilon_{i+1}}_{i+1} $ and $v_1:=\mu^{\epsilon_{k}}_{k} \cdots \mu^{\epsilon_{j+1}}_{j+1}$, then $v=v_2 v_1 \mathfrak{b}$.

Note that in the string $v$, vertex $y$ is of degree strictly larger than $2$ and by the defining condition $y$ is actually the only vertex that $v_1$ revisits. Namely, $v_1$ is the cyclic substring of $w$ which starts and ends at $y$ and revisits no other vertex. It is obvious that $v_1$ cannot visit any vertex on $\mathfrak{b}$. To show that $v_1$ also does not visit any vertex on $u$, assume otherwise and, for the sake of contradiction, suppose $v_1$ visits a vertex $z$ on $u$. Then, one can find a band which starts at $x$, continues via $u$ and $\mathfrak{b}$, passes through $y$ and goes along $v_1$ until it reaches $z$ and then lies on $u$ to return to $x$. This band obviously does not support the substring of $v_1$ which starts at $z$ and continues to the second visit of $v_1$ at $y$. Hence, the existence of such a band gives the desired contradiction. 
Therefore, we have shown that $u$ and $v_1$ are disjoint cyclic strings, meaning that they start and end at the same vertex, respectively at $x$ and $y$, and they do not share any vertex.

Finally, we show that $v_2= \mathfrak{b}^{-1}$. To see this, note that by the defining conditions for $i$ and $j$, it is obvious that the only vertex that $\mathfrak{b}$ shares with $v_1$ (respectively with $u$) is $y=e(\mathfrak{b)}$ (respectively $x=s(\mathfrak{b})$).
This holds because if we assume otherwise, it is easy to construct a band which does not support arrows in $v_1$. 
{Each of the two ends of $v_2$ is a $3$-vertex which must be involved in exactly one relation. In particular, $\mu_1\mu_i$ is a path of length two that passes through $x$ and belongs to $I$, whereas at the other end of $v_2$, depending on the sign $\epsilon_{j+1}$, we have either $\mu_{j+1}\mu_{k}$ or $\mu_{k}\mu_{j+1}$ is a path of length two that passes through $y$ and belongs to $I$ (otherwise, $u$ or $v_1$ generates a band of a shorter length or an oriented cycle).} 
Finally, observe that if $v_2 \neq \mathfrak{b}^{-1}$, there exist two distinct bands in $(Q,I)$, being $w$ and $w''= \mathfrak{b}^{-1} v_1 \mathfrak{b} u$. It is evident that $w''$ does not support some of the arrows in $v_2$, therefore we get the desired contradiction and this finishes the proof.
\end{proof}

From the above reductive process, the following corollary is immediate.

\begin{corollary}
Let $A=kQ/I$ be a representation-infinite gentle algebra. Then $(Q,I)$ contains a subquiver $(Q',I')$ of one of the following types:
\begin{enumerate}
    \item $(Q',I')$ has a single-support band and is hereditary (i.e, $I'=0$ and $kQ'=k\widetilde{\mathbb{A}}_n$, for some $n \in \mathbb{Z}_{>0}$).
    \item $(Q',I')$ has infinitely many bands of arbitrary large length that support the entire arrow set $Q'_1$.
\end{enumerate}
\end{corollary}

\begin{proof}
First we observe that if $(Q,I)$ is the non-hereditary bound quiver given in Proposition \ref{quiver of fully-reduced gentle algebras}, it admits infinitely many (non-isomorphic) bands and for any arbitrary $l \in \mathbb{Z}_{>0}$, there exists a band $w \in \Str(Q,I)$ with $l(w)>l$. This is because if we cosider the cyclic strings $u:=\mathfrak{b}^{-1}C_R\mathfrak{b}C_L$ and $v:=\mathfrak{b}^{-1}C_R^{-1}\mathfrak{b}C_L$, it is immediate that they are bands in $\Str(A')$. 
In fact, one can easily verify that every string of the form $w_{d}(v,u):=v^{j_d} u^{i_d} \cdots v^{j_1} u^{i_1}$ is also a band in $(Q,I)$, provided that $i_r >0$ and $j_s>0$, for some $1\leq r, s \leq d$, and $w_d(v,u)$ is not a proper power, inverse or a cyclic permutation of $w_{d'}(v,u)$, for some $d'\leq d$.

Assume $A$ is not fully reduced, otherwise there previous argument and Proposition \ref{quiver of fully-reduced gentle algebras} finishes the proof. By the reductive process introduced in Definition \ref{reduced gentle} and its following remarks, $A$ has at least one proper algebra quotient $A'=kQ'/I'$ which is fully reduced. Now, the assertion follows from the explicit description of the bound quiver $(Q',I')$, as one of the two possible cases given in Proposition \ref{quiver of fully-reduced gentle algebras}. 
\end{proof}

Now have the sufficient tools to show the main result of this section.

\begin{proof}[Proof of Theorem \ref{tau-tilting finiteness of gentle algebras}]
By Theorem \ref{tau-finiteness} and because brick-finiteness is preserved under surjective algebras maps, we only need to show that every rep-infinite gentle algebra admits infinitely many (non-isomorphic) bricks. 
Moreover, by the reductive method that we developed via Lemmas \ref{trimming} and \ref{trimmed properties} and their following paragraphs, without loss of generality, we can assume $A$ is a fully reduced gentle algebra.
Thus, we only need to show that the bound quivers given in Proposition \ref{quiver of fully-reduced gentle algebras} admit infinitely many bricks.
Hence, we have exactly one of the following:

\begin{enumerate}
    \item $A=k\widetilde{\mathbb{A}}_n$, for some $n \in \mathbb{Z}_{>0}$.
    \item $l(\mathfrak{b})>0$ and $\mathfrak{b}$ is non-serial.
    \item $l(\mathfrak{b})>0$ and $\mathfrak{b}$ is serial. 
    \item $l(\mathfrak{b})=0$.

\end{enumerate}

Regarding the first case, it is known that a hereditary is rep-infinite if and only if it is $\tau$-tilting infinite. If $A$ is of the second or third case, then it is a barbell algebra. Hence, by Proposition \ref{Barbells tau-infinite}, $A$ admits infinitely many bricks.

As for the case (4), first note that $C_L$ and $C_R$ cannot be uniserial strings simultaneously (otherwise, $\Lambda$ becomes infinite dimensional).
This, in particular, shows that $l(C_L)+l(C_R) \geq 4$.
For the sake of clarity, we first prove the assertion for a specific choice of the bound quiver in Proposition \ref{quiver of fully-reduced gentle algebras}, and then show that the same argument applies to any bound quiver of case $(4)$. 

Using the same notation as in Proposition \ref{quiver of fully-reduced gentle algebras}, since we are in case (4), we have $x=y$. We further consider the following specific choice of a fully reduced bound quiver with $l(\mathfrak{b})=0$. Note in particular that $\gamma=\delta$.

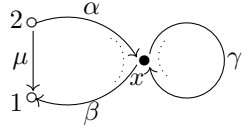
\begin{center}
    \begin{tikzpicture}

    \draw[<-] (3.52,2.42) arc (190:530:0.5cm);
\draw [dotted] (3.75,2.3) to [bend left=50](3.75,2.8);

\node at (4.65,2.5) {$\gamma$};
 \node at (3.45,2.5) {$\bullet$};

 \node at (3.35,2.25) {$x$};

\node at (2.75,3.2) {$\alpha$};
\draw [->] (2,3) to [bend left=40](3.35,2.55);
 \node at (1.95,3) {$\circ$};
  \node at (1.75,3.05) {$2$};

\draw [<-] (2,2) to [bend right=40](3.35,2.45);
\node at (2.75,1.8) {$\beta$};
 \node at (1.95,2) {$\circ$};
  \node at (1.75,1.95) {$1$};

\draw [dotted] (3,2.8) to [bend left=80](3,2.2);

\draw[->] (1.97,2.9) to (1.97,2.1);
 \node at (1.8,2.5) {$\mu$};
    \end{tikzpicture}
\end{center}

To generate infinitely many strings with no nontrivial graph maps, consider the band $w= \beta \gamma \alpha \mu^{-1}$, whose diagram is as follows:

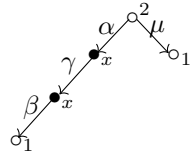
\begin{center}
\begin{tikzpicture}[scale=0.5]

\node at (0.2,0) {$\circ_1$};
\draw [->] (-0.95,1) -- (-0.1,0.1);
\node at (-0.4,0.75) {$\mu$};

\node at (-0.9,1.15) {$\circ^2$};
--
\draw [->] (-1.15,1) -- (-2,0.1);
\node at (-1.75,0.75) {$\alpha$};
\node at (-1.9,0) {$\bullet_x$};
--
\draw [->] (-2.1,0) -- (-3,-1);
\node at (-2.75,-0.25) {$\gamma$};
\node at (-2.95,-1.15) {$\bullet_x$};
--
\draw [<-] (-4.1,-2.15) -- (-3.2,-1.15);
\node at (-3.75,-1.35) {$\beta$};
\node at (-4,-2.3) {$\circ_1$};

\end{tikzpicture}
\end{center}

From the above diagram and Definition \ref{admissible_pair}, it is obvious that $w$ admits no nontrivial graph map, thus $M(w)$ is a brick.
Moreover, via a simple inductive argument one can verify that the string module $M(w^d)$ is also a brick, for every $d\in \mathbb{Z}_{>1}$. To show this, first observe that $w^2$, given by the concatenation of two copies of the diagram above, admits no nontrivial graph map. Thus, $M(w^2)$ is a brick. 
Then, for the sake of contradiction, assume $d$ is the smallest number such that $M(w^{d-1})$ is a brick but $\End(M(w^{d}))$ contains a nonzero endomorphism which is not invertible.
Since the graph maps form a basis for the space $\End(M(w^{d}))$, there must exist a proper substring $u$ of $w^{d}$ which is not a proper substring of $w^{d-1}$, such that $u$ appears both on the top (say as $u_1$) and at the bottom (say as $u_2$) of $w^{d}$ and $u_1$ and $u_2$ satisfy the configuration of the admissible pairs at their ends, as in Definitions \ref{factorization def} and \ref{admissible_pair}.
This is obviously a contradiction: because $d>2$ and by the properties of $u$, it contains a full copy of $w$ as a proper substring. Hence, by removing exactly one identical copy of $w$ from the middle of $u_1$ and $u_2$, we obtain the string $u'_1$ and $u'_2$ with $l(u_i)=l(u'_i)+4$ (for $i=1,2$) such that the configurations of $u_i$ and $u'_i$, as well as their ends, are the same as $u_i$ (for $i=1,2$). This gives rise to a substring $u'$ of $w^{d-1}$ which appear on the top and bottom of $w^{d-1}$ and admit a graph map. This obviously contradicts the assumption that $M(w^{d-1})$ is a brick. Hence, $M(w^{d})$ is a brick, for every $d \in \mathbb{Z}_{\geq 1}$.

From the above construction, it is immediate that the general case for $l(\mathfrak{b})=0$ could be shown in the same way. Namely, by the remark preceding the proof, one observes that if $l(\mathfrak{b})=0$, then $C_L$ and $C_R$ can have any internal configuration, provided they are not simultaneously uniserial. This implies that at least one of them is non-serial and $l(C_L)+l(C_R)\geq 4$. Hence, the same argument generates an infinite family of bricks.
In particular, let $C_L=\alpha \nu^{\epsilon_p}_{p} \cdots \nu^{\epsilon_2}_2 \beta$ and $C_R=\gamma \mu^{\epsilon'_q}_q \cdots \mu^{\epsilon'_2}_2 \beta$, with $\mu_j, \nu_i \in Q_1$ and $\epsilon_i, \epsilon'_j \in \{\pm 1\}$, for every $1 \leq i \leq p$ and $1 \leq j \leq q$.
Since $C_L$ and $C_R$ cannot be simultaneously cyclic strings, withouth loss of generality, we can assume there exists $1 \leq i \leq p$ such that $\nu^{\epsilon_i}_{i} \cdots \nu^{\epsilon_2}_2 \beta$ is serial, but $\nu^{\epsilon_{i+1}}_{i+1} \nu^{\epsilon_i}_{i} \cdots \nu^{\epsilon_2}_2 \beta$ is non-serial. Assume $i$ is chosen such that it is the smallest integer with this property.
Then, consider the string $w=(\nu^{\epsilon_i}_{i} \cdots\nu^{\epsilon_1}_{1} \beta) C_R (\alpha\nu^{\epsilon_p}_{p}\cdots\nu^{\epsilon_{i+1}}_{i+1})$. It is easy to check that no substring of $w$ appears both on the top and bottom of $w$ and therefore $M(w)$ is a brick. An induction, similar to the above case, shows that $M(w^d)$ is a brick for every $d \in \mathbb{Z}_{>0}$, which finishes the proof. 
\end{proof}

\begin{remark}
One should note that the construction of the infinitely many brick for the case of barbell algebras cannot be directly used for the last case of the previous proof (where the generalized barbell quiver has $l(\mathfrak{b})=0$). This is because if one simply forgets the bar in the string $w:= \mathfrak{b}^{-1} C_R \mathfrak{b} C_L$ given in the proof Proposition \ref{Barbells tau-infinite}, then the vertices $x$ and $y$, which are identified to one in the new case, appear both on the top and the bottom of the resulting string $w':= C_R C_L$, hence they give rise to a graph map.

With Theorem 8.1 in hand, we now know that the collection of minimal $\tau$-tilting infinite gentle algebras coincides with the collection of fully reduced gentle algebras classified in Proposition \ref{quiver of fully-reduced gentle algebras}. 
As remarked before, in a sequel to this paper, we extend this classification to all minimal $\tau$-tilting string algebras. In the next section, we show that although the fully reduced gentle algebras are not necessarily minimal representation-infinite, they share a lot of nice properties with the min-rep-infinite special biserial algebras.
\end{remark}

\section{Minimal Representation-infinite Algebras and more}\label{section:Minimal representation-infinite algebras and more}

In this section, by considering the notion of distributive algebras, we extend the main results of the preceding sections to those families of algebras whose representation theory does not admit as explicit a combinatorial description as that of special biserial algebras. 
In the second subsection, we revisit the fully reduced gentle algebras from the previous section and show some  interesting facts about them. 
We further use this family to elaborate on the direction of research we will pursue in our follow-up study and to pose some ideas that may provide new impetus to investigation of modern incarnations of some fundamental classical concepts in representation theory of algebras.

\subsection{Minimal representation-infinite biserial algebras}
As the main problem of this paper, we treated the minimal representation-infinite members of the family of special biserial algebras from a new perspective, where we fully determined which ones are $\tau$-tilting finite and which ones are not.
In this subsection we aim to extend the family to that of biserial algebras, which properly contains the family of special biserial algebras. It is known that the representation theory of a biserial algebra is less explicit in terms of the classification of all indecomposable modules. 
As we will see, as far the minimal representation-infinite algebras of these two families are concerned, they are the same, which allows us to restate our results in terms of the larger family.

Recall that $\Lambda$ is \emph{biserial} if for any left or right non-serial indecomposable projective module $P$, the radical of $P$ is a sum of two uniserial submodules $X$ and $Y$ such that $X \cap Y$ is either zero or a simple module. Since $\Lambda=kQ/I$ is assumed to be basic, this is equivalent to saying that for each $x \in Q_0$ and the associated indecomposable (left or right) projective module $P_x$, we must have $\rad(P_x)=M+N$ with $M$ and $N$ uniserial and $\dim_k(M\cap N)\leq 1$. As shown by Crawley-Boevey \cite{CB2}, every biserial algebra is tame. Moreover, it is known that every special biserial algebra is biserial (see \cite{SW}).

In fact, as mentioned in the proof of Lemma \ref{Monomial ideal}, if $\Lambda=kQ/I$ is special biserial, then the admissible ideal $I$ has a set of generators consisting of monomial and particular commutativity relations. Namely, there exists a set of relations of the form $u+\lambda v$ which generates $I$, where $\lambda \in k$, and $u$ and $v$ are two different paths in $(Q,I)$ such that $l(u), l(v) \geq 2$, with $s(u)=s(v)$ and $e(u)=e(v)$ such that $u$ and $v$ share no other vertex.

There are examples of biserial algebras which are not special biserial. For instance, suppose $Q'$ is the subquiver of the quiver $Q$ in Figure \ref{fig:example}, where $Q_0=Q'_0$ and $Q_1 \setminus Q'_1= \{\epsilon\}$. Then, consider $\Lambda'=kQ'/J$, where $J=\langle \beta \alpha-\gamma \delta \theta \alpha \rangle$. Then it is easy to verify that $\Lambda'e_1$, $\Lambda'e_3$, $\Lambda'e_4$, $\Lambda'e_5$, as well as, $e_1\Lambda'$, $e_2\Lambda'$, $e_3\Lambda'$, $e_4\Lambda'$ and $e_5\Lambda'$ are uniserial modules. Moreover, $\Lambda'e_2$ is a biserial indecomposable projective module where $\rad(\Lambda'e_2)$ is sum of the simple module $S_3$ and a uniserial module. Therefore, $\Lambda'$ is biserial, but it is not special biserial (because in $(Q',J)$ there exists a $3$-vertex with no relations on the two paths that pass through it).
This shows that the containment $\mathfrak{F}_{\sB} \subsetneq \mathfrak{F}_{\B}$ is sharp, where $\mathfrak{F}_{\sB}$ and $\mathfrak{F}_{\B}$ respectively denote the families of special biserial and biserial algebras.

In light of the reduction process that we have extensively used in the study of $\tau$-tilting infiniteness of different families of algebras, it is natural to ask whether the family of biserial algebras $\mathfrak{F}_{\B}$ is quotient-closed. The following proposition shows that this is the case.

\begin{proposition}\cite[Corollary 1.]{CB+}
Every quotient algebra of a biserial algebra is again biserial. 
\end{proposition}

In the previous sections, we considered the families of string and gentle algebras, respectively denoted by $\mathfrak{F}_{\St}$ and $\mathfrak{F}_{\G}$, which form two well-known subfamilies of $\mathfrak{F}_{\sB}$. In particular, we studied the $\tau$-tilting finiteness of those rep-infinite algebras in $\mathfrak{F}_{\St}$ and $\mathfrak{F}_{\G}$ which are minimal in their own family. We observed that the classification of min-rep-infinite special biserial algebras in Section \ref{Section:tau-Tilting Finite Node-free Special Biserial Algebras} and \ref{section: Nody algebras} plays an important role in the study of minimal rep-infinite algebras in either of these subfamilies. In particular, by Lemma \ref{Monomial ideal}, we had $\Mri(\mathfrak{F}_{\St})=\Mri(\mathfrak{F}_{\sB})$ and in Section \ref{Section:tau-tilting finite gentle algebras are representation-finite} we derived similar results for the algebras in $\Mri(\mathfrak{F}_{\G})$.
Because biserial algebras form an important over-family of $\mathfrak{F}_{\B}$, it is natural to ask the same question about those rep-infinite algebras in $\mathfrak{F}_{\B}$ which are minimal. In the rest of this subsection, we address this problem. As before, to determine whether or not an algebra $\Lambda$ in $\Mri(\mathfrak{F}_{\B})$ is $\tau$-tilting infinite, our strategy will be based on the notion of brick-finiteness (For the details, see Section \ref{Section:Reduction to mild special biserial algebras}). One should note that, due to the proper containment $\mathfrak{F}_{\sB} \subsetneq \mathfrak{F}_{\B}$, \emph{a priori}, there may exist more complicated algebras in $\Mri(\mathfrak{F}_{\B})$, in comparison with those from $\Mri(\mathfrak{F}_{\sB})$.

Motivated by the complete classification of indecomposable modules over special biserial algebras (as in \cite{WW}), there have been various attempts to fully classify the indecomposable modules over biserial algebras. However, as far as we are aware, there is no such a complete classification in the literature yet. 
Nevertheless, thanks to the vast study of minimal representation-infinite algebras, which has resulted in the description of most of them in terms of quivers and relations, one can fully determine the min-rep-infinite biserial algebras by their bound quivers. This, in particular, is the approach we employ in this section to ascertain whether or not a given min-rep-infinite biserial algebra is $\tau$-tilting finite. We hope this classification also sheds a new light on the study of biserial algebras from the viewpoint of the modern notion of $\tau$-tilting theory and the crucial role of bricks in the module category of such algebras.

Before we explore the $\tau$-tilting finiteness of the minimal representation-infinite biserial algebras, we need to recall some definitions and handy results. Although some of our statements here are independent of the ground field, to result in a neat classification, in the rest of this section, $k$ is assumed to be algebraically closed, unless specified otherwise. 

Following the literature (as in \cite{J} and \cite{Bo3}), we say $\Lambda$ is \emph{distributive} if the lattice of two-sided ideals of $\Lambda$ is distributive. In \cite{J}, Jans studies this notion to deal with some fundamental problems in representation theory which received a lot of attention in following decades. In particular, he showed that if $\Lambda$ is non-distributive, it must be strongly unbounded (i.e, there are infinitely many integers $d$ for each of which there are infinitely many isomorphism classes of $d$-dimensional indecomposable $\Lambda$-modules). Since then, different versions of this problem have been studied for various types of algebras and under several titles, the most well-known of which is the Second Brauer-Thrall Conjecture. As remarked previously, the minimal representation-infinite algebras played a crucial role in the full solution to the aforementioned conjecture, which was eventually settled by Bautista, in \cite{Ba}.

To analyze the minimal representation-infinite algebras in the family of biserial algebras, the following result of Skowro\'nski and Waschb\"usch \cite{SW} is decisive.

\begin{lemma}\cite[Lemma 2]{SW}\label{distributive biserial}
Any distributive biserial algebra is special biserial.
\end{lemma}

Due to the above lemma and Theorems \ref{tau-finite node-free mSB Thm} and \ref{tau-finiteness of min-rep-inf special biseiral}, our primary task will be to answer the following question:

\textbf{Question:}
Let $\Lambda$ be a minimal representation-infinite biserial algebra which is non-distributive. Is $\Lambda$ $\tau$-tilting finite or infinite?

To settle the above question, we employ the explicit characterizations of the non-distributive minimal representation-infinite algebras, recently given by Bongartz \cite{Bo3}. In particular, via an easy verification of the properties of biserial algebras, the following proposition  immediately follows from his classification.

\begin{proposition}\cite[Theorem 1]{Bo3}
{\label{Bongartz classification of min-rep-inf non-distributive}}
Let $\Lambda=kQ/I$ be a minimal representation-infinite algebra. If $\Lambda$ is biserial and non-distributive, the bound quiver $(Q,I)$ is one of the following bound quivers:

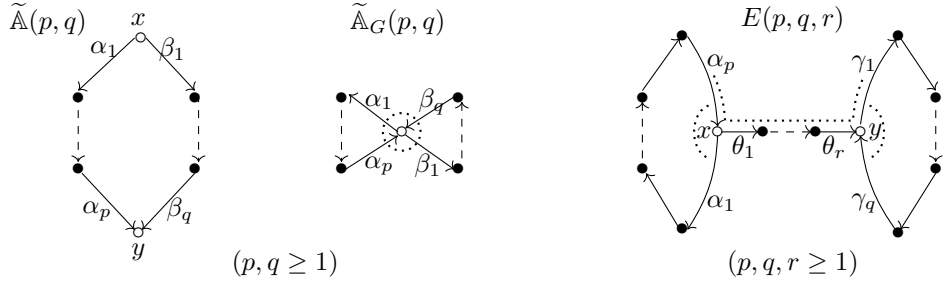
\begin{center}
\begin{tikzpicture}


\node at (0.08,0.75) {$\circ$};
\node at (0.05,1) {$x$};
\draw [->] (0,0.75)-- (-0.75,0.05);
\node at (-0.4,0.6) {$\alpha_1$};
\node at (-0.75,-0.05) {$\bullet$};

\draw [dashed,->] (-0.75,0)-- (-0.75,-0.9);

\node at (-0.75,-1) {$\bullet$};
\draw [->] (-0.7,-1.05)-- (0,-1.8);
\node at (-0.5,-1.55) {$\alpha_p$};

\draw [->] (0.15,0.75)-- (0.8,0.05);
\node at (0.5,0.6) {$\beta_1$};
    \node at (0.8,-0.05) {$\bullet$};  

\draw [dashed,->] (0.85,0)-- (0.85,-0.9);

\node at (0.8,-1) {$\bullet$};
\draw [->] (0.8,-1.05)--(0.1,-1.8);
\node at (0.6,-1.55) {$\beta_q$};

\node at (0.05,-1.85) {$\circ$};
\node at (0.05,-2.1) {$y$};

\node at (-1.15,1) {$\widetilde{\mathbb{A}}(p,q)$};



\node at (3.55,-0.5) {$\circ$};
\draw [->] (3.5,-0.5)-- (2.85,0);
\node at (3.25,-0.1) {$\alpha_1$};
\node at (2.75,-0.05) {$\bullet$};

\draw [dashed,->] (2.75,0)-- (2.75,-0.9);

\node at (2.75,-1) {$\bullet$};
\draw [->] (2.8,-1)-- (3.50,-0.55);
\node at (3.25,-0.95) {$\alpha_p$};

\draw [<-] (3.6,-0.45)-- (4.3,0.0);
\node at (3.95,-0.1) {$\beta_q$};
    \node at (4.3,-0.05) {$\bullet$};  

\draw [dashed,<-] (4.35,-0.1)-- (4.35,-0.9);

\node at (4.3,-1) {$\bullet$};
\draw [<-] (4.25,-1)--(3.6,-0.54);
\node at (3.9,-0.95) {$\beta_1$};

\draw [dotted,thick] (3.55,-0.5) circle (0.25cm);

\node at (3.5,1) {$\widetilde{\mathbb{A}}_G(p,q)$};
\node at (2,-2.25) {$(p,q \geq 1)$};


\node at (8.75,1) {$E(p,q,r)$};

\node at (8.75,-2.25) {$(p,q,r \geq 1)$};

\node at (7.75,-0.5) {$\circ$}; \node at (7.57,-0.5) {$x$};
\draw [->] (7.75,-0.55) to [bend left=15](7.35,-1.8);
\node at (7.28,-1.8) {$\bullet$};
\draw [<-] (6.8,-1.05)-- (7.25,-1.8);
\node at (6.75,-1) {$\bullet$};
\draw [dashed,<-] (6.75,-0.1)-- (6.75,-0.9);
\node at (6.75,-0.05) {$\bullet$};
\draw [<-] (7.25,0.75)-- (6.75,0.05);
\node at (7.28,0.77) {$\bullet$};
\draw [->] (7.35,0.77) to [bend left=15](7.75,-0.43);

\draw [thick,dotted] (7.63,-0.15) to [bend right=70](7.63,-0.85);

\node at (7.8,-1.45) {$\alpha_1$};
\node at (7.8,0.4) {$\alpha_p$};

\draw [->] (7.8,-0.5)-- (8.3,-0.5);
\node at (8.1,-0.7) {$\theta_1$};
\node at (8.35,-0.5) {$\bullet$};
\draw [dashed,->] (8.45,-0.5)-- (9,-0.5);
\node at (9.05,-0.5) {$\bullet$};
\draw [->] (9.1,-0.5)-- (9.6,-0.5);
\node at (9.3,-0.7) {$\theta_r$};
\node at (9.65,-0.5) {$\circ$};
\node at (9.85,-0.5) {$y$};

\node at (10.15,0.77) {$\bullet$};
\draw [->]  (9.65,-0.4) to [bend left=15] (10.13,0.74);
\draw [->] (10.18,0.75)-- (10.65,0.05);
\node at (10.65,-0.05) {$\bullet$};  
\draw [dashed,->] (10.65,0)-- (10.65,-0.9);
\node at (10.65,-1) {$\bullet$};
\draw [->] (10.65,-1.05)--(10.18,-1.8);
\node at (10.15,-1.85) {$\bullet$};
\draw [->]  (10.15,-1.85) to [bend left=15] (9.65,-0.55);

\draw [thick,dotted] (9.77,-0.15) to [bend left=70](9.77,-0.85);

\node at (9.7,0.4) {$\gamma_1$};
\node at (9.7,-1.45) {$\gamma_q$};

\draw [thick,dotted] (7.7,0.25) to [bend left=10] (7.85,-0.35)-- (9.55,-0.35) to [bend left=10] (9.7,0.25); 
\end{tikzpicture}
\end{center}
Here, $p,q,r \geq 1$ and $\widetilde{\mathbb{A}}(p,q)$ has no relations. $\widetilde{\mathbb{A}}_G(p,q)$ is obtained via gluing the vertices $x$ and $y$ in $\widetilde{\mathbb{A}}(p,q)$, thus all composition of arrows at the resulting vertex vanish.
For $E(p,q,r)$ the relations are given by the set $\{ \alpha_1 \alpha_p, \gamma_1 \gamma_q, \gamma_1 \theta_r \cdots \theta_1 \alpha_p \}$, as shown by the dotted lines. 
\end{proposition}

From Table 
\ref{Table}, it follows that $\widetilde{\mathbb{A}}(p,q)$, $\widetilde{\mathbb{A}}_G(p,q)$ and $E(p,q,r)$ in the preceding proposition are respectively hereditary, nody and wind wheel algebras. Hence, they are particular cases of those family of algebras that we studied in the course of the classification of $\tau$-tilting finite min-rep-infinite special biserial algebras in Sections \ref{Section:Bound Quivers of Mild Special Biserial Algebras} and \ref{section: Nody algebras}. In fact, in \cite{Bo3}, it is shown that the non-distributive min-rep-infinite algebras form five different classes up to gluing, three of which are obviously not biserial.
Thanks to this explicit characterization, we can now answer the above-mentioned question, which we articulate as the main theorem of this subsection. This could be viewed as the extension of our results from the family $\mathfrak{F}_{\sB}$ to its proper over-family $\mathfrak{F}_{\B}$.

\begin{theorem}\label{classification of min-rep-inf biserial algebras}
Let $\Lambda$ be a minimal representation-infinite biserial algebra. Then, $\Lambda$ is either a cycle, barbell, wind wheel or a nody algebra. In particular, $\Lambda$ is $\tau$-tilting finite if and only if it is a nody or a wind wheel algebra. 
\end{theorem}

\begin{proof}
By Proposition \ref{Bongartz classification of min-rep-inf non-distributive} and the classification of min-rep-infinite special biserial algebras in terms of their bound quivers (as in Sections \ref{Section:Bound Quivers of Mild Special Biserial Algebras} and \ref{section: Nody algebras}), every min-rep-infinite biserial algebra is in fact special biserial. Then, the result follows from Theorem \ref{tau-finiteness of min-rep-inf special biseiral}.
\end{proof}

Similarly, we can rephrase Corollary \ref{tau-infinite min-rep-sp.biserial} as follows.

\begin{corollary}\label{tau-infinite min-rep- biserial}
If $\Lambda$ is a mild biserial algebra, it is minimal $\tau$-tilting infinite if and only if $\Lambda$ is a cycle or a barbell algebra. 
\end{corollary}

In the following subsections, we further elaborate on how this observation may lead to a better understanding of the representation theory of biserial algebras.

\subsection{More on fully reduced gentle algebras}\label{subsection: More on fully reduced gentle algebras}

In this subsection, we show some interesting facts about fully reduced gentle algebras, which again highlight the interactions between two of the important notions of minimality that we alluded to in the preceding sections.
In particular, the following proposition shows that despite the fact that fully reduced gentle algebras are not necessarily min-rep-infinite, a lot of properties of their module categories and Auslander-Reiten quivers are common with those of the min-rep-infinite special biserial algebras. In light of our reductive method in Section \ref{Section:tau-tilting finite gentle algebras are representation-finite}, the next proposition could be viewed as a result on certain types of minimal $\tau$-tilting infinite algebras.

\begin{proposition}\label{Components of reduced gentle algebra}
Let $A$ be a fully reduced gentle algebra. Then,
\begin{enumerate}
    \item Almost every indecomposable module over $A$ is faithful. Moreover, all components of $\Gamma(\modu A)$ are faithful.
    
    \item Every generalized standard tube $\mathcal{T}$ in $\Gamma(\modu A)$ is hereditary.
    
    \item $\irigid(A) \setminus \mathtt{i}\tau \trig(A)$ is a finite set.
    
    \item $A$ has no projective-injective module.
    
    \item $A$ has no preprojective component, unless it is hereditary of type $\widetilde{\mathbb{A}}_n$.
\end{enumerate}
\end{proposition}

\begin{proof}
Proofs of $(1)$, $(2)$ and $(3)$ are similar to those given in Section \ref{Section:Reduction to mild special biserial algebras}, hence we omit them. To show $(4)$, one only needs to note that a string module $M(w)$ over a fully reduced gentle algebra is projective-injective if and only if $w$ is uniserial and $s(w)$ and $e(w)$ are $2$-vertices with quadratic relations. However, by Proposition \ref{quiver of fully-reduced gentle algebras}, the bound quiver of a fully reduced gentle algebra never satisfies this property. The last one follows from the classification of weakly minimal representation-infinite algebras with a preprojective components, given by Happel and Vossicek \cite{HV}.
\end{proof}

Thanks to the rich combinatorics of gentle algebras and their well-studied module categories, in the following theorem we capture some important properties of the fully reduced gentle algebras from the homological perspective.

Recall that $\Lambda$ is \emph{Gorenstein} if it has finite injective dimension as both a left and right $\Lambda$-module, in which case these dimensions are known to be equal and are called the \emph{Gorenstein dimension} of $\Lambda$.
Moreover, following \cite{AR2}, $\Lambda$ is called \emph{$d$-Gorenstein} if for a minimal injective resolution $\Lambda \rightarrow I^0(\Lambda) \rightarrow I^1(\Lambda)\rightarrow I^2(\Lambda) \rightarrow \cdots$ of $\Lambda$ as a left $\Lambda$-module, $\pd_{\Lambda} (I^i(\Lambda)) \leq i$, for every $0 \leq i \leq d-1$.
Self-injective algebras, as well as algebras of finite global dimension, are examples of Gorenstein algebras. Moreover, in \cite{GR}, Geiss and Reiten show that every gentle algebra is Gorenstein.

To prove the next result, we first need to recall some of the techniques developed by Geiss and Reiten \cite{GR} in the study of some important homological properties of gentle algebras.
In particular, we employ the elegant algorithm in the aforementioned work to compute the injective dimension of the fully reduced gentle algebra.

For a gentle algebra $A=kQ/I$, an arrow $\beta \in Q$ is called \emph{gentle} if there is no arrow $\alpha \in Q$ such that $\beta \alpha$ is a path of length two and $\beta \alpha \in I$. Moreover, a direct string $\alpha_m \cdots \alpha_1$ is said to be \emph{critical} if $\alpha_{i+1} \alpha_{i} \in I$, for every $1\leq i \leq {m-1}$. Finally, $n(A)$ denotes the maximal length of the critical strings which start with a gentle arrow. It is easy to show that $n(A) \leq|Q_1|$, and we put $n(A)=0$, if $A$ has no gentle arrow. 

\begin{theorem}(\cite[Theorem 3.4]{GR})\label{Theorem of Geiss-Reiten}
Let $A=kQ/I$ be a gentle algebra. Then, $\id_A(A)=n(A)=\pd_A \D(A^{op})$ if $n(A)>0$ and $\id_A(A)=\pd_A \D(A^{op}) \leq 1$ if $n(A)=0$. In particular, $A$ is Gorenstein.
\end{theorem}

Now, we can show the following proposition.

\begin{proposition}
For a fully reduced gentle algebra $A$, using the same notation as in Proposition \ref{quiver of fully-reduced gentle algebras}, we have the following:
\begin{enumerate}
\item $A$ has finite global dimension if and only if $A= k \widetilde{\mathbb{A}}_n$, or $A$ is a generalized barbell algebra where both $l(C_L), l(C_R)>1$. Provided either of these conditions holds, we have $\gl.dim(A)\leq 2$. 
\item $\id_A(A)=1$ if and only if $A= k \widetilde{\mathbb{A}}_n$, or $l(C_L)=1=l(C_R)$. Otherwise, $\id_A(A)=2$.
\item $A$ is of Gorenstein dimension $1$ if and only $A= k \widetilde{\mathbb{A}}_n$, or $l(C_L)=1=l(C_r)$. Otherwise, $A$ is of Gorenstein dimension $2$.

\end{enumerate}
\end{proposition}

\begin{proof}
Let $A=kQ/I$. 
For $(1)$, if $Q=\widetilde{\mathbb{A}}_n$, there is nothing to show. Hence, we assume $I\neq 0$ and by Proposition \ref{quiver of fully-reduced gentle algebras}, there exist $C_L$ and $C_R$ in $(Q,I)$. Now, the assertion is a direct consequence of the method applied in \cite{GHZ}.

In particular, without loss of generality, assume $l(C_L)=1$ with $x=s(C_L)=e(C_L)$ (the case where $l(C_R)=1$ is similar). Then, we obtain an infinite sequence of overlapping relations at $x$, which implies that $\pd_{A}(S_x)=\infty$, and therefore $\gl.dim(A)=\infty$. 
The converse follows from the fact that (under the assumptions $l(C_L) >1$ and $l(C_R)>1$), the sequence of overlapping relations in $(Q,I)$ is finite and of length one, which gives the desired result.

For the last part, if $I=0$, we obviously have $\gl.dim(A)=1$. As the only other possible case, as in Proposition \ref{quiver of fully-reduced gentle algebras}, if $I=\langle \beta\alpha, \delta\gamma \rangle$, it is easy to see that $\pd_A(S_i)=1$, for each vertex $i \in Q_0\setminus \{s(\alpha), s(\gamma)\}$. Moreover, from
$$ 0\rightarrow P_{e(\beta)} \rightarrow P_{e(\alpha)} \rightarrow P_{s(\alpha)} \rightarrow S_{s(\alpha)} \rightarrow 0,$$
as the projective resolution of $S_{s(\alpha)}$, we get $\pd_A(S_{s(\alpha)})=2$. By a similar argument for $S_{s(\gamma)}$, the result is immediate. This finishes the proof of $(1)$.

For the hereditary case $\Lambda=k\widetilde{\mathbb{A}}_n$, it is well-known that $\id_A(A)=1$. Hence, for the statement $(2)$, we only consider the second case of fully reduced gentle algebras given in Proposition \ref{quiver of fully-reduced gentle algebras}.
In order to determine the injective dimension of $A$ as a left $A$-module, we apply the algorithm given by Geiss and Reiten in \cite{GR} and use their main theorem stated above.

First note that every fully reduced gentle algebra contains a gentle arrow (in fact $A$ always has at least two gentle arrows). Therefore, $n(A) \geq 1$.

By Proposition \ref{quiver of fully-reduced gentle algebras}, if $l(C_L)=l(C_R)=1$ (i.e, $\alpha=\beta$ and $\gamma=\delta$), we must have $l(b)>0$. 
Then, from the definition, it is obvious that the only critical walks which start with a gentle arrow are in fact the arrows lying in $\mathfrak{b}$. Hence, $n(A)=1$ and therefore, by Theorem \ref{Theorem of Geiss-Reiten}, we have $\id(A)=1$.

If $l(C_L)>1$ (respectively $l(C_R)>1$) the critical walk starting with a gentle arrow which is of the maxmimum length is given by $\beta \alpha$ (respectively $\delta \alpha$), which shows that $n(A)=2$, and therefore $\id_A(A)=2$. 
According to the configuration of fully reduced gentle algebras (see Proposition \ref{quiver of fully-reduced gentle algebras}), this obviously hold for any fully reduced gentle algebra with $l(\mathfrak{b})=0$.

Proof of $(3)$ is an immediate consequence of Theorem \ref{Theorem of Geiss-Reiten} and statement $(1)$.

\end{proof}

Note that although the above proposition shows that every fully reduced gentle algebra $A$ is of Gorenstein dimension $1$ or $2$, one should note that by Proposition \ref{Components of reduced gentle algebra}, $A$ is never $d$-Gorenstein, for any $d \in \mathbb{Z}_{\geq 0}$, because it has no projective-injective module (For the definition and new results on the $d$-Gorenstein algebras, see \cite{IZ}.). Moreover, we remark that the last assertion of the above proposition extends a similar result of Ringel \cite[Theorem 14.2]{R1} on the barbell algebras which are shown to be Gorenstein of Gorenstein dimension $1$.

The previous proposition shows that for a fully reduced gentle algebra $A$ we always get {$\gl.dim(A) \in \{1, 2, \infty\}$.} In fact, the following corollary is immediate from the proof, showing that a fully reduced gentle algebra is of global dimension at most $2$ if and only if every simple module is rigid.
\begin{corollary}
Let $A=kQ/I$ be a fully reduced gentle algebra. Then, $\gl.dim(A)\leq 2$ if and only if $Q$ has no loop (i.e, an oriented cycle of length $1$). 
\end{corollary}

\subsection{Sufficient conditions for $\tau$-tilting infiniteness}\label{subsection:Explicit criteria for tau-tilting infiniteness}

In this subsection, we aim to give some explicit sufficient conditions for the $\tau$-tilting infiniteness of an arbitrary algebra over an algebraically closed field. To do so, 
we use our main reductive method and the new results on $\tau$-tilting infinite minimal biserial algebras. Furthermore, we employ the results of Happel and Vossieck \cite{HV} on the study of a family of algebras similar to those treated here.

As remarked in the introduction and observed in Section \ref{subsection:Module category of representation-infinite algebras}, studying an algebra from the viewpoint of $\tau$-tilting finiteness can also provide a new insight into the components of their Auslander-Reiten quivers.
For the min-rep-infinite special biserial algebras, provided that they are node-free, Ringel \cite{R2} has thoroughly studied their module categories and Auslander-Reiten quivers. 
Thanks to our new approach to the study of min-rep-infinite algebras, where bricks play a decisive role, one can easily conclude some important facts about certain components of the Auslander-Reiten quiver of any min-rep-infinite biserial algebra.

From Proposition \ref{tau-finiteness on the components and radical} we know that every rep-infinite algebra with a preprojective/preinjective component is $\tau$-tilting infinite. Viewing such algebras as a special case of $\tau$-tilting infinite type is somehow dual to the perspective we proposed for seeing $\tau$-tilting finite algebras as a generalization of rep-finite type. Thus, it is natural to ask whether those min-rep-infinite biserial algebras which are $\tau$-tilting infinite always admit a preprojective or preinjective component. The following theorem addresses this question.

\begin{theorem}
The Asulander-Reiten quiver of a generalized barbell algebra admits no preprojective component.
In particular, if $\Lambda$ is a minimal representation-infinite special biserial algebra, then $\Gamma(\modu \Lambda)$ has a preprojective component if and only if $\Lambda=k\widetilde{\mathbb{A}}_n$ (for some $n \in \mathbb{Z}_{>0}$).
\end{theorem}

Before we present the proof, we should draw the attention of the reader who wishes to consult the result of Happel and Vossieck \cite{HV} used in the following proof that the authors use the term ``minimal representation-infinite" for what we call ``weakly minimal representation-infinite". In particular, we say $\Lambda$ is \emph{weakly minimal representation-infinite} if $\Lambda=kQ/I$ is rep-infinite but every $\Lambda/\langle e_x \rangle$ is rep-finite, where $e_x$ is the idempotent associated to the vertex $x \in Q$.
Although every minimal representation-infinite algebra is obviously weakly minimal representation-infinite, the converse is not necessarily true (as a non-example, consider the a quiver with two vertices $x$ and $y$ and $3$ arrows from $x$ to $y$).

\begin{proof} 
Note that every min-rep-infinite algebra $\Lambda$ is obviously weakly minimal representation-infinite. 
The classification of weakly min-rep-infinite algebras which admit a preprojective component appears in \cite{HV}. 
From the configuration of a generalized barbell algebra (see Proposition \ref{quiver of fully-reduced gentle algebras}) and those listed in the aforementioned paper, it is easy to verify that a generalized barbell algebra also has no preprojective component. 
One should note that a generalized barbell algebra with a bar of length zero is not necessarily min-rep-infinite, but it is always weakly min-rep-infinite.

For the second part, because the assertion is well-known for $\Lambda=k\widetilde{\mathbb{A}}_n$, we only verify the statement for the barbell, wind wheel and nody algebras.
From Lemma \ref{Preproj/Postinj brick} and Propositions \ref{Wind wheel tau-finite} and \ref{nody algebras are tau-finite}, it follows that the Auslander-Reiten quiver of a nody or wind wheel algebra has no preprojective component and the other case follows from the firt part. 
\end{proof}

As a direct consequence of the previous theorem, the generalized barbell algebras form an concrete family of minimal $\tau$-tilting infinite algebras with no preprojective component. Obviously, a similar explicit description of all minimal $\tau$-tilting infinite algebras without preprojective components should be a big step towards a complete classification of all minimal $\tau$-tilting infinite algebra.

The above results, along with the classification of weakly minimal representation-infinite algebras studied in \cite{HV}, can also be employed as concrete criteria to verify $\tau$-tilting infiniteness of arbitrary algebras. For a given algebra $\Lambda$, recall that by $\Mri(\Lambda)$ we denote the set of isomorphism classes of all quotient algebras of $\Lambda$ which are min-rep-infinite. Similarly, let $\Mri_{\mathtt{W}}(\Lambda)$ denote the isomorphism classes of the weakly minimal-representation-infinite quotient algebras of $\Lambda$. Then, obviously, we have $\Mri(\Lambda)\subseteq \Mri_{\mathtt{W}}(\Lambda)$.

Now the following statement is immediate from Proposition \ref{tau-finiteness on the components and radical} and our classification of min-rep-infinite biserial algebras with respect to $\tau$-tilting finiteness.

\begin{corollary}
An algebra $\Lambda$ is $\tau$-tilting infinite if it has a representation-infinite quotient algebra $\Lambda'$ which is gentle or admits a preprojective/preinjective component. In particular, $\Lambda$ is $\tau$-tilting infinite if $\Mri_{\mathtt{W}}(\Lambda)$ contains a generalized barbell algebra or any of the algebras given in \cite{HV}.
\end{corollary}

\vskip 1 cm
\noindent
\textbf{Acknowledgements.} The author would like to thank his Ph.D. advisor, Hugh Thomas, for numerous stimulating discussions and useful comments in the course of preparing this paper. The author is also grateful to Charles Paquette for several helpful suggestions at various stages of this project.

\end{document}